\documentclass[12pt]{amsart}

\usepackage{graphicx}
\usepackage{amsmath}
\usepackage{amssymb}
\usepackage{setspace}
\usepackage{amsthm}
\usepackage{appendix}
\allowdisplaybreaks[4]
\usepackage{geometry}
\newtheorem{thm}{Theorem}[section]
\newtheorem{cor}[thm]{Corollary}

\newtheorem{lem}[thm]{Lemma}
\newtheorem{prop}[thm]{Proposition}
\theoremstyle{definition}

\theoremstyle{remark}
\newtheorem{rem}[thm]{Remark}
\theoremstyle{conclusion}

\theoremstyle{conjecture}

\numberwithin{equation}{section}

\newcommand{\be}{\begin{equation}}
\newcommand{\ee}{\end{equation}}

\newcommand{\dis}{\displaystyle}

\newcommand{\dx}{\textup{d}x}

\begin{document}
\title[Nonlinear Neumann boundary problems for $n$-Laplacian equations]{Nonlinear Neumann boundary problems for $n$-Laplacian Liouville equation on a half space}

\author{Wei Dai, Changfeng Gui, Yichen Hu, Shaolong Peng}

\address{School of Mathematical Sciences, Beihang University (BUAA), Beijing 100191, P. R. China, and Key Laboratory of Mathematics, Informatics and Behavioral Semantics, Ministry of Education, Beijing 100191, P. R. China}
\email{weidai@buaa.edu.cn}

\address{Department of Mathematics, University of Macau, Macau SAR, P. R. China}
\email{changfenggui@um.edu.mo}

\address{School of Mathematical Sciences, Dalian University of Technology, Dalian, 116024, P. R. China}
\email{huyc24@dlut.edu.cn}

\address{School of Mathematical Sciences, Beihang University (BUAA), Beijing, 100191,  P.R.China}
\email{{slpeng@amss.ac.cn}}

\thanks{Wei Dai is supported by the NNSF of China (No. 12222102 \& 12571113), the National Science and Technology Major Project (2022ZD0116401) and the Fundamental Research Funds for the Central Universities.  Changfeng Gui is supported by  NSFC Key Program (Grant No.12531010), University of Macau research grants CPG2024-00016-FST, CPG2025-00032-FST, CPG2026-00027-FST, SRG2023-00011-FST, MYRGGRG2023-00139-FST-UMDF, UMDF Professorial Fellowship of Mathematics, Macao SAR FDCT 0003/2023/RIA1 and Macao SAR FDCT 0024/2023/RIB1.  Yichen Hu is supported by NNSF of China (No. 12501135) and the Fundamental Research Funds for the Central Universities (DUT24RC(3)110). Shaolong Peng is supported by NNSF of China (No. 12401148) and Beijing Natural Science Foundation (No. 1262014). }

\begin{abstract}
In this paper, for general $n\geq2$, we classify solutions to $n$-Laplacian Liouville equation with positive nonlinear Neumann boundary condition on the half-space $\mathbb{R}^{n}_{+}$. Under the positive nonlinear Neumann boundary condition, our result extend the classification result for the second order Liouville equation in \cite{Li} from $n=2$ to general $n\geq2$, and also extend the classification result for critical $p$-Laplacian equation in \cite{Zhou} from $p<n$ to $p=n$.
\end{abstract}
\maketitle {\small {\bf Keywords:} $n$-Laplacian Liouville equations, Nonlinear Neumann boundary condition, Classification theorem, Quasilinear elliptic  equations, Half space. \\

}

\section{Introduction}

\subsection{Background and setting of the problem}
In this paper, we are mainly concerned with the following $n$-Laplacian Liouville equation in half-space with finite mass:
     \begin{equation}\label{eq:1-2}
      \left\{
          \begin{aligned}
          &-\Delta _{n}u=e^{u} \quad\,\,\, &{\rm{in}} \,\, \mathbb{R}^{n}_{+}, \\
          &|\nabla u|^{n-2}\frac{\partial u}{\partial \nu} =e^{\frac{n-1}{n}u} \quad\,\,\, &{\rm{on}} \,\, \partial\mathbb{R}^{n}_{+},\\
          &\int_{\partial\mathbb{R}^{n}_{+}}e^{\frac{n-1}{n}u}\mathrm{d}\sigma+\int_{\mathbb{R}^{n}_{+}}e^{u}\mathrm{d}x<+\infty,
          \end{aligned}
          \right.
    \end{equation}
where $n\geq2$ and $\nu=(0,\cdots,0,-1)$ is the unit outer normal vector on the boundary $\partial\mathbb{R}^{n}_{+}$. We say $u$ is a solution of the $n$-Laplacian Liouville equation \eqref{eq:1-2}, provided that $u$ is a function satisfying $u\in W^{1,n}_{\text{loc}}( \overline{\mathbb{R}^{n}_{+}}) $ and
\begin{equation*}
    \begin{split}
        \int_{\mathbb{R}^{n}_{+} }|\nabla u|^{n-2}\langle \nabla u,\nabla \psi\rangle\dx = \int_{\mathbb{R}^{n}_{+} }e^{u(x)}\psi(x)\dx +\int_{\partial\mathbb{R}^{n}_{+} }e^{\frac{n-1}{n}u(x)}\psi(x)\textup{d}\sigma
    \end{split}
\end{equation*}
for all $ \psi\in W^{1,n}_{0}(\Omega)\cap L^{\infty}(\Omega)$ with $\Omega\subset\mathbb{R}^{n}$ is bounded.

In particular, when $n=2$, the $n$-Laplacian degenerates into the regular second order Laplacian, by using the method of moving spheres and integral representation formula, Li and Zhu \cite{Li} proved that the $C^2$-smooth solutions to
\begin{equation}\label{eq:1-2'}
      \left\{
          \begin{aligned}
          &-\Delta u=e^{u} \quad\,\,\, &{\rm{in}} \,\, \mathbb{R}^{2}_{+}, \\
          &\frac{\partial u}{\partial \nu} =\pm e^{\frac{1}{2}u} \quad\,\,\, &{\rm{on}} \,\, \partial\mathbb{R}^{2}_{+},\\
          &\int_{\partial\mathbb{R}^{2}_{+}}e^{\frac{1}{2}u}\mathrm{d}\sigma+\int_{\mathbb{R}^{2}_{+}}e^{u}\mathrm{d}x<+\infty
          \end{aligned}
          \right.
    \end{equation}
must be of the form
$$
u(y,t)=\ln\left[\frac{8\lambda^2}{(\lambda^2+|y-y_0|^2+|t-t_{0}|^2)^{2}}\right]
$$
for some $\lambda>0$, $y_{0}\in \partial\mathbb{R}^{2}_{+}$ and $t=\mp\frac{\lambda}{\sqrt{2}}$. Zhang \cite{Zhang} removed the assumption $\int_{\partial\mathbb{R}^{2}_{+}}e^{\frac{1}{2}u}\mathrm{d}\sigma<+\infty$ in \cite{Li}.

Equation \eqref{eq:1-2} is related to the following Liouville equation in the whole space $\mathbb{R}^{n}$:
\begin{equation}\label{reeq:1-1}
    \begin{cases}
        -\Delta _{n}u=e^{u} \quad\,\,\, &{\rm{in}} \,\, \mathbb{R}^{n}, \\
        \displaystyle\int_{\mathbb{R}^{n}}e^{u}\mathrm{d}x<+\infty,
    \end{cases}
\end{equation}
where $n\geq 2$. When $n=2$, by using the method of moving planes, Chen and Li \cite{CL} proved that all the $C^2$-smooth solutions to equation \eqref{reeq:1-1} must be of the form
$$
u(x)=\ln \bigg[\frac{2\lambda}{1+\lambda^{2}|x-x_{0}|^{2}}  \bigg]
$$
for some $\lambda>0$ and $x_0\in \mathbb{R}^{2}$. The Liouville equation \eqref{reeq:1-1} was initially investigated by Liouville \cite{Liou}, which arises from a variety of situations, such as from prescribing Gaussian curvature in geometry and from combustion theory in physics. In \cite{CK,CW}, the authors reproved the classification results in \cite{CL} for the $2$-D Liouville equation \eqref{reeq:1-1} via different approaches. By exploiting the isoperimetric inequality and Pohozaev identity, Esposito \cite{E} classified all solutions of equation \eqref{reeq:1-1}. He show that all the  weak solution to the equation \eqref{reeq:1-1} must be of the form
$$
u(x)=\ln \left[\frac{c_n\lambda^n}{(1+\lambda^{\frac{n}{n-1}}|x-x_0|^{\frac{n}{n-1}}   )^{n} }  \right]
$$
for some $\lambda>0$ and $x_0\in \mathbb{R}^{n}$. Subsequently, Ciraolo and Li \cite{CL2} derived classification result for the anisotropic $n$-Laplacian Liouville equation in the whole space $\mathbb{R}^{n}$, Dai, Gui and Luo \cite{DGL} extended the classification results in \cite{CK,CL,CW,CL2,E} for $n$-D quasi-linear Liouville equation in the whole space $\mathbb{R}^{n}$ to anisotropic $n$-Laplacian Liouville equation in general convex cones $\mathcal{C}$.

One should note that, due to the nonlinearity of the $n$-Laplacian and the nonlinear Neumann boundary condition in \eqref{eq:1-2}, all the method of moving planes, the method of moving spheres in conjunction with the integral representation formula and the isoperimetric inequality approaches are not available. For sphere covering inequality and its applications on Liouville type equations, please refer to \cite{GM} and see also \cite{BGJM,GHM,GJM,GL}. For classification of solutions to semi-linear equations on Heisenberg group or CR manifolds via Jerison-Lee identities and invariant tensor techniques, see \cite{MO,MOW} and the references therein. For classification and other related results on Liouville type systems or (higher order) Liouville type equations, c.f. \cite{BF,C,CY,CK,DQ2,Lin,WX} and the references therein.

Damascelli, Merch\'{a}n, Montoro and Sciunzi \cite{DMS}, Sciunzi \cite{BS16}, V\'{e}tois \cite{VJ16}, Oliva, Sciunzi and Vaira \cite{OSV} established the classification results on critical $p$-Laplacian equations with $1<p<n$ via the method of moving planes, see also \cite{CDL,DLL}. For the classification results on critical $p$-Laplacian equations with $1<p<n$ on Riemannian manifold, refer to e.g. \cite{SW}. By using the vector field and integration method, Ciraolo, Figalli and Roncoroni \cite{CFR} classified positive solutions to the critical anisotropic $p$-Laplacian equation in general convex cones $\mathcal{C}$. For $\frac{n+1}{3}<p<n$, Ou \cite{ou} also used the vector field and integration method to prove the classification results for critical $p$-Laplacian equations in $\mathbb{R}^{n}$ without the finite energy condition. For $1<p<n$, Zhou \cite{Zhou} adapted the method of \cite{CFR,ou} and classified the positive solutions of the critical $p$-Laplacian equation with nonlinear Neumann boundary condition in the half space $\mathbb{R}^{n}_{+}$. Very recently, under the assumption $\int_{\mathbb{R}^{n}_{+}}e^{u}\mathrm{d}x+\int_{\partial\mathbb{R}^{n}_{+}}e^{\frac{n-1}{n}u}\mathrm{d}\sigma<+\infty$, Dou, Han, Yuan and Zhou \cite{DHYZ} proved the classification result for the $n$-harmonic equation $-\Delta_n u=0$ in the half space $\mathbb{R}^{n}_{+}$ with Neumann boundary condition $|\nabla u|^{n-2}\frac{\partial u}{\partial \nu} =e^{\frac{n-1}{n}u}$ on $\partial\mathbb{R}^{n}_{+}$.

\subsection{Main results}

In this paper, we prove the following classification theorem for \eqref{eq:1-2}.
\begin{thm}\label{Th:1-1}
Assume $n\geq2$. Any solution $u$ to equation \eqref{eq:1-2} must be of the form
\begin{equation*}
u(x) =\ln \frac{c_{n}\lambda^{n}}{\bigg(1+(\lambda|x-x^{0}|)^{\frac{n}{n-1}}\bigg)^{n}}
\end{equation*}
for some $x^0\in \mathbb{R}^{n}_{-}$ and $\lambda=-\frac{c_{n}^{\frac{1}{n}}}{nx^{0}_{n}}>0$, where $c_{n}=n(\frac{n^2}{n-1})^{n-1}$.
\end{thm}

\begin{rem}
Under the positive nonlinear Neumann boundary condition, our result extend the classification result for the second order Liouville equation in \cite{Li} from $n=2$ to general $n\geq2$, and also extend the classification result for critical $p$-Laplacian equation in \cite{Zhou} from $p<n$ to $p=n$.
\end{rem}

\subsection{Crucial ingredients and sketch of our proof}

Due to the nonlinearity of the $n$-Laplacian and the nonlinear Neumann boundary condition in \eqref{eq:1-2}, all the method of moving planes, the method of moving spheres in conjunction with the integral representation formula and the isoperimetric inequality approaches are not available. We will use the vector field and integration method to prove Theorem \ref{Th:1-1}.

\medskip

However, being different from the $p<n$ case, in order to apply the vector field and integration method for $n$-Laplacian, we need first to establish the sharp asymptotic estimates for solution $u$ and $\nabla u$ at $\infty$ (see Section 3). Due to the nonlinear Neumann boundary condition, it is difficult to construct suitable auxiliary upper and lower solutions for problem \eqref{eq:1-2}. Inspired by \cite{E,E1}, we will use a spherical reflection transformation and transform the sharp asymptotic estimates at $\infty$ to the asymptotic estimates near the origin $0$, and derive the sharp asymptotic estimate for $\nabla u$ via similar arguments as in \cite{CL2,DGL}. To this end, in order to deal with the nonlinear Neumann boundary condition, we need to prove many useful technique tools which has its own interests and importance, such as Brezis-Merle type exponential inequality with Neumann boundary condition, the global $L^{\infty}$-estimates under Neumann boundary condition, the gobal $L^{\infty}$-estimate for mixed boundary condition and $L^{\infty}$-estimate up to boundary for Neumann boundary condition, see Lemmas \ref{pro2-1}, \ref{rele:2}, \ref{le:2} and Proposition \ref{prop3-1}.

\medskip

Then, we need to prove the second order regularity $|\nabla u|^{n-2}\nabla u \in W^{1,2}_{\text{loc}}$ and get an asymptotic integral estimate on second order derivatives (see Section 4), which are crucial to our proof. However, since the solution $u$ only belongs to $W^{1,n}_{\text{loc}}(\mathbb{R}^{n}_{+})$, we are unable to produce a sequence of smooth, globally bounded approximate solutions to the solution $u$ of equation \eqref{eq:1-2} on the half space. Instead, we construct, on every fixed ball, smooth functions that satisfy a suitable approximate equation and converge to the solution $u$ of equation \eqref{eq:1-2}, see Lemma \ref{Apale1}. Furtheromre, in our case, we require not only the regularity results for Neumann boundary condition but also the regularity results for mixed boundary condition. To this end, we need first to establish the gobal $L^{\infty}$-estimate for mixed boundary condition and $L^{\infty}$-estimate up to boundary for Neumann boundary condition, see Lemmas \ref{rele:2} and \ref{le:2}.

\medskip

Finally, in Section 5, by using Bochner's skill and Pohozaev identity as in \cite{GuL,Zhou}, we constructed two important integral identities, which implies Theorem \ref{Th:1-1}. Comparing with their approaches in \cite{GuL,Zhou}, we have two new difficulties. First, as mentioned above, we cannot construct smooth approximate solutions on the whole half space. Second, the integrals under consideration are no longer finite and may blow up to infinity. To address the first difficulty, we construct, on every fixed ball, smooth functions that satisfy a suitable approximate equation and converge to the solution $u$ of equation \eqref{eq:1-2}, see Lemma \ref{Apale2}. For the second difficulty, we introduced a suitable perturbation $\epsilon$ when integrating over the ball $B_R(0)$. Then, we first let $R\to+\infty $ and then let $\epsilon\to 0$, see Proposition \ref{prop9-1}.

\medskip

We are only able to treat the positive nonlinear Neumann boundary condition in \eqref{eq:1-2}. The main reason is that, in the proof of the sharp asymptotic estimates for solution $u$ and $\nabla u$ at $\infty$ in Section 3, comparison principles are key tools and have been used in many places, which do not work on negative Neumann boundary condition, since the boundary value is negative and decay slowly at $\infty$ on $\partial\mathbb{R}^{n}_+$, please see the proof of Proposition \ref{prop:3.1} and Proposition \ref{th3-2} for details. Therefore, a new approach may be required to deal with the negative Neumann boundary condition. One should note that, our method can also derive the classification result on Neumann boundary value problem in $\mathbb{R}^{n}_{+}$ for $n$-harmonic equation $-\Delta_n u=0$ in \cite{DHYZ}.

\subsection{Structure of the paper}
The rest of our paper is organized as follows. In Section 2, we will give some preliminaries on the Brezis-Merle type exponential inequality for Finsler $n$-Laplacian, Serrin's local $L^{\infty}$-estimate for Neumann boundary condition, the comparison principle and Liouville type result and so on. Section 3 and Section 4 are devoted to proving sharp asymptotic estimates at infinity on $u$ and $\nabla u$, and the second-order regularity for weak solutions, respectively. In Section 5 and Section 6, we carry out our proof of Theorem \ref{Th:1-1} by using Bochner's skill and a Pohozaev identity like \cite{GuL}. In appendix A, we will prove the existence of the solution of n-laplacian equation with mixed boundary condition which is useful in the second-order regularity for weak solutions, Proposition \ref{prop9-1}. In appendix B, we give the proof of Lemma \ref{rele:2}.

\subsection{Some notation}

\smallskip
\noindent (1) $\mathbb{R}^{n}_{+} : = \{x\in \mathbb{R}^{n},x_{n}>0 \}$, $\mathbb{R}^{n}_{-} : = \{x\in \mathbb{R}^{n},x_{n}<0 \}$.

\noindent (2) $u\in W^{1,p}_{\text{loc}}(\overline{\mathbb{R}^{n}_{+}}\cap \Omega)$ means that for any compact set $\Omega_1$ such that $\Omega\subset \mathbb{R}^{n}_{+}\cap  \Omega$ and $dist(\partial \Omega_1,\partial \Omega\cap \mathbb{R}^{n}_{+})>0 $, we get that $ u\in W^{1,p}(\Omega_1 )$.

\noindent (3) $u\in W^{1,p}_{0}(\overline{\mathbb{R}^{n}_{+}}\cap \Omega)$ means that $u\in W^{1,p}(\overline{\mathbb{R}^{n}_{+}}\cap \Omega)$ and $u|_{\partial\Omega\cap \mathbb{R}^{n}_{+}} =0$.

\noindent (4) $\omega_n$ denote the volume of the $n$-dimensional unit ball.

\section{Preliminaries}
First, we need the following Sobolev inequality for function $u$ such that $u\in W^{1,p}\big(\mathbb{R}^{n}_{+}\cap B_{R}(x_{0})\big)$ and $ u|_{\partial B_{R}(x_{0})\cap \mathbb{R}^{n}_{+} } =0$, where $x_0\in\overline{\mathbb{R}^{n}_{+}} $.
\begin{lem}\label{relem:2.1}
  Let $1\leq p<n$ and $x_0\in\overline{\mathbb{R}^{n}_{+}}$. If $u\in W^{1,p}\big(\mathbb{R}^{n}_{+}\cap B_{R}(x_{0})\big)$ and $ u|_{\partial B_{R}(x_{0})\cap \mathbb{R}^{n}_{+} } =0,$ then
    \begin{equation}\label{reeq:2.1}
        \begin{split}
            \| u\|_{L^{p^{*}}\big( \mathbb{R}^{n}_{+}\cap B_{R}(x_{0}) \big)}\leq C_{n,p}\|\nabla u\|_{L^{p} \big(\mathbb{R}^{n}_{+}\cap B_{R}(x_{0})\big)},
        \end{split}
    \end{equation}
    where $C_{n,p}$ is a constant depending on $n$ and $p$.
\end{lem}
\begin{proof}
    If $ \partial \mathbb{R}^{n}_{+}\cap B_{R}(x_{0}) =\emptyset$, then \eqref{reeq:2.1} can be directly obtained from Sobolev inequality for $W^{1,p}_{0}\big(B_{R}(x_{0})\big)$. In the case that $ \partial \mathbb{R}^{n}_{+}\cap B_{R}(x_{0}) \neq\emptyset$, we have $ B_{R}(x_{0})\subset B_{2R}(x_{1})$, where $ x_{1}\in\partial \mathbb{R}^{n}_{+}$ and $ dist(x_{0},x_{1})=dist(x_{0},\partial\mathbb{R}^{n}_{+})$. We set
    \begin{equation*}
        \begin{cases}
        u_1(x) =u(x),&\hbox{ }x\in B_{R}(x_{0})\cap\mathbb{R}^{n}_{+},\\
        u_1(x) =0,&\hbox{ }x\in \big(B_{2R}(x_1)\setminus B_{R}(x_{0})\big)\cap\mathbb{R}^{n}_{+},\\
        u_1(x',x_{n}) =u_1(x',-x_{n}),&\hbox{ }x\in B_{2R}(x_1)\cap\mathbb{R}^{n}_{-}.
    \end{cases}
    \end{equation*}
    Then $u_{1}(x)\in W^{1,p}_{0}\big(B_{R}(x_{1})\big)$, $ \| u_{1}  \|_{L^{p^{*}}\big( \mathbb{R}^{n}_{+}\cap B_{R}(x_{0}) \big)} =2\| u\|_{L^{p^{*}}\big( \mathbb{R}^{n}_{+}\cap B_{R}(x_{0}) \big)} $ and $\|\nabla u_1\|_{L^{p} \big(\mathbb{R}^{n}_{+}\cap B_{R}(x_{0})\big)}=2\|\nabla u\|_{L^{p} \big(\mathbb{R}^{n}_{+}\cap B_{R}(x_{0})\big)} $. Thus applying Sobolev inequality for $W^{1,p}_{0}\big(B_{R}(x_{1})\big)$, we get that
    \begin{equation*}
        \begin{split}
            \| u\|_{L^{p^{*}}\big( \mathbb{R}^{n}_{+}\cap B_{R}(x_{0}) \big)}=\frac{1}{2} \| u_{1}  \|_{L^{p^{*}}\big( \mathbb{R}^{n}_{+}\cap B_{R}(x_{1}) \big)}\leq \frac{1}{2}C_{n,p}\|\nabla u_1\|_{L^{p} \big(\mathbb{R}^{n}_{+}\cap B_{R}(x_{1})\big)}=C_{n,p}\|\nabla u\|_{L^{p} \big(\mathbb{R}^{n}_{+}\cap B_{R}(x_{0})\big)}.
        \end{split}
    \end{equation*}
This finishes our proof of Lemma \ref{relem:2.1}.
\end{proof}

Next, we will prove the following three lemmas on Brezis-Merle type exponential inequality (c.f. \cite{BF}) for Finsler $n$-Laplacian, global $L^{\infty}$-estimate for mixed boundary condition and $L^{\infty}$-estimate up to boundary for Neumann boundary condition. For Brezis-Merle type exponential inequality, c.f. \cite{AP,E,RW} for $n$-Laplacian and c.f. \cite{CL2,WX2,XG} for Finsler $n$-Laplacian. For Serrin's local $L^{\infty}$-estimate, c.f. \cite{S}, see also \cite{CL2}. However, due to the nonlinear boundary conditions, we need to make some improvements.

Let $\Omega \subset \mathbb{R}^n$ be a bounded domain and $\mathbf{a}: \Omega \times \mathbb{R}^n \rightarrow \mathbb{R}^n$ be a Carathodory function so that
\begin{equation}\label{re2-1}
    |\mathbf{a}(x, p)| \leq c\left(a(x)+|p|^{n-1}\right), \qquad \text{ }\forall p \in \mathbb{R}^n \text {, a.e. } x \in \Omega,
\end{equation}
\begin{equation}\label{re2-2}
\langle\mathbf{a}(x, p)-\mathbf{a}(x, q), p-q\rangle \geq d|p-q|^n, \qquad  \text{ }\forall p, q \in \mathbb{R}^n \text {, a.e. } x \in \Omega
\end{equation}
for some $c, d>0$ and $a \in L^{\frac{n}{n-1}}(\Omega)$. Given $f \in L^1(\Omega)$, let $u \in W^{1, n}(\Omega)$ be a weak solution of
\begin{equation}\label{re2-3}
\begin{cases}
       -\operatorname{div} \mathbf{a}(x, \nabla u)=f &\quad \hbox{ in } \Omega ,
       \\ \mathbf{a}(x, \nabla u)\cdot\nu =g  &\hbox{ on }\Omega \cap \partial\mathbb{R}^{n}_{+}.
\end{cases}
\end{equation}
Thanks to \eqref{re2-1}, equation \eqref{re2-3} is interpreted in the following sense:
\begin{equation}\label{re2-4}
    \int_{\Omega}\langle\mathbf{a}(x, \nabla u), \nabla \phi\rangle=\int_{\Omega \cap \partial\mathbb{R}^{n}_{+}} g \phi +\int_{\Omega} f \phi, \quad \forall \phi \in W_0^{1, n}(\Omega\cap \overline{\mathbb{R}^{n}_{+}}) \cap L^{\infty}(\Omega).
\end{equation}

Since $u \in W^{1, n}(\Omega)$, let us consider the weak solution $h \in W^{1, n}(\Omega)$ of
 \begin{equation}\label{2-3}
\begin{cases}\operatorname{div} \mathbf{a}(x, \nabla h)=0 & \text { in } \Omega \cap \mathbb R^n_+, \\ h=u, & \text { on } \partial \Omega\cap \mathbb R^n_+,\\ \mathbf{a}(x, \nabla h)\cdot\nu=0 & \text { on }  \Omega\cap \partial\mathbb R^n_+.\end{cases}
\end{equation}

Similar to \cite{AP,BBGGPV,BG,E}, we can prove the following Brezis-Merle type exponential inequality with Neumann boundary condition.
\begin{lem}[Brezis-Merle type exponential inequality with Neumann boundary condition]\label{pro2-1}
     Let $f \in L^1(\Omega)$, $g \in L^1(\Omega\cap\partial\mathbb{R}^{n}_+)$ and assume \eqref{re2-1}-\eqref{re2-2}. Let $u$ be a weak solution of \eqref{re2-3} in the sense \eqref{re2-4}, and set
$$
\Lambda_q=\left(\frac{S_{q,+}^{\frac{n}{q}} d}{\|f\|_1+\|g\|_1}\right)^{\frac{1}{n-1}},
$$
where $S_{q,+}$ is the best constant of the Sobolev embedding $\mathcal{D}^{1, q}\left(\mathbb{R}^{n}_{+}\right) \hookrightarrow L^{\frac{n q}{n-q}}\left(\mathbb{R}^{n}_{+}\right), 1 \leq q<n$. Then, for every $0<\lambda<\Lambda_1$ and $ 0<\frac{q}{q-1}\beta<\Lambda_{1}$, there hold
 \begin{equation}\label{2-5}
\int_{\Omega} e^{\lambda|u-h|} \leq \frac{|\Omega|}{1-\lambda \Lambda_1^{-1}}, \quad \int_{\Omega}|\nabla(u-h)|^q \leq \frac{2 S_{q,+}}{\Lambda_q^{\frac{q(n-1)}{n}}}\left(1+\frac{2^{\frac{n}{q(n-1)}}}{(n-1)^{\frac{1}{n-1}} \Lambda_q}\right)^{\frac{q}{n}}|\Omega|^{\frac{n-q}{n}},
\end{equation}
and
\begin{equation}\label{re2-6}
    \begin{split}
        &\int_{\Omega\cap \partial\mathbb{R}^{n}_{+}} e^{\beta|u-h|}\textup{d}\sigma \\\leq &\tilde{S}_{1}\beta(\frac{|\Omega|}{1-\frac{q}{q-1}\beta \Lambda_1^{-1}})^{\frac{q-1}{q}}\bigg(\frac{2 S_{q,+}}{\Lambda_q^{\frac{q(n-1)}{n}}}\left(1+\frac{2^{\frac{n}{q(n-1)}}}{(n-1)^{\frac{1}{n-1}} \Lambda_q}\right)^{\frac{q}{n}}|\Omega|^{\frac{n-q}{n}}\bigg)^{\frac{1}{q}}+|\Omega\cap \partial\mathbb{R}^{n}_{+}|,
    \end{split}
\end{equation}
where $\tilde{S}_{1}$ is the best constant of the Sobolev trace embedding $\mathcal{D}^{1, 1}\left(\mathbb{R}^{n}_{+}\right) \hookrightarrow L^{1}\left(\partial\mathbb{R}^{n}_{+}\right)$.

\end{lem}
\begin{proof}
Fix $k \geq 0, a>0$. We define the truncation operator $T_k$ ($k>0$) by
 \begin{equation}\label{2-4}
T_k(u)=\left\{\begin{array}{cl}
u, & \text { if }|u| \leq k, \\
k \frac{u}{|u|}, & \text { if }|u|>k.
\end{array}\right.
\end{equation}
Since $T_{k+a}(u-h)-T_k(u-h) \in W_0^{1, n}(\Omega) \cap L^{\infty}(\Omega)$, by \eqref{2-3}-\eqref{2-4}, we get that
\begin{equation*}
    \begin{split}
        \int_{\Omega}\left\langle\mathbf{a}(x, \nabla u)-\mathbf{a}(x, \nabla h), \nabla\left[T_{k+a}(u-h)-T_k(u-h)\right]\right\rangle&=\int_{\Omega} f\left[T_{k+a}(u-h)-T_k(u-h)\right]
\\&+\int_{\Omega\cap \partial\mathbb{R}^{n}_{+}} g\left[T_{k+a}(u-h)-T_k(u-h)\right],
    \end{split}
\end{equation*}
which yields that
\begin{equation}\label{re2-7}
    \frac{1}{a} \int_{\{k<|u-h| \leq k+a\}}|\nabla(u-h)|^n \leq \frac{\|f\|_1 +\|g\|_1 }{d}
\end{equation}
in view of \eqref{re2-2}. By \eqref{re2-7} and Lemma \ref{le2.3}, we deduce the validity of \eqref{2-5}.

Since $u =h $ for $y \in  \partial\Omega\cap \mathbb{R}^{n}_{+}$, by Sobolev trace inequality, we get that
\begin{equation*}
    \begin{split}
      \int_{\Omega\cap \partial\mathbb{R}^{n}_{+}} |e^{\beta|u-h|}-1|\textup{d}\sigma \leq & \tilde{S}_{1}  \int_{\Omega} |\nabla e^{\beta|u-h|}|\dx
      \\\leq & \tilde{S}_{1} \int_{\Omega} |\beta e^{\beta|u-h|}\nabla |u-h||\dx
      \\\leq & \tilde{S}_{1}\beta \left(\int_{\Omega} | e^{\frac{q}{q-1}\beta|u-h|}\dx\right)^{\frac{q-1}{q}}\left(\int_{\Omega}|\nabla (u-h)|^{q}\dx \right)^{\frac{1}{q}}
      \\\leq & \tilde{S}_{1}\beta\left(\frac{|\Omega|}{1-\frac{q}{q-1}\beta \Lambda_1^{-1}}\right)^{\frac{q-1}{q}}\bigg(\frac{2 S_{q,+}}{\Lambda_q^{\frac{q(n-1)}{n}}}\left(1+\frac{2^{\frac{n}{q(n-1)}}}{(n-1)^{\frac{1}{n-1}} \Lambda_q}\right)^{\frac{q}{n}}|\Omega|^{\frac{n-q}{n}}\bigg)^{\frac{1}{q}},
    \end{split}
\end{equation*}
where $ 0<\frac{q}{q-1}\beta< \Lambda_{1}$. Thus
\begin{equation*}
    \begin{split}
        &\int_{\Omega\cap \partial\mathbb{R}^{n}_{+}} |e^{\beta|u-h|}|\textup{d}\sigma \\\leq
        &\tilde{S}_{1}\beta\left(\frac{|\Omega|}{1-\frac{q}{q-1}\beta \Lambda_1^{-1}}\right)^{\frac{q-1}{q}}\bigg(\frac{2 S_{q,+}}{\Lambda_q^{\frac{q(n-1)}{n}}}\left(1+\frac{2^{\frac{n}{q(n-1)}}}{(n-1)^{\frac{1}{n-1}} \Lambda_q}\right)^{\frac{q}{n}}|\Omega|^{\frac{n-q}{n}}\bigg)^{\frac{1}{q}}+|\Omega\cap \partial\mathbb{R}^{n}_{+}|.
    \end{split}
\end{equation*}
Thus we derive the validity of \eqref{re2-6}.
\end{proof}
\begin{lem}\label{le2.3}
     Let $w$ be a measurable function so that for all $k \geq 0, a>0$, $T_k(w) \in W^{1, n}(\Omega)$, $T_k(w)|_{\partial\Omega\cap \mathbb{R}^{n}_{+}} =0 $ and
$$
\frac{1}{a} \int_{\{k<|w| \leq k+a\}}|\nabla w|^n \leq C_0
$$
for some $C_0>0$, where $T_{k}$ is given by \eqref{2-4}. Then there hold
 \begin{equation*}
\int_{\Omega} e^{\lambda|w|} \leq \frac{|\Omega|}{1-\lambda \Lambda^{-1}}, \quad \int_{\Omega}|\nabla w|^q \leq 2 C_0^{\frac{q}{n}}\left(1+\left(\frac{2^{\frac{n}{q}} C_0}{(n-1) S_{q,+}^{\frac{n}{q}}}\right)^{\frac{1}{n-1}}\right)^{\frac{q}{n}}|\Omega|^{\frac{n-q}{n}}
\end{equation*}
for every $0<\lambda<\Lambda=\left(\frac{S_1^n}{C_0}\right)^{\frac{1}{n-1}}$ and $1 \leq q<n$.
\end{lem}
\begin{proof}
    Since $T_k(w) \in W^{1, n}(\Omega)$ and $T_k(w)|_{ \partial\Omega\cap \mathbb{R}^{n}_{+}} =0 $, by using Sobolev embedding $\mathcal{D}^{1, q}\left(\mathbb{R}^{n}_{+}\right)\hookrightarrow L^{\frac{n q}{n-q}}\left(\mathbb{R}^{n}_{+}\right)$, we can prove this Lemma in a similar way as the proof of Lemma 2.2 in \cite{E}.
\end{proof}

\begin{lem}[Gobal $L^{\infty}$-estimate for mixed boundary condition]\label{rele:2}
Assume that $ \mathbf{a}(x,p)$ is a Carathodory function satisfying $\mathbf{a}(x,0) =0$,
\begin{equation}\label{mixeq1}
    |\mathbf{a}(x, p)| \leq d\left(s+|p|^{n-1}\right), \qquad \text{ }\forall p \in \mathbb{R}^n \text {, a.e. } x \in \Omega,
\end{equation}
\begin{equation}\label{mixeq2}
\langle\mathbf{a}(x, p)-\mathbf{a}(x, q), p-q\rangle \geq e|p-q|^n,\qquad \text{ }\forall p, q \in \mathbb{R}^n \text {, a.e. } x \in \Omega,
\end{equation}
and
\begin{equation}\label{mixeq3}
|\mathbf{a}(x, p)-\mathbf{a}(x, p+t)|\geq  e\bigg(\big(s+|p|\big)^{n-2} |t|+  |t|^{n-1}\bigg),\qquad \text{ }\forall p, t \in \mathbb{R}^n \text {, a.e. } x \in \Omega
\end{equation}
for some $d,e,s>0$.
    Let $u\in W^{1,n}(B_{R}(0)\cap \mathbb{R}^{n}_{+})$ be a weak solution of the following equation
    \begin{equation}\label{rehareq1}
    \begin{cases}
         -div \big(\mathbf{a}(x,\nabla u)\big)=f\ \ &\text{\rm in}\ \Omega:=B_{R}(0)\cap \mathbb{R}^{n}_{+},\\
         u =h &\text{\rm on }\Gamma_1:=\partial B_{R}(0)\cap \mathbb{R}^{n}_{+},
         \\
        \mathbf{a}(x,u,\nabla u)\cdot\nu=g \ \ &\text{\rm on }\Gamma_2:=B_{R}(0)\cap\partial \mathbb{R}^{n}_{+}
    \end{cases}
    \end{equation}with $f\in L^{\frac{n}{n-\varepsilon}}(\Omega)$, $ h\in C^{1}(\overline{\Omega})$ and $g\in L^{\frac{n-1}{n-1-\varepsilon}}(\Gamma_2) $ for
    some $0 < \varepsilon < 1$. Then

    \noindent{(i)} for any $r>0$ and $x_{0}\in  \Gamma_1$ such that $B_{2r}(x_{0})\cap \Gamma_2 =\emptyset $, we have that
    \begin{equation*}
    \begin{split}
        \|u\|_{L^{\infty}(B_{r}(x_{0})\cap\Omega )}
       \leq &C_{d,e,n,h,\epsilon}\bigg(\|u\|_{L^{n}(B_{2r}(x_{0})\cap\Omega)}+\|h\|_{L^{n}(B_{2r}(x_{0})\cap\Omega)}
       \\&+r^{-\epsilon}\|f\|_{L^{\frac{n}{n-\varepsilon}}(B_{2r}(x_{0})\cap\Omega)}^{\frac{1}{n-1}}
       +(s+s^{n-1}+\|h\|_{C^{1}(\overline{\Omega})})^{\frac{1}{n-1}}\bigg),
    \end{split}
    \end{equation*}
    where $C_{d,e,n,h,\epsilon}$ is a constant depending on $e,d,n,\|h\|_{C^{1}(\overline{\Omega})}$ and $\epsilon$;

     \noindent{(ii)} for any $r>0$ and $x_{0}\in  \Gamma_2$ such that $B_{2r}(x_{0})\cap \Gamma_1 =\emptyset $, we have that
    \begin{equation*}
    \begin{split}
        \|u\|_{L^{\infty}(B_{r}(x_{0})\cap\Omega )}
       \leq &C_{d,e,n,\epsilon}\bigg(\|u\|_{L^{n}(B_{2r}(x_{0})\cap\Omega)}+\|u\|_{L^{n}(B_{2r}(x_{0})\cap\Gamma_2)}
       \\&+r^{-\epsilon}\|f\|_{L^{\frac{n}{n-\varepsilon}}(B_{2r}(x_{0})\cap\Omega)}^{\frac{1}{n-1}}+r^{-\epsilon}\|g\|_{L^{\frac{n-1}{n-1-\varepsilon}}(B_{2r}(x_{0})\cap\Omega)}^{\frac{1}{n-1}}
       \\&+(s+s^{n-1}+\|h\|_{C^{1}(\overline{\Omega})})^{\frac{1}{n-1}}\bigg),
    \end{split}
    \end{equation*}
    where $C_{d,e,n,\epsilon}$ is a constant depending on $e,d,n$ and $\epsilon$;

       \noindent{(iii)} for any $0<r<\frac{R}{8}$ and $x_{0}\in \partial \Gamma_2$, we have that
    \begin{equation*}
    \begin{split}
        \|u\|_{L^{\infty}(B_{r}(x_{0})\cap\Omega )}
       \leq &C_{d,e,n,h,\epsilon}\bigg(\|u\|_{L^{n}(B_{2r}(x_{0})\cap\Omega)}+\|u\|_{L^{n}(B_{2r}(x_{0})\cap\Gamma_2)}
       \\&+\|h\|_{L^{n}(B_{2r}(x_{0})\cap\Omega)}+\|h\|_{L^{n}(B_{2r}(x_{0})\cap\Gamma_2)}
       \\&+r^{-\epsilon}\|f\|_{L^{\frac{n}{n-\varepsilon}}(B_{2r}(x_{0})\cap\Omega)}^{\frac{1}{n-1}}+r^{-\epsilon}\|g\|_{L^{\frac{n-1}{n-1-\varepsilon}}(B_{2r}(x_{0})\cap\Omega)}^{\frac{1}{n-1}}
       \\&+\left(s+s^{n-1}+\|h\|_{C^{1}(\overline{\Omega})}\right)^{\frac{1}{n-1}}\bigg),
    \end{split}
    \end{equation*}
    where $C_{d,e,n,h,\epsilon}$ is a constant depending on $e,d,n,\|h\|_{C^{1}(\overline{\Omega})}$ and $\epsilon$.
\end{lem}
 \begin{proof}
For the proof of Lemma \ref{rele:2}, see Appendix B.
\end{proof}

\begin{lem}[$L^{\infty}$-estimate up to boundary for Neumann boundary condition]\label{le:2}
 Assume that $ \mathbf{a}(x,p)$ is a Carathodory function satisfying $\mathbf{a}(x,0) =0$, \eqref{mixeq1}, \eqref{mixeq2} and \eqref{mixeq3} for $d,e,s>0$ as in Lemma \ref{rele:2}.
 Let $u\in W_{loc}^{1,n}(\mathbb{R}^{n}_{+})$ be a weak solution of the following equation
    \begin{equation*}
    \begin{cases}
        -\operatorname{div}( \mathbf{a}(x,\nabla u))=f\ \ \text{\rm in}\ \mathbb{R}^{n}_{+},\\
        \mathbf{a}(x,\nabla u)\cdot\nu=g\ \ \text{\rm on }\partial \mathbb{R}^{n}_{+}
    \end{cases}
    \end{equation*}with $f\in L^{\frac{n}{n-\varepsilon}}(\mathbb{R}^{n}_{+})$ and $g\in L^{\frac{n-1}{n-1-\varepsilon}}(\partial\mathbb{R}^{n}_{+}) $ for
    some $0 < \varepsilon < 1$. Then for any $x_{0}\in \overline{\mathbb{R}^{n}_{+}}$, we have that
    \begin{equation*}
    \begin{split}
        \|u^+\|_{L^{\infty}(B_R(x_{0})\cap\mathbb{R}^{n}_{+} )}
       \leq &C\bigg(R^{-1}\|u^+\|_{L^{n}(B_{2R}(x_{0})\cap\mathbb{R}^{n}_{+})}+R^{-\frac{n}{n-1}}\|u^+\|_{L^{n}(B_{2R}(x_{0})\cap\partial\mathbb{R}^{n}_{+})}
       \\&+R^{\frac{\varepsilon}{n-1}}\|f\|_{L^{\frac{n}{n-\varepsilon}}(B_{2R}(x_{0})\cap\mathbb{R}^{n}_{+})}^{\frac{1}{n-1}}+R^{\frac{\varepsilon}{n-1}}\|g\|_{L^{\frac{n-1}{n-1-\varepsilon}}(B_{2R}(x_{0})\cap\mathbb{R}^{n}_{+})}^{\frac{1}{n-1}}
       \\&+(s+s^{n-1})^{\frac{1}{n-1}}\bigg),
    \end{split}
    \end{equation*}
    where $C=C(n, \varepsilon,d,e)$ is a constant. Moreover, if $ g\equiv 0$, then we have that
    \begin{equation*}
    \begin{split}
        \|u^+\|_{L^{\infty}(B_R(x_{0})\cap\mathbb{R}^{n}_{+} )}
       \leq &C\bigg(R^{-1}\|u^+\|_{L^{n}(B_{2R}(x_{0})\cap\mathbb{R}^{n}_{+})}+R^{\frac{\varepsilon}{n-1}}\|f\|_{L^{\frac{n}{n-\varepsilon}}(B_{2R}(x_{0})\cap\mathbb{R}^{n}_{+})}^{\frac{1}{n-1}} \\&+(s+s^{n-1})^{\frac{1}{n-1}}\bigg),
    \end{split}
    \end{equation*}
    where $C=C(n, \varepsilon,d,e)$ is a constant.
\end{lem}
\begin{proof}
 The proof is similar to the proof of (ii) in Lemma \ref{rele:2}, so we omit it.
\end{proof}

The estimates in Lemma \ref{pro2-1} are not sufficient to establish the logarithmic behavior of $u$ at infinity but are essentially optimal in the limiting case $f \in L^1(\Omega)$ and $ g\in L^1(\partial\Omega)$. Similar to \cite{S,S2}, a bit more integrability on $f$ and $g$ will give the global $L^{\infty}$-bound as stated in the following Proposition. We can also see that the following Proposition \ref{prop3-1} is an improvement of Proposition 4.1 in \cite{E}.
\begin{prop}\label{prop3-1}
    Let $\Omega\subset\mathbb{R}^{n}_{+}$, $ p>1$, $f \in L^p(\Omega)$, $ g\in L^p(\partial\Omega)$, assume \eqref{re2-2} and $\mathbf{a}(x,0) =0$. Let $u \in W^{1, n}(\Omega)$ be a weak solution of
   \begin{equation}\label{0923eq1}
        \begin{cases}
            -\operatorname{div} \mathbf{a}(x, \nabla u)=f \quad \text{ in }\Omega,\\
            \mathbf{a}(x, \nabla u)\cdot\nu =g \quad \text{ on }\Omega\cap\partial\mathbb{R}^{n}_{+},\\
            u=0 \quad \text{ on }\partial\Omega\cap\mathbb{R}^{n}_{+}.
        \end{cases}
    \end{equation}
    Then
 \begin{equation*}
 \begin{split}
     \|u\|_{\infty} \leq &C\left(\frac{(\|f\|_p +\|g\|_p)( |\Omega|^{\frac{(n-1)(p-1)}{2pn}}+|\Omega|^{\frac{p-1}{2p}} )}{d}+1\right)^{\alpha_0}(|\Omega|+|\Omega\cap\partial\mathbb{R}^{n}_{+}|+1)^{\beta_0}\\&\times\left(\|u\|_{L^{\frac{n p q_1}{p-1}}(\Omega)}^{\bar{q}}+ \|u\|_{L^{\frac{n p q_1}{p-1}}(\partial\mathbb{R}^{n}_+\cap\Omega)}^{\bar{q}}\right)
 \end{split}
\end{equation*}
for some constants $C, \alpha_0, \beta_0, \bar{q}>0$ depending only on $n, p$ and $q_1 \geq 1$.
\end{prop}
\begin{proof}
    Given $q \geq 1$ and $k>0$, we set
 \begin{equation*}
F(s)= \begin{cases}s^q, & \text { if } 0 \leq s \leq k, \\ q k^{q-1} s-(q-1) k^q, & \text { if } s \geq k,\end{cases}
\end{equation*}
and $G(s)=F(s)\left[F^{\prime}(s)\right]^{n-1}$. Notice that $G$ is a piecewise $C^1$-function with a corner just at $s=k$ so that
 \begin{equation}\label{3-4}
\left[F^{\prime}(s)\right]^n \leq G^{\prime}(s), \quad G(s) \leq q^{n-1} F^{\frac{n(q-1)+1}{q}}(s).
\end{equation}
Since $G(|u|) \in W^{1, n}(\Omega)$ for $G$ is linear at infinity, we use $\operatorname{sign}(u) G(|u|)$ as a test function in \eqref{0923eq1} and get
\begin{equation} \label{3-5}
    \begin{split}
        \int_{\Omega}|\nabla F(|u|)|^n &\leq \frac{1}{d} \int_{\Omega} G^{\prime}(|u|)\langle\mathbf{a}(x, \nabla u), \nabla u\rangle
        \\&=\frac{1}{d} \bigg(\int_{\Omega} f \operatorname{sign}(u) G(|u|)+\int_{\Omega\cap\partial\mathbb{R}^{n}_{+}} g \operatorname{sign}(u) G(|u|)\bigg)
    \end{split}
\end{equation}
in view of \eqref{re2-2}, \eqref{3-4} and $\mathbf{a}(x,0) =0$. Setting $m=\frac{p}{p-1}$ in view of $p>1$, by \eqref{3-4} and H\"{o}lder's inequality, we deduce that
$$
\left|\int_{\Omega} f \operatorname{sign}(u) G(|u|)\right| \leq q^{n-1} \int_{\Omega}|f| F^{\frac{n(q-1)+1}{q}}(|u|) \leq q^{n-1}|\Omega|^{\frac{n-1}{m n q}}\|f\|_p\left(\int_{\Omega} F^{m n}(|u|)\right)^{\frac{n(q-1)+1}{m n q}}
$$
and
\begin{equation}\label{3-6}
    \begin{split}
        \left|\int_{\Omega\cap\partial\mathbb{R}^{n}_{+}} g \operatorname{sign}(u) G(|u|)\right| &\leq q^{n-1} \int_{\Omega}|g| F^{\frac{n(q-1)+1}{q}}(|u|)
        \\&\leq q^{n-1}|\Omega\cap\partial\mathbb{R}^{n}_{+}|^{\frac{n-1}{m n q}}\|g\|_p\left(\int_{\Omega\cap\partial\mathbb{R}^{n}_{+}} F^{m n}(|u|)\right)^{\frac{n(q-1)+1}{m n q}}.
    \end{split}
\end{equation}
 Since $F(|u|) \in W^{1, n}(\Omega)$ and $F(|u|) =0 $ on $\partial\Omega\cap\mathbb{R}^{n}_{+} $, using Lemma \ref{relem:2.1}, H\"{o}lder's inequality and Sobolev trace inequality, we get that
 \begin{equation*}
    \begin{split}
        &\left(\int_{\Omega\cap\partial\mathbb{R}^{n}_{+}} F^{2 m n}(|u|)\right)^{\frac{1}{2 m}}+\left(\int_{\Omega} F^{2 m n}(|u|)\right)^{\frac{1}{2 m}}
        \\\leq & C\left(\left(\int_{\Omega}|\nabla F(|u|)|^{\frac{2mn^2}{n-1+2mn}}\right)^{\frac{n-1+2mn}{2mn}}+ \left(\int_{\Omega}|\nabla F(|u|)|^{\frac{2mn^2}{n+2mn}}\right)^{\frac{n+2mn}{2mn}} \right)
        \\\leq &C ( |\Omega|^{\frac{n-1}{2mn}}+|\Omega|^{\frac{1}{2m}} )\int_{\Omega\cap\partial\mathbb{R}^{n}_{+}}|\nabla F(|u|)|^n
        \\\leq &\frac{C( |\Omega|^{\frac{n-1}{2mn}}+|\Omega|^{\frac{1}{2m}} )}{d} q^{n-1}\bigg( \|f\|_p+\|g\|_p  \bigg)\bigg(|\Omega|+|\Omega\cap\partial\mathbb{R}^{n}_{+}|  \bigg)^{\frac{n-1}{m n q}}
        \\&\times\bigg(\left(\int_{\Omega} F^{m n}(|u|)\right)^{\frac{n(q-1)+1}{m n q}}+\left(\int_{\Omega\cap\partial\mathbb{R}^{n}_{+}} F^{m n}(|u|)\right)^{\frac{n(q-1)+1}{m n q}} \bigg)
    \end{split}
\end{equation*}
for some $C \geq 1$ in view of \eqref{3-4}, \eqref{3-5} and \eqref{3-6}. Since $F(s) \rightarrow s^q$ in a monotone way as $k \rightarrow+\infty$, by $(a+b)^p\leq a^p+b^p$ for $a,b\geq 0,\hbox{ }0<p<1$, and $a^p+b^p\leq 2(a+b)^p $ for $a,b,p>0$, we have that
 \begin{equation}\label{3-7}
    \begin{split}
        &\left(\int_{\Omega\cap\partial\mathbb{R}^{n}_{+}}|u|^{2 m n q}+\int_{\Omega}|u|^{2 m n q}\right)^{\frac{1}{2 m q}}
        \\\leq &\left( \left(  \int_{\Omega\cap\partial\mathbb{R}^{n}_{+}}|u|^{2 m n q}   \right)^{\frac{1}{2m} } + \left(\int_{\Omega}|u|^{2 m n q} \right)^{\frac{1}{2m}} \right)^{\frac{1}{ q}}
        \\\leq &\exp \left[\frac{1}{q} \ln \frac{CD_{f,g,p,n,\Omega}}{d}+\frac{(n-1) \ln D_{\Omega}}{m n q^2}+(n-1) \frac{\ln q}{q}\right]
        \\&\times\left(\left(\int_{\Omega\cap\partial\mathbb{R}^{n}_{+}}|u|^{m n q}\right)^{\frac{n(q-1)+1}{m n q} } +\left(\int_{\Omega}|u|^{m n q}\right)^{\frac{n(q-1)+1}{m n q}}\right)^{\frac{1}{q}}
         \\\leq &\exp \left[\frac{1}{q} \ln \frac{2CD_{f,g,p,n,\Omega}}{d}+\frac{(n-1) \ln D_{\Omega}}{m n q^2}+(n-1) \frac{\ln q}{q}\right]
        \\&\times\left(\int_{\Omega\cap\partial\mathbb{R}^{n}_{+}}|u|^{m n q}+\int_{\Omega}|u|^{m n q}\right)^{\frac{1}{m q}\left[1-\frac{n-1}{n q}\right]},
    \end{split}
\end{equation}
where $ D_{f,g,p,n,\Omega} = (\|f\|_p +\|g\|_p)( |\Omega|^{\frac{(n-1)(p-1)}{2pn}}+|\Omega|^{\frac{p-1}{2p}} )$ and $ D_{\Omega} =|\Omega|+|\Omega\cap\partial\mathbb{R}^{n}_{+}|$.
Assume now that $u \in L^{m n q_1}(\Omega)$ for some $q_1 \geq 1$. Setting $q_j=2^{j-1} q_1, j \in \mathbb{N}$, by iterating \eqref{3-7}, we deduce that
\begin{equation*}
    \begin{split}
     &\left(\int_{\Omega\cap\mathbb{R}^{n}_{+}}|u|^{2 m n q}+\int_{\Omega}|u|^{2 m n q}\right)^{\frac{1}{2 m q}}
     \\\leq &\exp \left[\frac{1}{q_j} \ln \frac{CD_{f,g,p,n,\Omega}}{d}+\frac{(n-1) \ln D_{\Omega}}{m n q_j^2}+(n-1) \frac{\ln q_j}{q_j}\right]\left[\left(\int_{\Omega\cap\mathbb{R}^{n}_{+}}|u|^{m n q_j}+\int_{\Omega}|u|^{m n q_j}\right)^{\frac{1}{m q_j}}\right]^{1-\frac{n-1}{n q_j}} \\
 \leq &\exp \left[\ln \frac{CD_{f,g,p,n,\Omega}}{d} \sum_{k=j-1}^j \frac{a_k^j}{q_k}+\frac{(n-1) \ln D_{\Omega}}{m n} \sum_{k=j-1}^j \frac{a_k^j}{q_k^2}+(n-1) \sum_{k=j-1}^j \frac{a_k^j \ln q_k}{q_k}\right]
 \\&\times\left[\left(\int_{\Omega\cap\mathbb{R}^{n}_{+}}|u|^{m n q_{j-1}}+\int_{\Omega}|u|^{m n q_{j-1}}\right)^{\frac{1}{m q_{j-1}}}\right]^{a_{j-2}^j} \\
 \cdots \leq &\exp \left[\ln \frac{CD_{f,g,p,n,\Omega}}{d} \sum_{k=1}^j \frac{a_k^j}{q_k}+\frac{(n-1) \ln D_{\Omega}}{m n} \sum_{k=1}^j \frac{a_k^j}{q_k^2}+(n-1) \sum_{k=1}^j \frac{a_k^j \ln q_k}{q_k}\right]
 \\&\times\left(\int_{\Omega\cap\mathbb{R}^{n}_{+}}|u|^{m n q_1}+\int_{\Omega}|u|^{m n q_1}\right)^{\frac{a_0^j}{m q_1}},
    \end{split}
\end{equation*}
where
$$
a_k^j= \begin{cases}{\left[1-\frac{n-1}{n q_{k+1}}\right] \times \cdots \times\left[1-\frac{n-1}{n q_j}\right]}, & \text { if } 0 \leq k<j, \\ 1, & \text { if } k=j.\end{cases}
$$
Since $a_k^j \leq 1$ for all $k=0, \ldots, j$, we have that
 \begin{equation*}
\begin{aligned}
& \alpha_0=\frac{1}{n} \sup _{j \in \mathbb{N}} \sum_{k=1}^j \frac{a_k^j}{q_k} \leq \frac{1}{n} \sup _{j \in \mathbb{N}} \sum_{k=1}^j \frac{1}{q_k}=\frac{2}{n} \sum_{k=1}^{\infty} \frac{1}{q_1 2^k}<\infty, \\
& \beta_0=\frac{n-1}{m n^2} \sup _{j \in \mathbb{N}} \sum_{k=1}^j \frac{a_k^j}{q_k^2} \leq \frac{4(n-1)}{m n^2} \sum_{k=1}^{\infty} \frac{1}{q_1^2 4^k}<+\infty, \\
& \gamma_0=\frac{n-1}{n} \sup _{j \in \mathbb{N}} \sum_{k=1}^j \frac{a_k^j \ln q_k}{q_k} \leq 2 \frac{n-1}{n} \sum_{k=1}^{\infty} \frac{(k-1) \ln 2+\ln q_1}{q_1 2^k}<+\infty.
\end{aligned}
\end{equation*}
Thus we get that
\begin{equation}\label{3-8}
    \begin{split}
        &\left(\int_{\Omega\cap\mathbb{R}^{n}_{+}}|u|^{m n q_{j+1}}+\int_{\Omega}|u|^{m n q_{j+1}}\right)^{\frac{1}{m n q_{j+1}}}
        \\\leq &\exp \left[\alpha_0 \ln C\left(\frac{D_{f,g,p,n,\Omega}}{d}+1\right)+\beta_0 \ln (D_{\Omega}+1)+\gamma_0\right]\left(\int_{\Omega\cap\mathbb{R}^{n}_{+}}|u|^{m n q_1}+\int_{\Omega}|u|^{m n q_1}\right)^{\frac{a_0^j}{m n q_1}}.
    \end{split}
\end{equation}
Since
 \begin{equation*}
\bar{q}=\lim _{j \rightarrow+\infty} a_0^j=\prod_{k=1}^{\infty}\left(1-\frac{n-1}{n q_k}\right)<\infty,
\end{equation*}
letting $j \rightarrow+\infty$ in \eqref{3-8}, we finally deduce that
 \begin{equation*}
\|u\|_{\infty} \leq e^{\alpha_0 \ln C+\gamma_0}\left(\frac{D_{f,g,p,n,\Omega}}{d}+1\right)^{\alpha_0}(D_{\Omega}+1)^{\beta_0}(\|u\|_{L^{m n q_1}(\partial\Omega)}^{\bar{q}}+\|u\|_{L^{m n q_1}(\Omega)}^{\bar{q}} ),
\end{equation*}
and hence complete our proof.
\end{proof}

\smallskip

We also need the following generalized version of the divergence theorem from \cite{CL1}.
\begin{lem}\cite[Lemma 4.3]{CL1}\label{le:2.3}
   Let $\Omega$ be a bounded open subset of $\mathbb{R}^{n}_{+}$ with Lipschitz boundary and let $F\in L^1(\Omega)$. Assume that
   ${\bf{a}} \in C^0(\overline{\Omega};\mathbb{R}^{n}_{+})$ satisfies $\operatorname{div} \, {\bf{a}}=F$ in the sense of distributions
   in $\Omega$. Then we have
   \begin{equation*}
    \int_{\partial\Omega}\left\langle {\bf{a}},\nu\right\rangle \mathrm{d}\sigma=\int_{\Omega}F(x)\mathrm{d}x.
   \end{equation*}
\end{lem}
We also need the following two Lemmas on the comparison principle and Liouville type result in $\mathbb{R}^{n}_{+}$.
\begin{lem}[Comparison principle 1]\label{le:2.5}
    Let $E\subset\mathbb{R}^{n}_{+}$ be a bounded domain such that $\mathcal{H}^{n-1}(\Gamma_0)>0$, where $\Gamma_0:=\mathbb{R}^{n}_{+}\cap \partial E$. Suppose that $\mathbb{R}^{n}_{+}\cap E$ is connected. Assume that $u,v\in W^{1,n}(\mathbb{R}^{n}_{+}\cap E)\cap C^{0}\big((\mathbb{R}^{n}_{+}\cap E)\cup \Gamma_0\big)$ satisfy
    \begin{equation*}
     \left\{
         \begin{aligned}
             &-\Delta_{n}u\leq-\Delta_{n}v \ &\rm{in}\,\, \mathbb{R}^{n}_{+}\cap E, \\
             &u\leq v\ &\rm{on}\,\, \Gamma_0,\\
             &\langle |\nabla u|^{n-2}\nabla u,\nu\rangle\leq \langle  |\nabla v|^{n-2}\nabla v,\nu\rangle \quad\,\, &\rm{on}\,\, \partial(\mathbb{R}^{n}_{+}\cap E)\setminus\Gamma_0.
             \end{aligned}
             \right.
    \end{equation*}
    Then $u\leq v$ in $\mathbb{R}^{n}_{+}\cap E$.
 \end{lem}
 \begin{proof}
     Since $\langle |\nabla u|^{n-2}\nabla u,\nu\rangle\leq \langle  |\nabla v|^{n-2}\nabla v,\nu\rangle \text{ on } \partial(\mathbb{R}^{n}_{+}\cap E)\setminus\Gamma_0,$ we can prove this Lemma by a similar way as the proof of Lemma 2.4 in \cite{CL1}.
 \end{proof}
\begin{lem}\cite[Lemma 3.2]{CL1}\label{le:2.8}
    Let $\mathbb{R}^{n}_{+}$ be the same as in Theorem \ref{Th:1-1}. Let $\gamma\in\mathbb{R}$ be a constant. Assume that $G(x)\in W^{1,n}_{loc}(\overline{\mathbb{R}^{n}_{+}}\setminus\{0\})\cap L^{\infty}(\mathbb{R}^{n}_{+})$ and the function $\gamma\log |x|+G(x)$ satisfies
    \begin{equation*}
        \left\{
            \begin{aligned}
                &\Delta_{n}(\gamma\log |x|+G(x))=0 \ &\rm{in}\ \mathbb{R}^{n}_{+}, \\
                &\langle |\nabla(\gamma\log |x|+G(x))|^{n-2}\big(\nabla(\gamma\log |x|+G(x))\big),\nu\rangle=0\ &\rm{on}\ \partial\mathbb{R}^{n}_{+}\setminus\{0\}.
                \end{aligned}
                \right.
       \end{equation*}
       Then $G(x)$ is a constant function.
\end{lem}

Finally, we also need the following comparison principle in $ B_{r}(x)\cap\mathbb{R}^{n}_{+}$.
\begin{lem}[Comparison principle 2]\label{le:2.7}
   Let $x\in \overline{\mathbb{R}^{n}_{+} }$ and $r>0$. Assume that $u,v\in W^{1,n}\big(\mathbb{R}^{n}_{+}\cap B_{r}(x)\big)$ satisfy
    \begin{equation*}
     \begin{cases}
          -\Delta_{n}u\leq-\Delta_{n}v \ &\rm{in}\,\, \mathbb{R}^{n}_{+}\cap B_{r}(x), \\
             u\leq v\ &\rm{on}\,\, \partial \big(\mathbb{R}^{n}_{+}\cap B_{r}(x)\big).\\
     \end{cases}
    \end{equation*}
    Then $u\leq v$ in $\mathbb{R}^{n}_{+}\cap B_{r}(x)$.
 \end{lem}
 \begin{proof}
      Since $u\leq v\ \rm{on}\,\, \partial \big(\mathbb{R}^{n}_{+}\cap B_{r}(x)\big)$, we can prove that $\displaystyle\int_{ \mathbb{R}^{n}_{+}\cap B_{r}(x) }|\nabla (u-v)^+|^n \dx =0$ in a similar way as the proof of \cite[Lemma 2.4]{CL1}. Thus by Corollary 3.32 in \cite{AF}, we get that $(u-v)^+ $ is a contant in $ \mathbb{R}^{n}_{+}\cap B_{r}(x)$. Since $(u-v)^+=0 $ on $\partial \big(\mathbb{R}^{n}_{+}\cap B_{r}(x)\big) $, we have $(u-v)^+=0 $ in $ \mathbb{R}^{n}_{+}\cap B_{r}(x)$. Thus $u\leq v$ in $\mathbb{R}^{n}_{+}\cap B_{r}(x)$.
 \end{proof}


\section{Behavior of $u$ and $\nabla u$ at infinity}
In this section, we will prove sharp asymptotic estimates on both $u$ and $ \nabla u$ at infinity. To this end, we will first show the following upper bound on $u$.
\begin{prop}\label{prop:3.1}
    Let $u$ be a solution of
    \begin{equation*}
        \left\{
        \begin{aligned}
        &-\Delta_{n}u=e^u \qquad &\rm{in}\,\, \mathbb{R}^{n}_{+}, \\
        &|\nabla u|^{n-2}\nabla u\cdot \nu =e^{\frac{n-1}{n}u}\qquad &\rm{on}\,\, \partial\mathbb{R}^{n}_{+},\\
        &\int_{\mathbb{R}^{n}_{+}}e^{u}\mathrm{d}x<+\infty \hbox{ and }\int_{\partial \mathbb{R}^{n}_{+}}e^{\frac{n-1}{n}u}\mathrm{d}x<+\infty,
        \end{aligned}
        \right.
    \end{equation*}
    where $n\geq 2$.
Then $u^{+}\in L^{\infty}(\mathbb{R}^{n}_{+})$ and $u\in C^{1,\theta}(\overline{\mathbb{R}^{n}_{+}})$ for some $\theta\in(0,1)$.
\end{prop}
\begin{proof}
 For $\overline{x}=(\overline{x}_1,\cdots,\overline{x}_n)\in\overline{\mathbb{R}^{n}_{+}}$ and $0<r<1$, let $\overline{x}^{*} =(\overline{x}_1,\cdots,-\overline{x}_n)\in\overline{\mathbb{R}^{n}_{-}}$. Since $ u\in W^{1,n}_{\text{loc}}(\mathbb{R}^{n}_{+})$, according to \cite{BBGGPV}, we let $v\in W^{1,n}\big(B_r(\overline{x})\cup B_r(\overline{x}^{*})\big)$ be the weak
 solution of the equation
 \begin{equation*}
    \left\{
        \begin{aligned}
        &-\operatorname{div} ( \mathbf{a}(x,v)) =0 \qquad &{\rm in}\,\, B_r(\overline{x})\cup B_r(\overline{x}^{*}), \\
        &v=0\qquad &{\rm{on}} \,\, \partial \big( B_r(\overline{x})\cup B_r(\overline{x}^{*})\big),\\
        \end{aligned}
        \right.
 \end{equation*}
 where $\mathbf{a}(x, p)=|\nabla \bar{u}(x)|^{n-2} \nabla \bar{u}(x)-|\nabla \bar{u}(x)-p|^{n-2}(\nabla \bar{u}(x)-p)$ and
 \begin{equation*}
     \begin{cases}
         \bar{u}(x',x_n) = u(x',x_n) \quad  \text{ for }x_{n}\geq 0,\\
         \bar{u}(x',x_n) = u(x',-x_n) \quad  \text{ for }x_{n}< 0.
     \end{cases}
 \end{equation*} Let $ h= v+\bar{u}$, then $ h$ satisfies
\begin{equation*}
    \left\{
        \begin{aligned}
        &-\Delta_{n}h=0 \qquad &{\rm in}\,\, B_r(\overline{x})\cup B_r(\overline{x}^{*}), \\
        &h=\bar{u}\qquad &{\rm{on}} \,\, \partial \big(B_r(\overline{x})\cup B_r(\overline{x}^{*})\big).\\
        \end{aligned}
        \right.
 \end{equation*}
 We will prove that $h(x',x_n)=h(x',-x_n) $. In fact, setting $h_1(x',x_n) = h(x',-x_n)$ for $x_{n}\geq 0$, then $ h_1$ satisfies
  \begin{equation*}
    \left\{
        \begin{aligned}
        &-\Delta_{n}h_1=0 \qquad &{\rm in}\,\, B_r(\overline{x})\cap\mathbb{R}^{n}_{+}, \\
        &h_1=h\qquad &{\rm{on}} \,\, \partial B_r(\overline{x})\cap\mathbb{R}^{n}_{+},\\
         &h_1=h\qquad &{\rm{on}} \,\,  B_r(\overline{x})\cap\partial\mathbb{R}^{n}_{+}.\\
        \end{aligned}
        \right.
 \end{equation*}
 By Lemma \ref{le:2.7}, we get that $h_1=h$, that is, $h(x',x_n)=h(x',-x_n) $. Since $h(x',x_n)=h(x',-x_n) $, we get that $h\in W^{1,n}(B_r(\overline{x})\cap \mathbb{R}^{n}_{+})$ satisfies
 \begin{equation*}
    \left\{
        \begin{aligned}
        &-\Delta_{n}h=0 \qquad &{\rm in}\,\, B_r(\overline{x})\cap\mathbb{R}^{n}_{+}, \\
        &h=u\qquad &{\rm{on}} \,\, \partial B_r(\overline{x})\cap\mathbb{R}^{n}_{+},\\
         &|\nabla h|^{n-2}\frac{\partial h}{\partial \nu}=0\qquad &{\rm{on}} \,\,  B_r(\overline{x})\cap\partial\mathbb{R}^{n}_{+}.\\
        \end{aligned}
        \right.
 \end{equation*}
 From Lemma \ref{le:2.5}, we get $h\leq u$ in $B_r(\overline{x})\cap\mathbb{R}^{n}_{+}$. Then we have
 \begin{equation}\label{ineq:3.3}
    \int_{B_r(\overline{x})\cap\mathbb{R}^{n}_{+}}(h^+)^n\mathrm{d}x\leq \int_{B_r(\overline{x})\cap\mathbb{R}^{n}_{+}}(u^+)^n\mathrm{d}x.
 \end{equation}
 Applying Lemma \ref{le:2} to $h^+$, we deduce that
 \begin{equation}\label{ineq:3.4}
    \|h^+\|_{L^{\infty}\left(B_{\frac{r}{2}}(\overline{x})\cap\mathbb{R}^{n}_{+}\right)}\leq C(r)
 \end{equation}
 for some constant $C(r)>0$. Now, let $r>0$ small enough such that
 \begin{equation}\label{ineq:3.5}
    \left(\int_{B_{r}(\overline{x})\cap\mathbb{R}^{n}_{+}}e^{u}\mathrm{d}x+\int_{B_{r}(\overline{x})\cap\partial \mathbb{R}^{n}_{+}}e^{\frac{n-1}{n}u}\mathrm{d}x\right)^{\frac{1}{n-1}}\leq \frac{S_{1,+}^{\frac{n}{n-1}}(n-1)}{n}
 \end{equation}
 and
 \begin{equation}\label{reineq:3.5}
     \left(\int_{B_{r}(\overline{x})\cap\mathbb{R}^{n}_{+}}e^{u}\mathrm{d}x+\int_{B_{r}(\overline{x})\cap\partial \mathbb{R}^{n}_{+}}e^{\frac{n-1}{n}u}\mathrm{d}x\right)^{\frac{1}{n-1}}\leq \frac{S_{2,+}^{\frac{n}{2(n-1)}}(n-1)}{n},
 \end{equation}
 where $S_{1,+},S_{2,+}$ is the same as in Lemma \ref{re2-1}. Choosing $\lambda=\frac{n}{n-1}$, $\beta=\frac{n}{n-1}$ and $q=2$  in Lemma \ref{re2-1}, we obtain that
 \begin{equation}\label{ineq:3.6}
    \int_{B_{r}(\overline{x})\cap\mathbb{R}^{n}_{+}}e^{\frac{n}{n-1}(u-h)}\mathrm{d}x\leq C(r)\qquad \hbox{ and } \qquad \int_{B_{r}(\overline{x})\cap\partial\mathbb{R}^{n}_{+}}e^{\frac{n}{n-1}(u-h)}\mathrm{d}x\leq C(r).
 \end{equation}
 Then, by \eqref{ineq:3.4} and \eqref{ineq:3.6}, we have
 \begin{equation}\label{ineq:3.7}
    \int_{B_{\frac{r}{2}}(\overline{x})\cap\mathbb{R}^{n}_{+}}e^{\frac{n}{n-1}u}\mathrm{d}x=\int_{B_{\frac{r}{2}}(\overline{x})\cap\mathbb{R}^{n}_{+}}e^{\frac{n}{n-1}(u-h)+\frac{n}{n-1}h}\mathrm{d}x
    \leq C(r)e^{C(r)}
 \end{equation}
 and
  \begin{equation}\label{reineq:3.7}
    \int_{B_{\frac{r}{2}}(\overline{x})\cap\partial\mathbb{R}^{n}_{+}}e^{\frac{n}{n-1}u}\mathrm{d}x=\int_{B_{\frac{r}{2}}(\overline{x})\cap\partial\mathbb{R}^{n}_{+}}e^{\frac{n}{n-1}(u-h)+\frac{n}{n-1}h}\mathrm{d}x
    \leq C(r)e^{C(r)}.
 \end{equation}
Using Lemma \ref{le:2} in conjunction with \eqref{ineq:3.3}, \eqref{ineq:3.7} and \eqref{reineq:3.7}, we obtain
\begin{equation}\label{ineq:3.8}
    \|u^+\|_{L^{\infty}\left(B_{\frac{r}{4}}(\overline{x})\cap\mathbb{R}^{n}_{+}\right)}\leq C(r).
\end{equation}
Since $\displaystyle\int_{\mathbb{R}^{n}_{+}}e^{u}\mathrm{d}x<+\infty$ and $\displaystyle\int_{\partial\mathbb{R}^{n}_{+}}e^{\frac{n-1}{n}u}\mathrm{d}x<+\infty$, there exists $R>0$ large enough such that
\begin{equation}\label{ineq:3.9}
    \left(\int_{\mathbb{R}^{n}_{+}\setminus B_{R}(0)}e^{u}\mathrm{d}x+\int_{\partial\mathbb{R}^{n}_{+}\setminus B_{R}(0)}e^{\frac{n-1}{n}u}\mathrm{d}x\right)^{\frac{1}{n-1}}\leq \frac{S_{1,+}^{\frac{n}{n-1}}(n-1)}{n}
 \end{equation}
 and
 \begin{equation}\label{reineq:3.9}
     \left(\int_{\mathbb{R}^{n}_{+}\setminus B_{R}(0)}e^{u}\mathrm{d}x+\int_{\partial\mathbb{R}^{n}_{+}\setminus B_{R}(0)}e^{\frac{n-1}{n}u}\mathrm{d}x\right)^{\frac{1}{n-1}}\leq \frac{S_{2,+}^{\frac{n}{2(n-1)}}(n-1)}{n}.
 \end{equation}

For every $\overline{x}\in\mathbb{R}^{n}_{+}\setminus \overline{B_{R+1}(0)}$, we know that $B_1(\overline{x})\cap\mathbb{R}^{n}_{+}\subset\mathbb{R}^{n}_{+}\setminus B_{R}(0)$. By \eqref{ineq:3.9} and \eqref{reineq:3.9}, we get the validity of \eqref{ineq:3.5} and \eqref{reineq:3.5} with $r=1$. Hence, by Lemmas \ref{re2-1} and \ref{le:2}, we have
\begin{equation}\label{ineq:3.10}
    \|u^+\|_{L^{\infty}(\mathbb{R}^{n}_{+}\setminus \overline{B_{R+1}(0)})}\leq C(1).
\end{equation}
By the compactness of $\overline{\mathbb{R}^{n}_{+}\cap B_{R+1}}$, there exist finite points $x_i$ and $r_{x_i}>0$ small enough
$(i=1,\cdots, L)$ fulfilling \eqref{ineq:3.5} so that
\begin{equation*}
    \overline{\mathbb{R}^{n}_{+}\cap B_{R+1}}\subset\bigcup_{i=1}^{L}B_{\frac{r_{x_i}}{4}}(x_i).
\end{equation*}
Thus by \eqref{ineq:3.8}, we have
\begin{equation}\label{ineq:3.12}
    \|u^+\|_{L^{\infty}(\mathbb{R}^{n}_{+}\cap\overline{B_{R+1}})}\leq \max_{1\leq i\leq L}C(r_{x_i})<+\infty.
\end{equation}
Combining \eqref{ineq:3.10} and \eqref{ineq:3.12}, we have $u^+\in L^{\infty}(\mathbb{R}^{n}_{+})$, and hence $e^{u}\in L^{\infty}(\mathbb{R}^{n}_{+})$. Since $u\in L^N_{loc}(\mathbb{R}^{n}_{+})$ and $e^{u}\in L^{\infty}(\mathbb{R}^{n}_{+})$, by the regularity results in \cite{D,L,S2,T}, we
 deduce that $u\in C^{1,\theta}_{loc}(\overline{\mathbb{R}^{n}_{+}})$ for some $\theta\in(0,1)$.
\end{proof}


\begin{prop}\label{reprop:3.1}
Let $u\in W^{1,n}_{\text{loc}}( \mathbb{R}^{n}_{+})$ be the solution of
    \begin{equation*}
    \begin{cases}
        -\Delta_n u = e^u&\text{ in } \mathbb{R}^{n}_{+},\\
        |\nabla u|^{n-2}\nabla u\cdot\nu =e^{\frac{n-1}{n}u} &\text{ on }\partial \mathbb{R}^{n}_{+},\\
        \displaystyle\int_{\mathbb{R}^{n}_{+} }e^{u}\dx <+\infty,\text{ }\displaystyle\int_{\partial\mathbb{R}^{n}_{+} }e^{\frac{n-1}{n}u}\dx <+\infty \quad \text{ and } \quad 0\leq f(x)\leq e^x,
    \end{cases}
    \end{equation*}
     and $\hat{u}(x)=u\left(\frac{x}{|x|^2}\right)$, then
     \begin{equation}\label{re2-11}
         \hat{u}\in W^{1,q}_{\text{loc}}(\mathbb{R}^{n}_{+})
     \end{equation} for all $1 \leq q<n$. Moreover, for $r$ fixed small enough, there are $ U_{0},H_{0} \in C^{1}(B_{r}\setminus\{0\})\cap C^{0}( \overline{B_{r}}\setminus\{0\})$ such that $\hat{u} = U_{0}+H_{0}$ .
\end{prop}
\begin{proof}
 Firstly, $\hat{u}$ solves
 \begin{equation*}
\begin{cases}
-\Delta_n \hat{u}=\frac{e^{\hat{u} }}{|x|^{2 n}} \quad &\text { in } \mathbb{R}^{n}_{+} \backslash\{0\} \\
|\nabla \hat{u} |^{n-2}\frac{\partial\hat{u}}{\partial \nu} = \frac{e^{\frac{n-1}{n}\hat{u}}}{|x|^{2(n-1)}}&\text{ on }\partial \mathbb{R}^{n}_{+}\backslash\{0\}\\
\displaystyle\int_{\mathbb{R}^{n}_{+}} \frac{e^{\hat{u} }}{|x|^{2 n}}<+\infty,\hbox{ }\displaystyle\int_{\partial\mathbb{R}^{n}_{+}} \frac{e^{\frac{n-1}{n}\hat{u}}}{|x|^{2 (n-1)}}<+\infty
\end{cases}
\end{equation*}
in the weak sense
$$
\int_{\mathbb{R}^{n}_{+}}|\nabla \hat{u}|^{n-2}\langle\nabla \hat{u}, \nabla \Phi\rangle=\int_{\mathbb{R}^{n}_{+}} \frac{e^{\hat{u} }}{|x|^{2 n}} \Phi+\int_{\partial\mathbb{R}^{n}_{+}} \frac{e^{\frac{n-1}{n}\hat{u}}}{|x|^{2 (n-1)}} \Phi, \quad \forall \Phi  \in \hat{H}=\{\Phi: \hat{\Phi} \in H\},
$$
where $H:= \{\Phi\in W^{1,n}_{0}(\Omega): \Omega\subset \mathbb{R}^{n} \hbox{ bounded} \} $.

    By Proposition \ref{prop:3.1}, we know that $\hat{u} \in C^{1, \alpha}\left(\mathbb{R}^{n}_{+} \backslash\{0\}\right)$. Since $\hat{u} \in C^{1, \alpha}\left(\mathbb{R}^{n}_{+} \backslash\{0\}\right)$ and using \cite[Theorem 15.19]{GT}, we fix $r>0$ small and, for all $0<\epsilon<r$, let $H_\epsilon \in C^{0}\left(\overline{A_\epsilon}\right)\cap C^{2,\alpha}\left(A_\epsilon\right)$ satisfy
\begin{equation*}
    \begin{cases}
        \Delta_n H_\epsilon=0 & \text { in }  A_\epsilon:=(B_r(0) \backslash B_\epsilon(0)), \\ H_\epsilon=\hat{u} & \text { on }  \partial (B_r(0) \backslash B_\epsilon(0))\cap \mathbb{R}^{n}_{+},
\\ H_\epsilon(x',x_n)=\hat{u}(x',-x_n) &\text { on }\partial (B_r(0) \backslash B_\epsilon(0))\cap \mathbb{R}^{n}_{-}.
    \end{cases}
\end{equation*}
From \cite[Theorem 2]{D}, we get that $\tilde{H}_{\epsilon} \in C^{1,\alpha}_{\text{loc}}(B_r(0) \backslash B_\epsilon(0))$. Since $\hat{u}\in C^{1,\alpha}_{\text{loc}}\left(\overline{\mathbb{R}^{n}_{+}}\setminus\{0\}\right) $, we get that $\tilde{H}_{\epsilon}\in C^{0,1}\left(\partial (B_r(0) \backslash B_\epsilon(0))\right) $. Thus from \cite[Theorem 3.1]{L2}, we get that $\tilde{H}_{\epsilon}\in C^{0,\alpha}\left( \overline{B_r(0) \backslash B_\epsilon(0) }\right) $. So $H_{\epsilon}\in C^{1, \alpha}\left(A_\epsilon\right) \cap C^{0, \alpha}\left(\overline{A_\epsilon}\right)$.

Moreover, we have $H(x',x_n)= H(x',-x_n)$. In fact, setting $ H_{1,\epsilon}(x',x_n) =H(x',-x_n)$ for $x\in \overline{A_\epsilon}\cap \mathbb{R}^{n}_{+} $, then $H_{1,\epsilon}$ satisfies
\begin{equation*}
    \begin{cases}
        \Delta_n H_{1,\epsilon}=0 & \text { in }  A_\epsilon:=(B_r(0) \backslash B_\epsilon(0)), \\ H_{1,\epsilon}=\hat{u} & \text { on }  \partial (B_r(0) \backslash B_\epsilon(0))\cap \mathbb{R}^{n}_{+},
\\ H_{1,\epsilon}=H_{\epsilon} &\text { on } (B_r(0) \backslash B_\epsilon(0))\cap\partial \mathbb{R}^{n}_{-}.
    \end{cases}
\end{equation*}
Using Lemma \ref{le:2.7}, we get that $ H_{1,\epsilon}=H_{\epsilon}$. That is, $H(x',x_n)= H(x',-x_n)$. Thus $ H_\epsilon  $ satisfies
 \begin{equation}\label{2-8}
\begin{cases}\Delta_n H_\epsilon=0 & \text { in } A_\epsilon:=(B_r(0) \backslash B_\epsilon(0))\cap \mathbb{R}^{n}_{+}, \\ H_\epsilon=\hat{u} & \text { on }  \partial (B_r(0) \backslash B_\epsilon(0))\cap \mathbb{R}^{n}_{+},
\\ |\nabla H_{\epsilon} |^{n-2}\frac{\partial H_{\epsilon}}{\partial \nu} = 0 &\text { on }A_{\epsilon}\cap \partial\mathbb{R}^{n}_{+}.\end{cases}
\end{equation}
    By \eqref{2-8}, the function $U_\epsilon=\hat{u}-H_\epsilon \in C^{1, \alpha}\left(B_\epsilon\right) \cap C^{0, \alpha}\left(\overline{A_\epsilon}\right)$ satisfies
 \begin{equation}\label{mo2-8}
\begin{cases}\Delta_n\left(\hat{u}-U_\epsilon\right)=0 & \text { in } B_\epsilon:=A_\epsilon\cap \mathbb{R}^{n}_{+}, \\ U_\epsilon=0 & \text { on }  \partial B_\epsilon \cap \mathbb{R}^{n}_{+},\\ |\nabla (\hat{u}-U_\epsilon) |^{n-2}\frac{\partial (\hat{u}-U_\epsilon)}{\partial \nu} = 0 &\text { on }B_\epsilon\cap \partial\mathbb{R}^{n}_{+}.
\end{cases}
\end{equation}
We aim to derive estimates on $H_\epsilon$ and $U_\epsilon$ on the whole $B_\epsilon$ by using Lemma \ref{pro2-1} with
\begin{equation}\label{re2-8}
    \mathbf{a}(x, p):=|\nabla \hat{u}(x)|^{n-2} \nabla \hat{u}(x)-|\nabla \hat{u}(x)-p|^{n-2}(\nabla \hat{u}(x)-p).
\end{equation}

    Since \eqref{2-8} holds in $A_\epsilon$, \eqref{2-9} can be re-written as
 \begin{equation}\label{2-9}
\begin{cases}\Delta_n\left(\hat{u}-U_\epsilon\right)-\Delta_n \hat{u}=\frac{e^{\hat{u} }}{|x|^{2 n}} & \text { in } A_\epsilon, \\ U_\epsilon=0 &\text { on }  \partial (B_r(0) \backslash B_\epsilon(0))\cap \mathbb{R}^{n}_{+},\\ \mathbf{a}(x,U_{\epsilon}) =  \frac{e^{\frac{n-1}{n}\hat{u}}}{|x|^{2(n-1)}} &\text { on }A_{\epsilon}\cap \partial\mathbb{R}^{n}_{+}.\end{cases}
\end{equation}
Since
 \begin{equation*}
d=\inf _{v \neq w} \frac{\left.\left.\langle | v\right|^{n-2} v-|w|^{n-2} w, v-w\right\rangle}{|v-w|^n}>0,
\end{equation*}
we can apply Lemma \ref{pro2-1} to $\mathbf{a}(x, p)$ given by \eqref{re2-8}. Since $\left|A_\epsilon\right| \leq \omega_n r^n$ and $\mathbf{a}(x, 0)=0$, we deduce that
 \begin{equation}\label{2-11}
\int_{A_\epsilon}\left|\nabla U_\epsilon\right|^q+\int_{A_\epsilon} e^{p U_\epsilon}  + \int_{A_\epsilon\cap \partial\mathbb{R}^{n}_{+}} e^{p U_\epsilon}\leq C
\end{equation}
for all $1 \leq q<n$ and $1\leq p < p_r$ with $ r$ sufficiently small, where $C$ is uniform in $\epsilon$ and $p_r\to+\infty$ as $ r\to 0$. Notice that
 \begin{equation*}
\Bigg|\int_{B_r(0)\cap\mathbb{R}^{n}_{+}} \frac{e^{\hat{u} }}{|x|^{2 n}}\Bigg|=  \Bigg|\int_{\mathbb{R}^n_{+} \backslash B_{\frac{1}{r}}(0)} e^u\Bigg|\rightarrow 0,
\end{equation*}
and
 \begin{equation*}
\int_{B_r(0)\cap\partial\mathbb{R}^{n}_{+}} \frac{e^{\frac{n-1}{n}\hat{u}}}{|x|^{2 (n-1)}}=\int_{\partial\mathbb{R}^n_{+} \backslash B_{\frac{1}{r}}(0)} e^{\frac{n-1}{n}u} \rightarrow 0,
\end{equation*}
as $r \rightarrow 0$. By the Sobolev embedding $\mathcal{D}^{1, \frac{n}{2}}\left(\mathbb{R}^n_{+}\right) \hookrightarrow L^n\left(\mathbb{R}^n_{+}\right)$, estimate \eqref{2-11} yields that
 \begin{equation}\label{2-14}
\int_{A_\epsilon}\left|U_\epsilon\right|^n \leq C
\end{equation}
for some $C$ that is uniform in $\epsilon$. Since $H_\epsilon=\hat{u}-U_\epsilon$ with $\hat{u} \in C^{1, \alpha}\left(\overline{\mathbb{R}^{n}_{+}} \backslash\{0\}\right)$, by \eqref{2-14} we deduce that
 \begin{equation*}
\left\|H_\epsilon\right\|_{L^n(A)} \leq C(A), \quad \forall A \subset \subset \overline{B_r(0)} \backslash\{0\}
\end{equation*}
for all $\epsilon$ sufficiently small.
From \cite[Theorem 10.3]{GT}, \cite[Theorem 2]{S}, \cite[Theorem 1, Theorem 2]{D} and \cite[Theorem 3.1]{L2}, we can get that
 \begin{equation*}
\left\|H_{\epsilon}\right\|_{C^{1, \alpha}(\overline{A})} \leq C(A), \quad \forall A \subset \subset B_r(0) \backslash\{0\},
\end{equation*}
and
\begin{equation*}
\left\|H_{\epsilon}\right\|_{C^{0, \alpha}(\overline{B})} \leq C\left(B,\sup_{B_{r}(0)\cap \mathbb{R}^{n}_{+}}|\hat{u}|\right)  , \quad \forall B \subset \subset \overline{B_r(0)} \backslash\{0\}
\end{equation*}
for $\epsilon$ small.
 By the Ascoli-Arzel's Theorem and a diagonal process, we can find a sequence $\epsilon \rightarrow 0$ so that $H_\epsilon \rightarrow H_0$ in $C_{\text {loc }}^1\left(B_r(0) \backslash\{0\}\right)$ and $H_\epsilon \rightarrow H_0$ in $C_{\text {loc }}^{0}\left(\overline{B_r(0)} \backslash\{0\}\right)$, where $H_0$ satisfies
 \begin{equation*}
\begin{cases}\Delta_n H_0=0 & \text { in } B_r(0) \backslash\{0\}, \\ H_0=\hat{u} & \text { on } \partial B_r(0)\cap\partial\mathbb{R}^{n}_{+},
\\ H_0(x,x_n)=\hat{u}(x',-x_n) & \text { on } \partial B_r(0)\cap\partial\mathbb{R}^{n}_{-},
\\|\nabla H_{0}|^{n-2}\frac{\partial H_{0}}{\partial \nu} =0 & \text { on }  B_r(0)\cap\partial\mathbb{R}^{n}_{+}.
\end{cases}
\end{equation*}
Since $H_\epsilon \leq \hat{u}$ in $A_\epsilon$, by Lemma \ref{le:2.5}, we have that $U_\epsilon \rightarrow U_0:=\hat{u}-H_0$ in $C_{\operatorname{loc}}^1\left(\overline{B_r(0)} \backslash\{0\}\right)$, where $U_0$ satisfies
$$
U_0 \geq 0 \qquad \text { in } B_r(0) \backslash\{0\}\cap \mathbb{R}^{n}_{+}.
$$
Moreover, from \eqref{2-11}, we get that
\begin{equation}\label{re2-9}
    U_0 \in W_0^{1, q}\left(B_r(0)\cap \mathbb{R}^{n}_{+}\right), \quad e^{U_0} \in L^p\left(B_r(0)\cap \mathbb{R}^{n}_{+}\right)
\end{equation}
for all $1 \leq q<n$ and $1\leq p < p_r$ with $ r$ sufficiently small.
 Since $H_{0}$ is a continuous $n$-harmonic function in $B_{r}(0)\setminus\{0\}$ with
$$
\tilde{H}_{0} \leq \sup _{\mathbb{R}^{n}_{+} \backslash\{0\}} \hat{u}=\sup _{\mathbb{R}^{n}_{+}} u<\infty
$$
in view of Proposition \ref{prop:3.1},
applying \cite[Theorem 12]{S} about isolated singularities, we get taht either $H_{0}$ has a removable singularity at $0$ or
$$
\frac{1}{C} \leq \frac{H_{0}(x)}{\ln |x|} \leq C
$$
near 0 for some $C>1$. According to \cite{S2}, in both situations, we have that
\begin{equation}\label{re2-10}
    H_{0} \in W^{1, q}\left(B_r(0)\cap \mathbb{R}^{n}_{+}\right)
\end{equation}
for all $1 \leq q<n$. The combination of \eqref{re2-9} and \eqref{re2-10} establishes the validity of \eqref{re2-11} for $\hat{u}=U_0+H_0$.
\end{proof}

\begin{prop}\label{th3-2}
  Let $u\in W^{1,n}_{\text{loc}}( \mathbb{R}^{n}_{+})$ be the solution of
    \begin{equation}\label{reeq:3.16}
    \begin{cases}
        -\Delta_n u = e^u&\text{ in } \mathbb{R}^{n}_{+},\\
        |\nabla u|^{n-2}\nabla u\cdot\nu =e^{\frac{n-1}{n}u} &\text{ on }\partial \mathbb{R}^{n}_{+},\\
        \displaystyle\int_{\mathbb{R}^{n}_{+} }e^{u}\dx <+\infty\text{ and } \text{ }\int_{\partial\mathbb{R}^{n}_{+} }e^{\frac{n-1}{n}u}\dx <+\infty,
    \end{cases}
    \end{equation}
     and $\hat{u}(x)=u\left(\frac{x}{|x|^2}\right)$. Then $u(x)$ satisfies
\begin{equation*}
    u(x)+\left(\frac{\gamma_0}{n \omega_n}\right)^{\frac{1}{n-1}} \ln |x| \in L_{l o c}^{\infty}\left(\mathbb{R}^n\right),
\end{equation*}
and
\begin{equation*}
    \sup _{|x|=r}|x|\left|\nabla\left(u(x)+\left(\frac{\gamma_0}{n \omega_n}\right)^{\frac{1}{n-1}} \ln |x|\right)\right| \rightarrow 0
\end{equation*}
for a sequence $r \rightarrow +\infty$, where $\gamma_0=2\left(\displaystyle\int_{\mathbb{R}^{n}_{+}} e^u\dx +\displaystyle\int_{\partial\mathbb{R}^{n}_{+}} e^{\frac{n-1}{n}u}\textup{d}\sigma\right)$ and $\left(\frac{\gamma_0}{n \omega_n}\right)^{\frac{1}{n-1}}>n$.
\end{prop}
\begin{proof}
    Given $r>0$ small enough, by Proposition \ref{reprop:3.1}, $\hat{u}$ has been decomposed in $B_r(0)$ as $\hat{u}=U_0+H_0$ with $U_0, H_0 \in C_{\operatorname{loc}}^1\left(B_r(0) \backslash\{0\}\right)\cap C_{\operatorname{loc}}^0\left(\overline{B_r(0)} \backslash\{0\}\right)$, where $H_0$ is a $n$-harmonic function in $B_r(0) \backslash\{0\}$ with $\sup\limits_{B_r(0) \backslash\{0\}} H_0<+\infty$ and $U_0 \geq 0$ satisfies \eqref{re2-9} and $U_0=0$ on $\partial B_r(0)$.

By the asymptotic behavior of $H_{0}$ at $0$, which was established in \cite{S,S2} and improved in Theorem 1 of \cite{KV}, we have, there exists $\gamma \geq 0$ such that
\begin{equation}\label{3-12}
H_{0}(x)-\left(\frac{\gamma}{n \omega_n}\right)^{\frac{1}{n-1}} \ln |x| \in L^{\infty}\left(B_r(0)\right), \quad \Delta_n H_{0}=\gamma \delta_0 \text { in } \mathcal{D}^{\prime}\left(B_r(0)\right).
\end{equation}
Since $\hat{u}(x) \geq  H_{0}(x)$ and $\displaystyle\int_{\partial\mathbb{R}^{n}_{+}} e^{\frac{n-1}{n}u}\textup{d}\sigma<+\infty $, we get that $ \left(\frac{\gamma}{n \omega_n}\right)^{\frac{1}{n-1}} > n$. Thus by \eqref{3-12}, we have that, for sufficiently small $r$,
\begin{equation*}
\frac{e^{H_0}}{|x|^{2 n}} \in L^q\left(B_r(0)\right),\qquad \hbox{ }\frac{e^{\frac{n-1}{n}H_0}}{|x|^{2 (n-1)}} \in L^q\left(B_r(0)\cap\partial\mathbb{R}^{n}_{+}\right)
\end{equation*}
for all $1 \leq q<\frac{n}{2n-\left(\frac{\gamma}{n \omega_n}\right)^{\frac{1}{n-1}}}$ if $\left(\frac{\gamma}{n \omega_n}\right)^{\frac{1}{n-1}}<2n$ and $1 \leq q<+\infty$ if $\left(\frac{\gamma}{n \omega_n}\right)^{\frac{1}{n-1}}\geq 2n$. It follows from \eqref{re2-9} that
\begin{equation}\label{3-14}
\bigg|\frac{e^{\hat{u} }}{ |x|^{2 n} }\bigg|\leq \frac{e^{\hat{u}}}{|x|^{2 n}}=e^{U_0} \frac{e^{H_0}}{|x|^{2 n}} \in L^q\left(B_r(0)\right),\hbox{ }\frac{e^{\frac{n-1}{n}\hat{u}}}{|x|^{2 (n-1)}}=e^{U_0} \frac{e^{\frac{n-1}{n}H_0}}{|x|^{2 (n-1)}} \in L^q\left(B_r(0)\cap\partial\mathbb{R}^{n}_{+}\right)
\end{equation}
for all $1 \leq q<\frac{n-1}{n-2}$ if $r>0$ is sufficiently small. Thanks to \eqref{3-14}, we can apply Proposition \ref{prop3-1} to $U_\epsilon$ on $A_\epsilon$ (see \eqref{2-8}-\eqref{mo2-8}) with $\mathbf{a}(x, p)$ given by \eqref{re2-8} to get
\begin{equation*}
\left\|U_\epsilon\right\|_{\infty, A_\epsilon} \leq C
\end{equation*}
for some uniform constant $C>0$, here we have used that
$$
\sup _\epsilon\left\|U_\epsilon\right\|_{p, A_\epsilon}<+\infty
$$
for all $p \geq 1$ in view of \eqref{2-11} and the Sobolev embedding theorem. Letting $\epsilon \rightarrow 0$, we get that $\left\|U_0\right\|_{\infty, B_r(0)}<+\infty$ and then
\begin{equation*}
\hat{u}=U_0+H_0=\left(\frac{\gamma}{n \omega_n}\right)^{\frac{1}{n-1}} \ln |x|+H(x), \quad H \in L_{\operatorname{loc}}^{\infty}\left(\mathbb{R}^{n}_{+}\right)
\end{equation*}
in view of \eqref{3-12}. That is,
\begin{equation}\label{3-17}
    u(x)+\left(\frac{\gamma}{n \omega_n}\right)^{\frac{1}{n-1}} \ln |x| \in L_{l o c}^{\infty}\left(\mathbb{R}^n\right).
\end{equation}
Now, letting $G(x):=u(x)+\left(\frac{\gamma}{n \omega_n}\right)^{\frac{1}{n-1}}\log |x|$, then \eqref{3-17} implies that $G(x)\in L^{\infty}(\overline{\mathbb{R}^{n}_{+}}\setminus B_{1}(0))$, and
 \begin{equation}\label{prop:3.25}
    u(x)=-\left(\frac{\gamma}{n \omega_n}\right)^{\frac{1}{n-1}}\log |x|+G(x)\ \quad\,\,\, \text{for}\,\,\, x\in\mathbb{R}^{n}_{+}\setminus B_{1}(0),
 \end{equation}
where $n<\left(\frac{\gamma}{n \omega_n}\right)^{\frac{1}{n-1}}$.

\medskip

For any sequence $R_k>0$ with $R_k\rightarrow+\infty$ as $k\rightarrow+\infty$, let
 \begin{equation*}
    \widehat{u}_{R_k}(x):=u(R_kx)+\left(\frac{\gamma}{n \omega_n}\right)^{\frac{1}{n-1}}\log R_k.
 \end{equation*}
 By \eqref{prop:3.25}, we know that
 \begin{equation}\label{eq:3.26}
    \widehat{u}_{R_k}(x)=-\left(\frac{\gamma}{n \omega_n}\right)^{\frac{1}{n-1}}\log|x|+G_{R_k}(x),
 \end{equation}
 where $G_{R_k}(x)=G(R_kx)$. In view of \eqref{reeq:3.16}, we obtain that
 \begin{equation*}
    \left\{
        \begin{aligned}
        & |-\Delta_{n} \widehat{u}_{R_k}|=|R_{k}^{n}f(\widehat{u}_{R_k} )|\leq R_{k}^{n-\left(\frac{\gamma}{n \omega_n}\right)^{\frac{1}{n-1}}}\frac{e^{G_{R_k}}}{|x|^{\left(\frac{\gamma}{n \omega_n}\right)^{\frac{1}{n-1}}}} \qquad \ &{\rm{in}}\,\, \mathbb{R}^{n}_{+}\setminus B_{\frac{1}{R_k}}(0),\\
        &a(\nabla  \widehat{u}_{R_k})\cdot \nu =R_{k}^{n-1-\left(\frac{\gamma}{n \omega_n}\right)^{\frac{1}{n}}}\frac{e^{\frac{n-1}{n}G_{R_k}}}{|x|^{\left(\frac{\gamma}{n \omega_n}\right)^{\frac{1}{n}}}} \qquad &{\rm{on}}\,\, \partial\mathbb{R}^{n}_{+}\setminus B_{\frac{1}{R_k}}(0).
        \end{aligned}
        \right.
 \end{equation*}
 Since $n<\left(\frac{\gamma}{n \omega_n}\right)^{\frac{1}{n-1}}$, we get that $ \widehat{u}_{R_k}$, $R_{k}^{n-\left(\frac{\gamma}{n \omega_n}\right)^{\frac{1}{n-1}}}\frac{e^{G_{R_k}}}{|x|^{\left(\frac{\gamma}{n \omega_n}\right)^{\frac{1}{n-1}}}}$ and $R_{k}^{n-1-\left(\frac{\gamma}{n \omega_n}\right)^{\frac{1}{n}}}\frac{e^{\frac{n-1}{n}G_{R_k}}}{|x|^{\left(\frac{\gamma}{n \omega_n}\right)^{\frac{1}{n}}}} $ are bounded in $L^{\infty}_{loc}(\overline{\mathbb{R}^{n}_{+}}\setminus \{0\})$, uniformly in $k$. Thus
 by the interior regularity estimates in \cite{D,S2,T} and the boundary regularity estimates in \cite{L}, we obtain that $ \widehat{u}_{R_k}$ is uniformly bounded in $C^{1,\theta}_{loc}(\overline{\mathbb{R}^{n}_{+}}\setminus \{0\})$.
 Then by the Ascoli--Arzel\'a's theorem and a diagonal argument, we can find a subsequence $R_{k_i}\rightarrow+\infty$ so that $ \widehat{u}_{R_{k_i}}\rightarrow  \widehat{u}_{\infty}$ in $C^{1,\theta}_{loc}(\overline{\mathbb{R}^{n}_{+}}\setminus \{0\})$ as $i\rightarrow+\infty$. By $\left(\frac{\gamma}{n \omega_n}\right)^{\frac{1}{n-1}}>n$, we deduce that $ \widehat{u}_{\infty}$ satisfies
 \begin{equation}\label{eq:3.27}
    \left\{
        \begin{aligned}
        & -\Delta_{n} \widehat{u}_{\infty}=0\qquad \ &{\rm{in}}\,\, \mathbb{R}^{n}_{+},\\
        &a(\nabla  \widehat{u}_{\infty})\cdot \nu =0\qquad &{\rm{on}}\,\, \partial\mathbb{R}^{n}_{+}.
        \end{aligned}
        \right.
 \end{equation}
 In view of \eqref{eq:3.26}, we have
 \begin{equation}\label{eq:3.28}
    \widehat{u}_{\infty}(x)=-\left(\frac{\gamma}{n \omega_n}\right)^{\frac{1}{n-1}}\log |x|+G_{\infty}(x),
 \end{equation}
 where $G_{\infty}$ is the limit of $G_{R_{k_i}}$ in $C^{1}_{loc}(\overline{\mathbb{R}^{n}_{+}}\setminus \{0\})$. Moreover, the uniform boundedness of $G_{R_k}$ implies $G_{\infty}\in L^{\infty}(\overline{\mathbb{R}^{n}_{+}}\setminus \{0\})$. By \eqref{eq:3.27}, \eqref{eq:3.28} and $G_{\infty}\in L^{\infty}(\overline{\mathbb{R}^{n}_{+}}\setminus \{0\})$, we can apply Lemma \ref{le:2.8} to $\widehat{u}_{\infty}$ and derive that $G_{\infty}$ is a constant function.
 By \eqref{prop:3.25} and $G_{R_{k_i}}\rightarrow G_{\infty}$ in $C^{1}_{loc}(\overline{\mathbb{R}^{n}_{+}}\setminus \{0\})$, we know that
 \begin{equation}\label{3-18}
 \begin{split}
        &\sup_{|x|=R_{k_i},\,x\in\mathbb{R}^{n}_{+}}|x|\left\lvert \nabla \left(u(x)+\left(\frac{\gamma}{n \omega_n}\right)^{\frac{1}{n-1}}\log|x|\right)\right\rvert\\
    =&\sup_{|y|=1,\,y\in\mathbb{R}^{n}_{+}}\left\lvert \nabla G_{R_{k_i}}(y)\right\rvert\rightarrow \sup_{|y|=1,\,y\in\mathbb{R}^{n}_{+}}\left\lvert \nabla G_{\infty}(y)\right\rvert=0.
 \end{split}
 \end{equation}

 In view of $u\in C^1\left(\overline{B_{R}(0)\cap\mathbb{R}^{n}_{+}}\right) $, by choosing $\Omega=B_{R}(0)\cap\mathbb{R}^{n}_{+}$, $F(x)=e^{u}$
   and ${\bf{a}} (x)=-|\nabla u|^{n-2}\nabla u$ in Lemma \ref{le:2.3}, we obtain that
   \begin{equation*}
   \begin{split}
        &\int_{\partial B_{R}(0)\cap\mathbb{R}^{n}_{+}}\left\langle |\nabla u|^{n-2}\nabla u,-\nu\right\rangle \mathrm{d}\sigma+
    \int_{\partial \mathbb{R}^{n}_{+}\cap B_{R}(0)}\left\langle |\nabla u|^{n-2}\nabla u,-\nu\right\rangle \mathrm{d}\sigma
    \\=  &\int_{B_{R}(0)\cap\mathbb{R}^{n}_{+}}e^{u} \mathrm{d}x.
   \end{split}
   \end{equation*}
   Since $\left\langle |\nabla u|^{n-2}\nabla u,\nu\right\rangle=e^{\frac{n-1}{n}u}$ on $\partial   \mathbb{R}^{n}_{+}$, we deduce that
\begin{equation}\label{eq:5.5}
    \int_{   \mathbb{R}^{n}_{+}\cap \partial B_{R}(0)}\left\langle a(\nabla u),-\nu\right\rangle \mathrm{d}\sigma=
    \int_{   \mathbb{R}^{n}_{+}\cap B_{R}(0)}e^{u} \mathrm{d}x+\int_{   \partial\mathbb{R}^{n}_{+}\cap B_{R}(0)}e^{\frac{n-1}{n}u} \mathrm{d}x.
\end{equation}

By \eqref{3-18}, we obtain that
\begin{equation*}
    \nabla u=-\left(\frac{\gamma}{n \omega_n}\right)^{\frac{1}{n-1}} \frac{x}{|x|^2}+o\left(\frac{1}{|x|}\right)
\end{equation*}
uniformly for $x\in\partial B_{R}(0)$, as $R\rightarrow+\infty$. By \eqref{eq:5.5}, we deduce that
\begin{equation*}
    \begin{aligned}
        \int_{   \mathbb{R}^{n}_{+}\cap B_{R}(0)}e^{u} \mathrm{d}x+\int_{   \partial\mathbb{R}^{n}_{+}\cap B_{R}(0)}e^{\frac{n-1}{n}u} \mathrm{d}x
            &=\int_{   \mathbb{R}^{n}_{+}\cap \partial B_{R}(0)}|\nabla u|^{n-2}\left\langle \nabla u,-\nu\right\rangle \mathrm{d}\sigma\\
            &=\frac{\gamma}{n \omega_n}\int_{   \mathbb{R}^{n}_{+}\cap \partial B_{R}(0)} \bigg(\frac{1}{|x|^{n-1}}+o\left(\frac{1}{|x|^{n-1}}\right)\bigg) \mathrm{d}\sigma.
    \end{aligned}
\end{equation*}
Let $R\rightarrow+\infty$, we obtain that
\begin{equation*}
    \int_{   \mathbb{R}^{n}_{+}}e^{u} \mathrm{d}x+\int_{   \partial\mathbb{R}^{n}_{+}}e^{\frac{n-1}{n}u} \mathrm{d}x=\frac{\gamma}{2}.
\end{equation*}
Thus we have finished our proof.
\end{proof}

Let $u$ be a solution of \eqref{eq:1-2}, by the Pohozaev type identity and Proposition \ref{th3-2}, we will prove the following quantization result on the total mass $\displaystyle\int_{\mathbb{R}^{n}_{+}} e^u\dx +\displaystyle\int_{\partial\mathbb{R}^{n}_{+}} e^{\frac{n-1}{n}u}\textup{d}\sigma$.
\begin{prop}\label{Prop:5.3}
    Let $u$ be a solution of \eqref{eq:1-2}, then
    \begin{equation*}
        2\left(\displaystyle\int_{\mathbb{R}^{n}_{+}} e^u\dx +\displaystyle\int_{\partial\mathbb{R}^{n}_{+}} e^{\frac{n-1}{n}u}\textup{d}\sigma\right)= n \omega_n\left(\frac{n^2}{n-1}\right)^{n-1}.
    \end{equation*}
\end{prop}
\begin{proof}
Let $\Omega\subset\mathbb{R}^{n}$ be a bounded open set with Lipschitz boundary and $u\in C^{1}(\overline{\Omega})$ solves $-\Delta_{n}u=f$ in $\Omega$, where $f\in L^{1}(\Omega)\cap L^{\infty}_{loc}(\Omega)$, then the following Pohozaev identity (c.f. \cite[Lemma 4.2]{CL2} in the case $p=n$, see also \cite[Lemma 3.3]{E}) holds for any $y\in\mathbb{R}^{n}$:
\begin{equation}\label{Poho}
\begin{aligned}
& -\int_{\Omega} f(x)\langle x-y, \nabla u\rangle \mathrm{d}x \\
= & \int_{\partial \Omega}\left[|\nabla u|^{n-2}\langle\nabla u, \nu\rangle\langle x-y, \nabla u\rangle-\frac{|\nabla u|^n}{n}\langle x-y, \nu\rangle\right] \mathrm{d}\sigma,
\end{aligned}
\end{equation}
where $\nu$ is the unit outer normal vector to $\partial\Omega$. For any $R>0$, $u\in C^1\left(\overline{\mathbb{R}^{n}_{+}\cap B_R( 0)}\right) $, we can choose $\Omega=\mathbb{R}^{n}_{+}\cap B_R( 0)$, $y=0$ and $f=e^u$ in the Pohozaev identity \eqref{Poho} and obtain that
     \begin{equation}\label{eq:5.18}
        \begin{aligned}
            &\quad-\int_{\mathbb{R}^{n}_{+}\cap B_R( 0)}e^u\left\langle x ,\nabla u\right\rangle \mathrm{d}x\\
            &=\int_{\partial(\mathbb{R}^{n}_{+}\cap B_R( 0))} \left[|\nabla u|^{n-2}\langle\nabla u, \nu\rangle\langle x-y, \nabla u\rangle-\frac{|\nabla u|^n}{n}\langle x-y, \nu\rangle\right]\mathrm{d}\sigma.
        \end{aligned}
     \end{equation}
By the divergence theorem, we have
\begin{equation*}
        -\int_{\mathbb{R}^{n}_{+}\cap B_R( 0)}e^u\left\langle x ,\nabla u\right\rangle \mathrm{d}x
        =n\int_{\mathbb{R}^{n}_{+}\cap B_R( 0)}e^u \mathrm{d}x
        -\int_{\partial(\mathbb{R}^{n}_{+}\cap B_R( 0))}e^u\left\langle x ,\nu\right\rangle \mathrm{d}\sigma.
\end{equation*}
Noting that
\begin{equation*}
    \begin{aligned}
        \int_{\partial(\mathbb{R}^{n}_{+}\cap B_R( 0))}e^u\left\langle x ,\nu\right\rangle \mathrm{d}\sigma
        &=\int_{\partial \mathbb{R}^{n}_{+}\cap B_R( 0)}e^u\left\langle x ,\nu\right\rangle \mathrm{d}\sigma\\
        &\quad+\int_{\mathbb{R}^{n}_{+}\cap \partial B_R( 0)}e^u\left\langle x ,\nu\right\rangle \mathrm{d}\sigma,
    \end{aligned}
\end{equation*}
thus we obtain that
\begin{equation*}
    \begin{aligned}
        -\int_{\mathbb{R}^{n}_{+}\cap B_R( 0)}e^u\left\langle x ,\nabla u\right\rangle \mathrm{d}x
        &=n\int_{\mathbb{R}^{n}_{+}\cap B_R( 0)}e^u \mathrm{d}x
        -\int_{\partial\mathbb{R}^{n}_{+}\cap B_R( 0)}e^u\left\langle x ,\nu\right\rangle \mathrm{d}\sigma\\
        &\quad-\int_{\mathbb{R}^{n}_{+}\cap \partial B_R( 0)}e^u\left\langle x ,\nu\right\rangle \mathrm{d}\sigma.
    \end{aligned}
\end{equation*}
Since $\left\langle x ,\nu\right\rangle =0$ for any $x\in\partial\mathbb{R}^{n}_{+}$, we derive that
\begin{equation}\label{eq:5.20}
    -\int_{\mathbb{R}^{n}_{+}\cap B_R( 0)}e^u\left\langle x ,\nabla u\right\rangle \mathrm{d}x
    =n\int_{\mathbb{R}^{n}_{+}\cap B_R( 0)}e^u \mathrm{d}x
    -\int_{\mathbb{R}^{n}_{+}\cap \partial B_R( 0)}Re^u\mathrm{d}\sigma.
\end{equation}
Also by the divergence theorem, we have
\begin{equation}\label{reeq:5.1}
    \begin{split}
        \int_{\partial\mathbb{R}^{n}_{+}\cap  B_R( 0)}e^{\frac{n-1}{n}u}\langle x, \nabla u\rangle\mathrm{d}\sigma=&\int_{\mathbb{R}^{n-1}\cap  B_R( 0)}e^{\frac{n-1}{n}u}\sum_{i=1}^{n-1}x_i u_{x_i}\mathrm{d}x
        \\=&\int_{\mathbb{R}^{n-1}\cap \partial B_R( 0)}e^{\frac{n-1}{n}u}R\mathrm{d}\sigma-\int_{\mathbb{R}^{n-1}\cap  B_R( 0)}ne^{\frac{n-1}{n}u}\mathrm{d}x.
    \end{split}
\end{equation}
From \eqref{eq:5.18}, we deduce that
\begin{equation}\label{eq:5.21}
    \begin{aligned}
        &\quad-\int_{\mathbb{R}^{n}_{+}\cap B_R( 0)}e^u\left\langle x ,\nabla u\right\rangle \mathrm{d}x-\int_{\partial\mathbb{R}^{n}_{+}\cap  B_R( 0)}e^{\frac{n-1}{n}u}\langle x, \nabla u\rangle\mathrm{d}\sigma \\
        &=\int_{\mathbb{R}^{n}_{+}\cap \partial B_R( 0)} \left[|\nabla u|^{n-2}\langle\nabla u, \nu\rangle\langle x, \nabla u\rangle-\frac{|\nabla u|^n}{n}\langle x, \nu\rangle \right]\mathrm{d}\sigma,
    \end{aligned}
\end{equation}
where we have used the fact that $|\nabla u|^{n-2}\langle\nabla u, \nu\rangle=e^{\frac{n-1}{n}u}$ on $\partial\mathbb{R}^{n}_{+}$ and $\left\langle x ,\nu\right\rangle =0$  for any $x\in\partial\mathbb{R}^{n}_{+}$. By combining \eqref{eq:5.20}, \eqref{reeq:5.1} with \eqref{eq:5.21}, we arrive at
\begin{equation}\label{eq:5.22}
    \begin{aligned}
        &\quad n\int_{\mathbb{R}^{n}_{+}\cap B_R( 0)}e^u \mathrm{d}x
        -\int_{\mathbb{R}^{n}_{+}\cap \partial B_R( 0)}e^uR \mathrm{d}\sigma
        \\&-\int_{\mathbb{R}^{n-1}\cap \partial B_R( 0)}e^{\frac{n-1}{n}u}R\mathrm{d}\sigma+\int_{\mathbb{R}^{n-1}\cap  B_R( 0)}ne^{\frac{n-1}{n}u}\mathrm{d}x\\
        & =\int_{\mathbb{R}^{n}_{+}\cap \partial B_R( 0)} \left|\nabla u|^{n-2}\langle\nabla u, \nu\rangle\langle x, \nabla u\rangle-\frac{|\nabla u|^n}{n}\langle x, \nu\rangle \right]\mathrm{d}\sigma.
    \end{aligned}
\end{equation}
 Since $\nabla u=-\left(\frac{\gamma_0}{n \omega_n}\right)^{\frac{1}{n-1}}\frac{x}{R|x|}+o(\frac{1}{R})$ uniformly for $x\in\partial B_R( 0)$, as $R\rightarrow+\infty$, we have
  \begin{equation*}
    \begin{aligned}
        \left\langle x,\nabla u\right\rangle
        =  \left\langle x,-\left(\frac{\gamma_0}{n \omega_n}\right)^{\frac{1}{n-1}}\frac{x}{R|x|}+o\left(\frac{1}{R}\right)\right\rangle
        =-\left(\frac{\gamma_0}{n \omega_n}\right)^{\frac{1}{n-1}}+o_{R}(1)
    \end{aligned}
  \end{equation*}
  uniformly for $x\in\partial B_R( 0)$, as $R\rightarrow+\infty$. Moreover, since $|\nabla u|=\frac{\left(\frac{\gamma_0}{n \omega_n}\right)^{\frac{1}{n-1}}}{R}+o(\frac{1}{R})$
  uniformly for $x\in\partial B_R( 0)$, as $R\rightarrow+\infty$, we have
  \begin{equation}\label{eq:5.23}
    \begin{aligned}
        &\quad\int_{\mathbb{R}^{n}_{+}\cap \partial B_R( 0)} |\nabla u|^{n-2}
        \left\langle\nabla u,\nu\right\rangle
        \left\langle x ,\nabla u\right\rangle \mathrm{d}\sigma\\
        &=\int_{\mathbb{R}^{n}_{+}\cap \partial B_R( 0)}\frac{1}{R}\left(\frac{\left(\frac{\gamma_0}{n \omega_n}\right)^{\frac{1}{n-1}}}{R}+o\left(\frac{1}{R}\right)\right)^{n-2}
         \left(-\left(\frac{\gamma_0}{n \omega_n}\right)^{\frac{1}{n-1}}+o_{R}(1)\right)^2\mathrm{d}\sigma \\
         &=n\left(\frac{\gamma_0}{n \omega_n}\right)^{\frac{n}{n-1}}\frac{\omega_n}{2}+o_{R}(1)
    \end{aligned}
  \end{equation}
  uniformly for $x\in\partial B_R( 0)$, as $R\rightarrow+\infty$. Moreover, we have
   \begin{equation}\label{eq:5.24}
    \begin{aligned}
        \quad\int_{\mathbb{R}^{n}_{+}\cap \partial B_R( 0)} \frac{|\nabla u|^n}{n}\langle x, \nu\rangle \mathrm{d}\sigma
        =\bigg(\big(\frac{\gamma_0}{n \omega_n}\big)^{\frac{n}{n-1}}+o_{R}(1)\bigg)\frac{\omega_n}{2}
    \end{aligned}
   \end{equation}
   uniformly for $x\in\partial B_R( 0)$, as $R\rightarrow+\infty$. By Proposition \ref{th3-2}, we obtain that
   \begin{equation*}
    e^{u(x)}\leq C|x|^{-\left(\frac{\gamma_0}{n \omega_n}\right)^{\frac{1}{n-1}}} \quad \hbox{ and } \quad e^{\frac{n-1}{n}u(x)}\leq C|x|^{-\frac{n-1}{n}\left(\frac{\gamma_0}{n \omega_n}\right)^{\frac{1}{n-1}}}
   \end{equation*}
   for some constant $C>0$, as $|x|\rightarrow+\infty$ with $x\in\mathbb{R}^{n}_{+}$. Finally, it follows from Proposition \ref{th3-2} that
   $$\left(\frac{\gamma_0}{n \omega_n}\right)^{\frac{1}{n-1}}>n.$$
   Thus
   \begin{equation}\label{ineq:5.25}
    \begin{aligned}
        &\int_{\mathbb{R}^{n}_{+}\cap \partial B_R( 0)}e^uR\mathrm{d}\sigma+\int_{\mathbb{R}^{n-1}\cap \partial B_R( 0)}e^{\frac{n-1}{n}u}R\mathrm{d}\sigma\\
        \leq &C \left(R^{n-\left(\frac{\gamma_0}{n \omega_n}\right)^{\frac{1}{n-1}}}+R^{n-1-\frac{n-1}{n}\left(\frac{\gamma_0}{n \omega_n}\right)^{\frac{1}{n-1}}}\right)\rightarrow0,\qquad\,\,\,\,\,\, \text{as}\ R\rightarrow+\infty.
    \end{aligned}
   \end{equation}
   By letting $R\rightarrow+\infty$ in \eqref{eq:5.22}, combining \eqref{eq:5.23} with \eqref{eq:5.24} and \eqref{ineq:5.25},
   we deduce that
   \begin{equation*}
    n\left(\int_{\mathbb{R}^{n}_{+}}e^u \mathrm{d}x +\int_{\partial\mathbb{R}^{n}_{+}}e^{\frac{n-1}{n}u} \mathrm{d}\sigma\right)=(n-1)\left(\frac{\gamma_0}{n \omega_n}\right)^{\frac{n}{n-1}}\frac{\omega_n}{2},
   \end{equation*}
   this concludes our proof of Proposition \ref{Prop:5.3}.
\end{proof}

Using Proposition \ref{prop:3.1} and Proposition \ref{Prop:5.3}, we have the following corollary.
\begin{cor}\label{cor3-1}
 Let $u$ be a solution of \eqref{eq:1-2}. Then $u(x)$ satisfies
\begin{equation*}
    u(x)+\frac{n^2}{n-1} \ln |x| \in L_{l o c}^{\infty}\left(\mathbb{R}^n\right),
\end{equation*}
and
\begin{equation*}
    \sup _{|x|=r}|x|\left|\nabla\left(u(x)+\frac{n^2}{n-1} \ln |x|\right)\right| \rightarrow 0,
\end{equation*}
for a sequence $r \rightarrow +\infty$.
\end{cor}

\section{Second-order regularity for weak solutions}
In this section, we will prove the second-order regularity for weak solutions via an approximation method.
\begin{prop}\label{newSECONprop1}
    Let $u$ be a solution to \eqref{eq:1-2}. Then $|\nabla u|^{n-2} \nabla u \in W_{\text {loc }}^{1,2}(\overline{\mathbb{R}^n_{+}})$, and for any $\gamma \in \mathbb{R}$ the following asymptotic estimate holds:
$$
\begin{array}{ll}
\dis\int_{B_R \cap \mathbb{R}^n_{+}}\left|\nabla\left(|\nabla u|^{n-2} \nabla u\right)\right|^2 e^{\gamma u} \textup{d}x \leq C\left(1+R^{-n-\gamma \frac{n^2}{n-1}}\right), & \forall R>1, \\
\dis\int_{\mathbb{R}^n_{+} \cap\left(B_{2 R} \backslash B_R\right)}\left|\nabla\left(|\nabla u|^{n-2} \nabla u\right)\right|^2 e^{\gamma u} \textup{d}x \leq C R^{-n-\gamma \frac{n^2}{n-1}}, & \forall R>1,
\end{array}
$$
where $C$ is a positive constant independent of $R$.
\end{prop}

\begin{proof}
    This proposition is obtained by using a Caccioppoli-type inequality. We argue by approximation, following the approach in \cite{AKM,CV,CFR}.
We set
$$
\begin{aligned}
a_i(x) & :=|x|^{n-2} x_i, \quad H(x):=|x|^n, \quad \text { for } x \in \mathbb{R}^n, \\
a_i^k(x) & :=\left(a_i * \phi_k\right)(x), \quad H_k(x):=\left(H * \phi_k\right)(x), \quad \text { for } x \in \mathbb{R}^n,
\end{aligned}
$$
where $\left\{\phi_k\right\}$ is a family of radially symmetric smooth mollifiers. Standard properties of convolution and the fact $a(x)$ is continuous imply $a^k \rightarrow a$ uniformly on compact subset of $\mathbb{R}^n$. From \cite[Lemma 2.4]{FF} and \cite{CFR}, we know that
 $a^k$ satisfies
 \begin{equation*}
    \left\langle \nabla a^k(z)\xi,\xi\right\rangle\geq\frac{1}{\lambda}(|z|^2+s_k^2)^{\frac{n-2}{2}}|\xi|^2
    \qquad \text{and}\qquad |\nabla a^k(z)|\leq \lambda (|z|^2+s_k^2)^{\frac{n-2}{2}}
 \end{equation*}
for every $z$, $\xi\in\mathbb{R}^n$, with $s_k\neq0$ and $s_k\rightarrow0$ as $k\rightarrow+\infty$, and $\lambda>0$ is a constant.
For $k$ fixed and using Lemma \ref{Apale1}, we let $u_{k,R} \in W^{1, n}\big(B_{6R}(0)\cap\mathbb{R}^{n}_{+}\big)$ be a solution of
\begin{equation}\label{neweq6-1}
    \left\{\begin{array}{l}
-\operatorname{div}\left(a^k\left(\nabla u_{k,R}\right)\right)=e^{u} \quad \text { in } B_{6R}(0)\cap\mathbb{R}^{n}_{+} \\
u_{k,R}=u \quad \text { on } \partial B_{6R}(0)\cap\mathbb{R}^n_{+}\\
a^k\left(\nabla u_k\right) \cdot \nu=e^{\frac{n-1}{n}u} \quad \text { on } \partial \mathbb{R}^n_{+}\cap B_{6R}(0),
\end{array}\right.
\end{equation}
where $\nu$ is the unit outward normal vector of $\partial \mathbb{R}^n_{+}$.

By the locally boundedness of $u$ and the elliptic estimates in \cite{L,S,T}, Lemma \ref{rele:2} and Lemma \ref{Apale1}, we deduce that $\{u_{k,R}\}$ are bounded in
 $C^{1,\theta}_{\text{loc}}\bigg(\overline{ B_{6R}(0)\cap\mathbb{R}^{n}_{+}}\setminus\big(\partial \mathbb{R}^n_{+}\cap\partial B_{6R}(0)\big)\bigg)$
  uniformly in $k$, as $k\rightarrow+\infty$. Then by the Ascoli--Arzel$\acute{\text{a}}$'s Theorem and a diagonal process, we obtain that $u_{k,R}$
   converges in $C^1_{loc}\bigg(\overline{ B_{6R}(0)\cap\mathbb{R}^{n}_{+}}\setminus\big(\partial \mathbb{R}^n_{+}\cap\partial B_{6R}(0)\big)\bigg)$ to the unique solution $\overline{u_R}$ of
   \begin{equation}\label{neweq6-2}
    \left\{
        \begin{aligned}
        & -\Delta_{n}\overline{u_R}=e^{u} \quad \, &{\rm{in}}\,\, B_{6R}(0)\cap\mathbb{R}^{n}_{+},\\
        &\overline{u_R}=u \quad\, &\text { on }  \mathbb{R}^n_{+}\cap \partial B_{6R}(0)\\
        &a(\nabla \overline{u_R})\cdot \nu =e^{\frac{n-1}{n}u}\quad\, &\text { on } \partial \mathbb{R}^n_{+}\cap B_{6R}(0).
        \end{aligned}
        \right.
   \end{equation}
Testing \eqref{neweq6-1} and \eqref{neweq6-2} with ( $\overline{u_R}-u$ ) in $ B_{6R}(0)\cap\mathbb{R}^{n}_{+}$, we obtain
$$
\int_{B_{6R}(0)\cap\mathbb{R}^{n}_{+}}(a(\nabla \overline{u_R})-a(\nabla u)) \cdot \nabla(\overline{u_R}-u) \textup{d}x=0,
$$
which implies $\overline{u_R}-u \equiv C$ in $\overline{ B_{6R}(0)\cap\mathbb{R}^{n}_{+}} $. Since $\overline{u_R}=u$ on $ \mathbb{R}^n_{+}\cap \partial B_{6R}(0) $, we get $\overline{u_R}=u$. Hence $u_{k,R}\rightarrow u$ in $C^1_{loc}\bigg( \overline{ B_{6R}(0)\cap\mathbb{R}^{n}_{+}}\setminus\big(\partial \mathbb{R}^n_{+}\cap\partial B_{6R}(0)\big)\bigg)$ as $k\rightarrow+\infty$.

Note that \eqref{neweq6-1} is uniformly elliptic and implies $u_{k,R} \in C^{\infty}\big(B_{6R}(0)\cap\mathbb{R}^{n}_{+}\big) \cap W_{\text {loc }}^{2,2}\bigg( \overline{ B_{6R}(0)\cap\mathbb{R}^{n}_{+}}\setminus\big(\partial \mathbb{R}^n_{+}\cap\partial B_{6R}(0)\big)\bigg)$. Since $u_{k,R} \in$ $C_{\text {loc }}^{1, \alpha}\bigg( \overline{ B_{6R}(0)\cap\mathbb{R}^{n}_{+}}\setminus\big(\partial \mathbb{R}^n_{+}\cap\partial B_{6R}(0)\big)\bigg)$, we have $a^k\left(\nabla u_{k,R}\right) \in W_{\text {loc }}^{1,2}\bigg( \overline{ B_{6R}(0)\cap\mathbb{R}^{n}_{+}}\setminus\big(\partial \mathbb{R}^n_{+}\cap\partial B_{6R}(0)\big)\bigg)$. In addition, since $\mathbb{R}^n_{+}$ is smooth, $u_{k,R} \in C_{\text {loc }}^{2, \alpha}\bigg( \overline{ B_{6R}(0)\cap\mathbb{R}^{n}_{+}}\setminus\big(\partial \mathbb{R}^n_{+}\cap\partial B_{6R}(0)\big)\bigg)$.

Given $R>1$ large, let $U$ be a $C^2$ domain such that $\mathbb{R}^n_{+} \cap B_{4 R} \subset U \subset \mathbb{R}^n_{+} \cap B_{5 R}$.


To simplify the notation, we shall drop the dependency of $a$ on k and we write $a$ instead of $a^k$ respectively.
Let $\psi \in C_c^{0,1}\left(B_{4 R}\right)$ and test \eqref{neweq6-1} with $\psi$, we have
\begin{equation}\label{neweq6-3}
    \int_{\mathbb{R}^n_{+}} a\left(\nabla u_{k,R}\right) \cdot \nabla \psi \textup{d} x=\int_{\partial \mathbb{R}^n_{+}} e^{\frac{n-1}{n}u} \psi \textup{d} \sigma +\int_{ \mathbb{R}^n_{+}} e^{u} \psi \textup{d} x.
\end{equation}
For $\delta>0$ small, we define the set
$$
U_\delta:=\{x \in U: \operatorname{dist}(x, \partial U)>\delta\}.
$$
It follows from \cite{GT} that for small $\delta>0$, we have $d(x):=\operatorname{dist}(x, \partial U) \in C^2\left(U \backslash U_{2 \delta}\right)$ and every point $x \in U \backslash U_{2 \delta}$ can be uniquely written as $x=y-|x-y| \nu(y)$, where $y(x) \in \partial U$ and $\nu(y)$ is the unit outward normal to $\partial U$ at $y$.
Set
$$
\gamma(t)=\left\{\begin{array}{ll}
0, & t \in[0, \delta], \\
\frac{t-\delta}{\delta}, & t \in[\delta, 2 \delta), \\
1, & t \in[2 \delta, \infty),
\end{array} \quad \text { and } \quad \xi_\delta(x)=\gamma(d(x)) \text { in } U .\right.
$$
Then we have $\xi_\delta \in C_c^{0,1}(U)$ satisfies
$$
\xi_\delta=1 \text { in } U_{2 \delta}, \quad \xi_\delta=0 \text { in } U \backslash U_\delta \quad \text { and } \quad \nabla \xi_\delta(x)=-\frac{1}{\delta} \nu(y(x)) \, \text { inside } U_\delta \backslash U_{2 \delta}.
$$
Let $\varphi \in C^2(U) \cap C^1(U)$ and $\operatorname{supp}(\varphi) \Subset B_{4 R}$. Using $\psi=\partial_m \varphi \xi_\delta$ in \eqref{neweq6-3} with $1 \leq m \leq n$ and integrating by parts, we obtain
\begin{equation}\label{neweq6-4}
    \begin{aligned}
\sum_i\left(-\int_U \partial_i \varphi \partial_m\left(a_i\left(\nabla u_{k,R}\right) \xi_\delta\right) \textup{d} x\right. & \left.+\int_U a_i\left(\nabla u_{k,R}\right) \partial_m \varphi \partial_i \xi_\delta \textup{d} x\right)+\int_U \varphi\partial_{m}(e^{u} \xi_\delta \textup{d}x)=0 \\
\Rightarrow \sum_i \int_U \partial_i \varphi \partial_m a_i\left(\nabla u_{k,R}\right) \xi_\delta \textup{d}x= & \sum_i\left(-\int_U \partial_i \varphi a_i\left(\nabla u_{k,R}\right) \partial_m \xi_\delta \textup{d}x\right. \\
& \left.+\int_U a_i\left(\nabla u_{k,R}\right) \partial_m \varphi \partial_i \xi_\delta \textup{d}x\right)+\int_U \varphi\partial_{m}(e^{u} \xi_\delta )\textup{d}x .
\end{aligned}
\end{equation}
Since $\varphi \in C^1(\bar{U})$ and $\operatorname{supp}(\varphi) \Subset B_{4 R}$, we can pass to the limit as $\delta \rightarrow 0$ in \eqref{neweq6-4} and get
\begin{equation}\label{neweq6-5}
\begin{aligned}
    \sum_i \int_U \partial_i \varphi \partial_m a_i\left(\nabla u_{k,R}\right) \textup{d}x= & \sum_i\left(\int_{\partial \mathbb{R}^n_{+}} \partial_i \varphi a_i\left(\nabla u_{k,R}\right) \nu^m \textup{d}\sigma\right. \\
& \left.-\int_{\partial \mathbb{R}^n_{+}} a_i\left(\nabla u_{k,R}\right) \partial_m \varphi \nu^i \textup{d}\sigma\right)+\int_U \varphi\partial_{m}(e^{u}  )\textup{d}x-\int_{\partial\mathbb{R}^{n}_{+}} \varphi e^{u}\nu^{m}\textup{d}\sigma,
\end{aligned}
\end{equation}
where $\nu=\left(\nu^1, \ldots, \nu^n\right)$ is the unit outward normal vector to $\partial \mathbb{R}^n_{+}$. We have $\nu=(0, \ldots, 0,-1)$ on $\partial \mathbb{R}^n_{+}$.
Let $\eta \in C_c^{\infty}\left(B_{4 R}\right)$. We choose $\varphi=a_m\left(\nabla u_{k,R}\right) e^{\gamma u} \eta^2$ and divide $m$ into two cases.

\noindent {\bf{Case 1}}: $m \neq n$.
It follows from \eqref{neweq6-1} that, \eqref{neweq6-5} becomes
\begin{equation}\label{neweq6-6}
    \begin{aligned}
&\sum_i \int_U \partial_m a_i\left(\nabla u_{k,R}\right) \partial_i\left(a_m\left(\nabla u_{k,R}\right) e^{\gamma u} \eta^2\right) \textup{d}x\\  =&-\sum_i \int_{\partial \mathbb{R}^n_{+}} e^{\frac{n-1}{n}u} \partial_m \varphi \textup{d}\sigma+\int_U \varphi\partial_{m}(e^{u}  )\textup{d}x \\
 =&\sum_i \int_{\partial \mathbb{R}^n_{+}} \partial_m\left(e^{\frac{n-1}{n}u}\right) a_m\left(\nabla u_{k,R}\right) e^{\gamma u} \eta^2 \textup{d}\sigma+\int_U a_m\left(\nabla u_{k,R}\right) e^{\gamma u} \eta^2\partial_{m}(e^{u}  )\textup{d}x.
\end{aligned}
\end{equation}

\noindent {\bf{Case 2}}: $m=n$.
Note that $\varphi=-e^{(\frac{n-1}{n}+\gamma )u } \eta^2$ on $\partial \mathbb{R}^n_{+}$ in this case. Then, \eqref{neweq6-5} becomes
\begin{equation*}
    \begin{aligned}
&\sum_i \int_U \partial_i \varphi \partial_m a_i\left(\nabla u_{k,R}\right) \textup{d}x\\=&\sum_{i=1}^n \int_{\partial \mathbb{R}^n_{+}} \partial_i \varphi a_i\left(\nabla u_{k,R}\right) \nu^n \textup{d}\sigma-\int_{\partial \mathbb{R}^n_{+}} a_n\left(\nabla u_{k,R}\right) \partial_n \varphi \nu^n \textup{d}\sigma+\int_U \varphi\partial_{n}(e^{u}  )\textup{d}x+\int_{\partial\mathbb{R}^{n}_{+}} \varphi e^{u}\textup{d}\sigma \\
=&-\sum_{i=1}^{n-1} \int_{\partial \mathbb{R}^n_{+}} \partial_i \varphi a_i\left(\nabla u_{k,R}\right) \textup{d}\sigma+\int_U \varphi\partial_{n}(e^{u}  )\textup{d}x+\int_{\partial\mathbb{R}^{n}_{+}} \varphi e^{u}\textup{d}\sigma
\\=&\sum_{i=1}^{n-1} \int_{\partial \mathbb{R}^n_{+}} \partial_i\left(e^{(\frac{n-1}{n}+\gamma )u } \eta^2\right) a_i\left(\nabla u_{k,R}\right) \textup{d}\sigma -\int_U e^{(\frac{2n-1}{n}+\gamma )u } \eta^2\partial_{n}u\textup{d}x-\int_{\partial\mathbb{R}^{n}_{+}} e^{(\frac{2n-1}{n}+\gamma )u } \eta^2 \textup{d}\sigma.
\end{aligned}
\end{equation*}

Combining \eqref{neweq6-5} and \eqref{neweq6-6}, we obtain
\begin{equation*}
    \begin{aligned}
&\sum_{i, m=1}^n\left|\int_U \partial_m a_i\left(\nabla u_{k,R}\right) \partial_i\left(a_m\left(\nabla u_{k,R}\right) e^{\gamma u} \eta^2\right) \textup{d}x\right| \\ \leq &C\bigg( \int_{\partial \mathbb{R}^n_{+}}\left|a\left(\nabla u_{k,R}\right)\right||\nabla u| e^{(\frac{n-1}{n}+\gamma )u } \eta^2 \textup{d}\sigma
 +\int_{\partial \mathbb{R}^n_{+}}\left|a\left(\nabla u_{k,R}\right)\right||\nabla \eta| e^{(\frac{n-1}{n}+\gamma )u } \eta \textup{d}\sigma
 \\&+\int_{\partial\mathbb{R}^{n}_{+}} e^{(\frac{2n-1}{n}+\gamma )u } \eta^2 \textup{d}\sigma+\int_{ \mathbb{R}^n_{+}}  \left|a\left(\nabla u_{k,R}\right)\right||\nabla u| e^{(1+\gamma )u } \eta^2 )\textup{d}x+\int_U e^{(\frac{2n-1}{n}+\gamma )u }|\nabla u| \eta^2\textup{d}x\bigg),
\end{aligned}
\end{equation*}
where $C$ only depends on $n, p, \gamma$.
Using \cite[Lemma 4.5]{AKM} and arguing as in \cite{AKM}, we obtain
\begin{equation}\label{neweq6-8}
    \begin{split}
& \int_U\left|\nabla a\left(\nabla u_{k,R}\right)\right|^2 e^{\gamma u} \eta^2 \textup{d}x \leq C\bigg(\int_U\left|a\left(\nabla u_{k,R}\right)\right|^2\left|\nabla\left(e^{(\frac{\gamma}{2})u } \eta\right)\right|^2 \textup{d}x \\
& +\int_{\partial \mathbb{R}^n_{+}}\left|a\left(\nabla u_{k,R}\right) \| \nabla u\right| e^{(\frac{n-1}{n}+\gamma )u } \eta^2 \textup{d}\sigma+\int_{\partial \mathbb{R}^n_{+}}\left|a\left(\nabla u_{k,R}\right)\right||\nabla \eta| e^{(\frac{n-1}{n}+\gamma )u } \eta \textup{d}\sigma
\\ &+\int_{\partial\mathbb{R}^{n}_{+}} e^{(\frac{2n-1}{n}+\gamma )u } \eta^2 \textup{d}\sigma+\int_{ \mathbb{R}^n_{+}}  \left|a\left(\nabla u_{k,R}\right)\right||\nabla u| e^{(1+\gamma )u } \eta^2 )\textup{d}x+\int_U e^{(\frac{2n-1}{n}+\gamma )u }|\nabla u| \eta^2\textup{d}x\bigg).
\end{split}
\end{equation}
Since $\left(u_{k,R}\right)_{k \in \mathbb{N}}$ is bounded in $C_{\text {loc }}^{1, \alpha}\bigg(\overline{ B_{6R}(0)\cap\mathbb{R}^{n}_{+}}\setminus\big(\partial \mathbb{R}^n_{+}\cap\partial B_{6R}(0)\big)\bigg)$, taking $\gamma=0$, it follows from \eqref{neweq6-8} that $a^k\left(\nabla u_{k,R}\right) \in W_{\text{loc}}^{1,2}\bigg(\overline{ \mathbb{R}^{n}_{+}}\cap B_{6R}(0)    \bigg)$ and $\left\{a^k\left(\nabla u_{k,R}\right)\right\}_{k \in \mathbb{N}}$ is uniformly bounded in $W_{\text{loc}}^{1,2}\bigg(\overline{ \mathbb{R}^{n}_{+}}\cap B_{6R}(0)    \bigg)$.
Since $\left(u_{k,R}\right)_{k \in \mathbb{N}}$ converges up to a subsequence in $C_{\text {loc }}^1\bigg(\overline{ B_{6R}(0)\cap\mathbb{R}^{n}_{+}}\setminus\big(\partial \mathbb{R}^n_{+}\cap\partial B_{6R}(0)\big)\bigg)$ to $u$ and $a^k \rightarrow a$ uniformly on compact subset of $\mathbb{R}^n$, we get $a^k\left(\nabla u_{k,R}\right) \rightarrow a(\nabla u)$ in $C_{\text {loc }}^0\bigg(\overline{ B_{6R}(0)\cap\mathbb{R}^{n}_{+}}\setminus\big(\partial \mathbb{R}^n_{+}\cap\partial B_{6R}(0)\big)\bigg)$. Therefore, as $k \rightarrow \infty$, up to a subsequence, we have $a^k\left(\nabla u_{k,R}\right) \rightharpoonup a(\nabla u)$ in $W_{\text{loc}}^{1,2}\bigg(\overline{ \mathbb{R}^{n}_{+}}\cap B_{6R}(0)    \bigg)$ and $a(\nabla u) \in W_{\text{loc}}^{1,2}\bigg(\overline{ \mathbb{R}^{n}_{+}}\cap B_{6R}(0)    \bigg)$. So, up to a subsequence, by passing to the limit as $k \rightarrow \infty$ into \eqref{neweq6-8}, we obtain
\begin{equation}\label{neweq6-9}
    \begin{split}
& \int_U|\nabla a(\nabla u)|^2 e^{\gamma u} \eta^2 \textup{d}x \leq C\bigg(\int_U|a(\nabla u)|^2\left|\nabla\left(e^{\frac{\gamma}{2}u } \eta\right)\right|^2 \textup{d}x \\
&+\int_{\partial \mathbb{R}^n_{+}}|a(\nabla u)||\nabla u| e^{(\frac{n-1}{n}+\gamma )u } \eta^2 \textup{d}\sigma+\int_{\partial \mathbb{R}^n_{+}}|a(\nabla u)||\nabla \eta| e^{(\frac{n-1}{n}+\gamma )u } \eta \textup{d}\sigma
\\ &+\int_{\partial\mathbb{R}^{n}_{+}} e^{(\frac{2n-1}{n}+\gamma )u } \eta^2 \textup{d}\sigma+\int_{ \mathbb{R}^n_{+}}  \left|a\left(\nabla u\right)\right||\nabla u| e^{(1+\gamma )u } \eta^2 \textup{d}x+\int_U e^{(\frac{2n-1}{n}+\gamma )u }|\nabla u| \eta^2\textup{d}x\bigg).
\end{split}
\end{equation}

First, let $\eta$ be a non-negative cut-off function such that $\eta=1$ in $B_R, \eta=0$ outside $B_{2 R}$ and $|\nabla \eta| \leq \frac{C_n}{R}$. It follows from Proposition \ref{th3-2} and \eqref{neweq6-9} that
$$
\int_{B_R \cap \mathbb{R}^n_{+}}|\nabla(a(\nabla u))|^2 e^{\gamma u} \textup{d}x \leq C\left(1+R^{-n-\gamma \frac{n^2}{n-1}}\right), \quad \forall R>1.
$$
Second, let $\eta$ be a non-negative cut-off function such that $\eta=1$ in $B_{2 R} \backslash B_R, \eta=0$ outside $B_{3 R} \backslash B_{\frac{R}{2}}$ and $|\nabla \eta| \leq \frac{C_n}{R}$. Using Proposition \ref{th3-2} and \eqref{neweq6-9}, we get
$$
\int_{\left(B_{2 R} \backslash B_R\right) \cap \mathbb{R}^n_{+}}|\nabla(a(\nabla u))|^2 e^{\gamma u} \textup{d}x \leq C R^{-n-\gamma \frac{n^2}{n-1}}, \quad \forall R>1,
$$
where $C$ is independent of $R$.

\end{proof}

\section{Proof of our main Theorem \ref{Th:1-1}}
 In this subsection, by using Bochner's skill and a Pohozaev identity like \cite{GuL,Zhou}, we show that $u$ satisfies the differential identity in Proposition \ref{prop9-1}, which implies Theorem \ref{Th:1-1}. Recall that
\begin{equation}\label{eq:9-1}
    \begin{cases}
        \Delta_{n}u(y,t) =-e^{u},\quad \hbox{ }(y,t)\in \mathbb{R}^{n}_{+}=\mathbb{R}^{n-1}\times\{t>0 \},\\
        |Du|^{n-2}\frac{\partial u}{\partial t}(y,0)=-e^{\frac{n-1}{n}u},
        \\\dis\int_{ \mathbb{R}^{n}_{+}}e^{u}<+\infty \quad \hbox{ and } \quad \displaystyle\int_{\partial \mathbb{R}^{n}_{+}}e^{\frac{n-1}{n}u}<+\infty.
    \end{cases}
\end{equation}
\begin{prop}\label{prop9-1}
    Let $u$ be a solution of \eqref{eq:9-1}, Then $|\nabla u|^{n-2}\nabla u\in C^{1,\alpha}_{\text{loc}}(\overline{\mathbb{R}^{n}_{+}}) $ and
    \begin{equation*}
        \begin{split}
            \partial_{j} (|\nabla u|^{n-2}u_i )= \frac{n-1}{n}(|\nabla u|^{n-2} u_{i}u_{j}-\frac{|\nabla u|^{n}}{n}\delta_{ij}    )-\frac{e^{u}}{n}\delta_{ij},
        \end{split}
    \end{equation*}
    where $u_i:=\frac{\partial u}{\partial x_i}$.
\end{prop}
\begin{proof}
 Since proof of integral inequality requires a regularization argument considering the solutions of the approximating equations \eqref{eq:9-2}, we set
$$
\begin{gathered}
a_i(x):=|x|^{n-2} x_i, \quad H(x):=|x|^n \quad \text { for } x \in \mathbb{R}^{n}, \\
a_i^k(x):=\left(\frac{1}{k^2}+|x|^2\right)^{\frac{n-2}{2}} x_i, \quad H_k(x):=\left(\frac{1}{k^2}+|x|^2\right)^{\frac{n}{2}} \quad \text { for } x \in \mathbb{R}^{n}.
\end{gathered}
$$
Straightforward computation implies that $a^k \rightarrow a$ uniformly on compact subset of $\mathbb{R}^{n}$ and $a^k$ satisfiesthe three condition \eqref{mixeq1}, \eqref{mixeq2} and \eqref{mixeq3} with $s$ replaced by $\frac{1}{k^2}$. Using Lemma  \ref{Apale2}, we obtain $u_{k,R} \in W^{1, n}_{\text{loc}}(\mathbb{R}^{n}_{+})$ is a solution of
\begin{equation}\label{eq:9-2}
    \begin{cases}\operatorname{div}\left(a^k\left(\nabla u_{k,R}\right)\right)=e^{u}  & \text { in } B_{6R}(0)\cap\mathbb{R}^{n}_{+}, \\ a^k\left(\nabla u_{k,R}\right) \cdot \nu=e^{\frac{n-1}{n}u} & \text { on } \partial \mathbb{R}^{n}_{+}\cap B_{6R}(0),
    \\u_{k,R} =u & \text { on }  \mathbb{R}^{n}_{+}\cap\partial B_{6R}(0) .\end{cases}
\end{equation}
By the locally boundedness of $u$ and the elliptic estimates in \cite{L,S,T}, Lemma \ref{rele:2} and Lemma \ref{Apale2}, we have $u_{k,R} \in C^{\infty}(\mathbb{R}^{n}_{+}\cap B_{6R}(0)) \cap C_{\text{loc}}^{2, \alpha}\bigg(\overline{ \mathbb{R}^{n}_{+}\cap B_{6R}(0)    }\setminus\big(\partial \mathbb{R}^n_{+}\cap\partial B_{6R}(0)\big)\bigg)$ and $u_{k,R} \in C_{\text{loc}}^{1, \alpha}\bigg(\overline{ \mathbb{R}^{n}_{+}\cap B_{6R}(0)    }\setminus\big(\partial \mathbb{R}^n_{+}\cap\partial B_{6R}(0)\big)\bigg)$, uniformly in $k$. Hence, $\left(u_{k,R}\right)_{k \in \mathbb{N}}$ converges up to a subsequence in $C_{\text{loc}}^1\bigg(\overline{ \mathbb{R}^{n}_{+}\cap B_{6R}(0)    }\setminus\big(\partial \mathbb{R}^n_{+}\cap\partial B_{6R}(0)\big)\bigg)$ to $u$. Besides, the proof of Proposition \ref{newSECONprop1} shows that as $k \rightarrow \infty$, up to a subsequence, we have $a^k\left(\nabla u_{k,R}\right) \rightharpoonup a(\nabla u)$ in $W_{\text{loc}}^{1,2}\bigg(\overline{ \mathbb{R}^{n}_{+}}\cap B_{6R}(0)    \bigg)$. In the following, we only consider the subsequence of $\left(u_{k,R}\right)_{k \in \mathbb{N}}$ satisfying that $u_{k,R} \rightarrow u$ in $C_{\text{loc}}^1\bigg(\overline{ \mathbb{R}^{n}_{+}\cap B_{6R}(0)    }\setminus\big(\partial \mathbb{R}^n_{+}\cap\partial B_{6R}(0)\big)\bigg)$ and $a^k\left(\nabla u_{k,R}\right) \rightharpoonup a(\nabla u)$ in $W_{\text{loc}}^{1,2}\bigg(\overline{ \mathbb{R}^{n}_{+}}\cap B_{6R}(0)    \bigg)$.

Set
$$
\begin{aligned}
& X=\left(X_i\right), \quad X_i:=|\nabla u|^{n-2} u_i, \quad X_{i j}:=\partial_j X_i, \quad L_{i j}:=X_i u_j-\frac{|\nabla u|^n}{n } \delta_{i j}, \\
& X_{i,R}^k:=a_i^k\left(\nabla u_{k,R}\right), \quad X_{i j,R}^k:=\partial_j X_{i,R}^k, \quad L_{i j,R}^k:=X_{i,R}^k \partial_j u_{k,R}-\frac{H_k\left(\nabla u_{k,R}\right)}{n} \delta_{i j}-D_{i j,R}^k,
\end{aligned}
$$
where
$$
D_{i j,R}^k=(n-2)\left(\frac{1}{k^2}+\left|\nabla u_{k,R}\right|^2\right)^{\frac{n-2}{2}} \frac{\frac{1}{k^2} \partial_i u_{k,R} \partial_j u_{k,R}}{\left(\frac{1}{k^2}+\left|\nabla u_{k,R}\right|^2\right)} .
$$
We denote
$$ W_{ij}:=X_{ij} + \frac{e^{u}}{n}\delta_{ij}, \quad \hbox{ }W_{ij,R}^{k}:=X_{ij,R}^{k} + \frac{e^{u}}{n}\delta_{ij}. $$
Note that $L_{i j}, L_{i j,R}^k$ and $D_{i j,R}^k$ are symmetry, but $X_{i j}$ is not.
Since $a b \leq \frac{(a+b)^2}{4}, u>0$ and $u_{k,R} \in C_{\text{loc}}^{1, \alpha}\bigg(\overline{ \mathbb{R}^{n}_{+}\cap B_{6R}(0)    }\setminus\big(\partial \mathbb{R}^n_{+}\cap\partial B_{6R}(0)\big)\bigg)$, uniformly in $k$, we can see $D^k \rightarrow 0$
uniformly on compact subset of $ \overline{ \mathbb{R}^{n}_{+}\cap B_{6R}(0)    }\setminus\big(\partial \mathbb{R}^n_{+}\cap\partial B_{6R}(0)\big)$ .
Note that we can write the matrix $X_{i j,R}^k-\frac{n-1}{n}L_{i j,R}^k$ as $A\left(B+C\right)$, where
$$
\begin{aligned}
& A_{i j}=\delta_{i j}+\frac{(n-2) \partial_i u_{k,R} \partial_j u_{k,R}}{\left(\frac{1}{k^2}+\left|\nabla u_{k,R}\right|^2\right)}, \quad B_{i j}=\left(\frac{1}{k^2}+\left|\nabla u_{k,R}\right|^2\right)^{\frac{n-2}{2}} \partial_{i j} u_{k,R}, \\
& C_{i j}=\frac{(n-1)\left(\frac{1}{k^2}+\left|\nabla u_{k,R}\right|^2\right)^{\frac{n}{2}}}{n}\left[\frac{1}{n}\delta_{i j}- \frac{\partial_i u_{k,R} \partial_j u_{k,R}}{\left(\frac{1}{k^2}+\left|\nabla u_{k,R}\right|^2\right)}\right].
\end{aligned}
$$
It follows from \cite[Lemma 4.5]{AKM} that
\begin{equation}\label{eq:9-3}
    \sum_{i, j}\left(W_{i j}^k-\frac{n-1}{n} L_{i j,R}^k\right)\left(W_{j i}^k-\frac{n-1}{n} L_{j i}^k\right) \geq C_p \sum_{i, j}\left|W_{i j}^k-\frac{n-1}{n} L_{i j,R}^k\right|^2 \geq 0 .
\end{equation}
To simplify the notation, we shall drop the dependency on $k$ and write $a, \tilde{u}$, $X, X_i, X_{i j}, L_{i j}, D_{i j},W_{i j}$ instead of $a^k, u_{k,R}, X^k, X_{i,R}^k, X_{i j,R}^k, L_{i j,R}^k, D_{i j,R}^k,W_{i j}^k$ respectively. First, using \eqref{eq:9-2}, that is, $\sum_i X_{i i}=-e^{u}$, we get the following Serrin-Zou type differential equality for $\epsilon>0$
\begin{equation}\label{eq:9-4}
\begin{aligned}
& \sum_i \partial_i\left(\sum_j e^{-(\frac{n-1}{n}-\epsilon)u} X_{i j} X_j-\frac{n-1}{n}\left( \frac{n-1}{n}+\epsilon \right) e^{-(\frac{n-1}{n}-\epsilon)u}|X|^{\frac{n}{n-1}} X_i+\frac{e^{(\frac{1}{n}+\epsilon)u}}{n}X_{i}\right) \\
& =\sum_{i, j}\left(e^{-(\frac{n-1}{n}-\epsilon)u} X_{i i j} X_j+e^{-(\frac{n-1}{n}-\epsilon)u} X_{i j} X_{j i}-\left(\frac{n-1}{n}-\epsilon\right) e^{-(\frac{n-1}{n}-\epsilon)u} X_{i j} X_j u_i\right. \\
& \left.-\left( \frac{n-1}{n}+\epsilon \right) e^{-(\frac{n-1}{n}-\epsilon)u}|X|^{\frac{1}{n-1}} \frac{X_j X_{j i} X_i}{|X|}\right)+\frac{n-1}{n}\left( \frac{n-1}{n}+\epsilon \right)^{2} \sum_i e^{-(\frac{n-1}{n}-\epsilon)u}|X|^{\frac{n}{n-1}} X_i u_i \\
&+ \sum_{i}\left(  \frac{e^{ (\frac{1}{n} +\epsilon )u }}{n}X_{ii}+\left( \frac{1}{n} +\epsilon \right)\frac{e^{ (\frac{1}{n} +\epsilon )u }}{n}u_{i}X_{i}  - \frac{n-1}{n}\left( \frac{n-1}{n}+\epsilon \right) e^{-(\frac{n-1}{n}-\epsilon)u}|X|^{\frac{n}{n-1}} X_{ii} \right)   \\
& =\sum_{i, j}\left(e^{-(\frac{n-1}{n}-\epsilon)u} X_{i j} X_{j i}-\frac{2(n-1)}{n} e^{-(\frac{n-1}{n}-\epsilon)u} X_{i j} X_j \tilde{u}_i\right)+\left( \frac{n-1}{n} \right)^{3} \sum_i e^{-(\frac{n-1}{n}-\epsilon)u}|X|^{\frac{n}{n-1}} X_i u_i \\
& -\frac{n-1}{n} \sum_{i, j} e^{-(\frac{n-1}{n}-\epsilon)u} X_{i j} X_j\left(u_i-\tilde{u}_i\right ) -\frac{2(n-1)}{n^2}e^{(\frac{1}{n}+\epsilon)u}|X|^{\frac{n}{n-1}} -\frac{e^{(\frac{n+1}{n}+\epsilon)u }}{n} + \frac{n-1}{n} P^k+ P_{1,\epsilon}^{k} \\
&+ \sum_{i}\frac{1}{n} \frac{e^{ (\frac{1}{n} +\epsilon )u }}{n}(u_{i} - \tilde{u}_i )X_{i}  \\
& =\sum_{i, j} e^{-(\frac{n-1}{n}-\epsilon)u}\left(W_{i j}-\frac{n-1}{n}L_{i j}\right)\left(W_{j i}-\frac{n-1}{n} L_{j i}\right)+E^k+M^k+ P_{1,\epsilon}^{k},
\end{aligned}\\
\end{equation}
where
\begin{equation*}
\begin{aligned}
M^k= & \frac{2(n-1)}{n^2}e^{(\frac{1}{n}+\epsilon)u} \left(-|X|^{\frac{n}{n-1}}+ \left(\frac{1}{k^2} +|\nabla \tilde{u}|^{2}\right)^{\frac{n}{2}}  \right)+ \sum_{i}\frac{1}{n} \frac{e^{ (\frac{1}{n} +\epsilon )u }}{n}(u_{i} - \tilde{u}_i )X_{i}  \\
&+ 2\frac{n-1}{n^2}\sum_{i}D_{ii}e^{u}  ,
\end{aligned}
\end{equation*}
\begin{equation*}
\begin{aligned}
P^k= & \sum_{i, j}\left(e^{-(\frac{n-1}{n}-\epsilon)u} X_{i j} X_j \tilde{u}_i-e^{-(\frac{n-1}{n}-\epsilon)u}|X|^{\frac{1}{n-1}} \frac{X_j X_{j i} X_i}{|X|}\right) \\
= & e^{-(\frac{n-1}{n}-\epsilon)u} \sum_{i, j} X_{i j} X_j\left(\tilde{u}_i-|X|^{\frac{1}{n-1}} \frac{\tilde{u}_i}{|\nabla \tilde{u}|}\right) ,
\end{aligned}
\end{equation*}
\begin{equation*}
\begin{aligned}
P_{1,\epsilon}^{k}= & \sum_{i, j}\left(\epsilon  e^{-(\frac{n-1}{n}-\epsilon)u} X_{i j} X_j u_i  - \epsilon  e^{-(\frac{n-1}{n}-\epsilon)u}|X|^{\frac{1}{n-1}} \frac{X_j X_{j i} X_i}{|X|}\right) \\
&+ \left(\frac{n-1}{n}\left( \frac{n-1}{n}+\epsilon \right)^{2} - \left( \frac{n-1}{n} \right)^{3} \right) \sum_i e^{-(\frac{n-1}{n}-\epsilon)u}|X|^{\frac{n}{n-1}} X_i u_i \\
&+\epsilon \sum_{i}e^{-(\frac{n-1}{n}-\epsilon)u}|X|^{\frac{n}{n-1}} X_i u_i+\left(\frac{n-1}{n}\left( \frac{n-1}{n}+\epsilon \right)-\left( \frac{n-1}{n}\right)^{2}\right) e^{(\frac{1}{n}+\epsilon)u}|X|^{\frac{n}{n-1}}  ,
\end{aligned}
\end{equation*}
\begin{equation*}
\begin{aligned}
E^k= & \sum_{i, j}\left(-\frac{n-1}{n} e^{-(\frac{n-1}{n}-\epsilon)u} X_{i j} X_j\left(u_i-\tilde{u}_i\right)-2\frac{n-1}{n} e^{-(\frac{n-1}{n}-\epsilon)u} X_{i j} D_{i j}\right), \\
& +\left( \frac{n-1}{n}\right)^{3} \sum_i e^{-(\frac{n-1}{n}-\epsilon)u}|X|^{\frac{n}{n-1}} X_i\left(u_i-\tilde{u}_i\right)+\frac{n-1}{n} P^k \\
& +\left( \frac{n-1}{n}\right)^{3} \sum_i e^{-(\frac{n-1}{n}-\epsilon)u}|X|^{\frac{n}{n-1}} X_i\tilde{u}_i-\left( \frac{n-1}{n}\right)^2 \sum_{i, j} e^{-(\frac{n-1}{n}-\epsilon)u} L_{i j} L_{j i} .
\end{aligned}
\end{equation*}

Note that
$$
\begin{aligned}
G^k:  =&\left( \frac{n-1}{n}\right)^{3} \sum_i e^{-(\frac{n-1}{n}-\epsilon)u}|X|^{\frac{n}{n-1}} X_i\tilde{u}_i-\left( \frac{n-1}{n}\right)^2 \sum_{i, j} e^{-(\frac{n-1}{n}-\epsilon)u} L_{i j} L_{j i} \\
 =&\left( \frac{n-1}{n}\right)^{3} \sum_i  e^{-(\frac{n-1}{n}-\epsilon)u}\left(|X|^{\frac{n}{n-1}}-\left(\frac{1}{k^2}+|\nabla \tilde{u}|^2\right)^{\frac{n}{2}}\right) X_i \tilde{u}_i \\
& +\left( \frac{n-1}{n}\right)^{2} e^{-(\frac{n-1}{n}-\epsilon)u} \frac{1}{k^2}\left(\frac{1}{k^2}+|\nabla \tilde{u}|^2\right)^{n-2}\left[\left(1-\frac{1}{n}\right)|\nabla \tilde{u}|^2-\frac{1}{n} \frac{1}{k^2}\right]+F^k,
\end{aligned}
$$
where
$$
F^k=-\left( \frac{n-1}{n}\right)^2 \sum_{i, j} e^{-(\frac{n-1}{n}-\epsilon)u} D_{i j} D_{j i}+2 \left( \frac{n-1}{n}\right)^2 \sum_{i, j} e^{-(\frac{n-1}{n}-\epsilon)u}\left(X_i \tilde{u}_j-\frac{\left(\frac{1}{k^2}+|\nabla \tilde{u}|^2\right)^{\frac{n}{2}}}{n} \delta_{i j}\right) D_{i j} .
$$
Since $u>0, u_{k,R} \in C_{\text{loc}}^{1, \alpha}\bigg(\overline{ \mathbb{R}^{n}_{+}\cap B_{6R}(0)    }\setminus\big(\partial \mathbb{R}^n_{+}\cap\partial B_{6R}(0)\big)\bigg)$, uniformly in $k$, and $D^k \rightarrow 0$ uniformly on compact subset of $ \overline{ \mathbb{R}^{n}_{+}\cap B_{6R}(0)    }\setminus\big(\partial \mathbb{R}^n_{+}\cap\partial B_{6R}(0)\big)$ , we have $F^k \rightarrow 0$ uniformly on compact subset of $ \overline{ \mathbb{R}^{n}_{+}\cap B_{6R}(0)    }\setminus\big(\partial \mathbb{R}^n_{+}\cap\partial B_{6R}(0)\big)$  and so does $G^k$.
Note that on compact subset of $ \overline{ \mathbb{R}^{n}_{+}\cap B_{6R}(0)    }\setminus\big(\partial \mathbb{R}^n_{+}\cap\partial B_{6R}(0)\big)$ , we have
$$
\left|P^k\right| \leq C|\nabla X||X| \Big| | X|^{\frac{1}{n-1}}-|\nabla \tilde{u}|\Big|.
$$
Because $\left(u_{k,R}\right)_{k \in \mathbb{N}}$ is bounded in $C_{\text{loc}}^1\bigg(\overline{ \mathbb{R}^{n}_{+}\cap B_{6R}(0)    }\setminus\big(\partial \mathbb{R}^n_{+}\cap\partial B_{6R}(0)\big)\bigg)$ and $\left(\frac{1}{k^2}+|x|^2\right)^{\frac{n-2}{2}} x \rightarrow|x|^{n-2} x$ uniformly on compact subset of $\mathbb{R}^{n}$, as $k$ goes to infinity, we get
$$
\left|X^k\right|-|\nabla \tilde{u}|^{n-1} \rightarrow 0, \quad\left|X^k\right|-\left(\frac{1}{k^2}+|\nabla \tilde{u}|^2\right)^{\frac{n-1}{2}} \rightarrow 0 \quad \text { in } C_{\text{loc}}^0\bigg(\overline{ \mathbb{R}^{n}_{+}\cap B_{6R}(0)    }\setminus\big(\partial \mathbb{R}^n_{+}\cap\partial B_{6R}(0)\big)\bigg).
$$
Since $\left(u_{k,R}\right)_{k \in \mathbb{N}}$ is bounded in $C_{\text{loc}}^1\bigg(\overline{ \mathbb{R}^{n}_{+}\cap B_{6R}(0)    }\setminus\big(\partial \mathbb{R}^n_{+}\cap\partial B_{6R}(0)\big)\bigg),\left\{a^k\left(\nabla u_{k,R}\right)\right\}_{k \in \mathbb{N}}$ is bounded in $W_{\text{loc}}^{1,2}\bigg(\overline{ \mathbb{R}^{n}_{+}}\cap B_{6R}(0)    \bigg), x \mapsto x^{\frac{1}{n-1}}$ is locally Lipchitz continuous and is uniformly continuous, we have $P^k \rightarrow 0$ in $L_{\text{loc}}^{2}\bigg(\overline{ \mathbb{R}^{n}_{+}}\cap B_{6R}(0)    \bigg)$.
Therefore, using the fact that $\left(u_{k,R}\right)_{k \in \mathbb{N}}$ converges in $C_{\text{loc}}^1\bigg(\overline{ \mathbb{R}^{n}_{+}\cap B_{6R}(0)    }\setminus\big(\partial \mathbb{R}^n_{+}\cap\partial B_{6R}(0)\big)\bigg)$ to $u, D^k \rightarrow 0$ uniformly on compact subset of $ \overline{ \mathbb{R}^{n}_{+}\cap B_{6R}(0)    }\setminus\big(\partial \mathbb{R}^n_{+}\cap\partial B_{6R}(0)\big)$  and $\left\{a^k\left(\nabla u_{k,R}\right)\right\}_{k \in \mathbb{N}}$ is bounded in $W_{\text{loc}}^{1,2}\bigg(\overline{ \mathbb{R}^{n}_{+}}\cap B_{6R}(0)    \bigg)$, we have $E^k \rightarrow 0$ in $L_{\text{loc}}^{2}\bigg(\overline{ \mathbb{R}^{n}_{+}}\cap B_{6R}(0)    \bigg)$.

Second, using \eqref{eq:9-2}, we get the following Pohozaev-type differential equality
\begin{equation}\label{eq:9-5}
\begin{aligned}
    & \sum_i \partial_n\left(X_i \tilde{u}_i\right)-n \sum_i \partial_i\left(X_i \tilde{u}_n\right)-n\partial_{n}(e^{u}) \\
= & \sum_i X_{i n} \tilde{u}_i-(n-1) \sum_i X_i \tilde{u}_{n i}-n\sum_{i}X_{ii}\tilde{u}_{n}-ne^{u}u_{n}\\
:=& I^{k}+ne^{u}(\tilde{u}_{n}-  u_{n}).
\end{aligned}
\end{equation}

Since $X_{i j}=\left(\frac{1}{k^2}+|\nabla \tilde{u}|^2\right)^{\frac{n-2}{2}}\left(\tilde{u}_{i j}+\frac{(n-2) \tilde{u}_i \tilde{u}_u \tilde{u}_{l j}}{\frac{1}{k^2}+|\nabla \tilde{u}|^2}\right)$, we have
\begin{equation}\label{eq:9-6}
    \begin{aligned}
\sum_i X_{i n} \tilde{u}_i & =\left(\frac{1}{k^2}+|\nabla \tilde{u}|^2\right)^{\frac{n-2}{2}}\left(1+\frac{(n-2)|\nabla \tilde{u}|^2}{\frac{1}{k^2}+|\nabla \tilde{u}|^2}\right) \sum_i \tilde{u}_i \tilde{u}_{n i} \\
& =\left(1+\frac{(n-2)|\nabla \tilde{u}|^2}{\frac{1}{k^2}+|\nabla \tilde{u}|^2}\right) \sum_i X_i \tilde{u}_{n i}.
\end{aligned}
\end{equation}
Using \eqref{eq:9-5} and \eqref{eq:9-6}, we get
\begin{equation}\label{eq:9-7}
    \begin{aligned}
I^k
=& \sum_i \partial_n\left(X_i \tilde{u}_i\right)-n \sum_i \partial_i\left(X_i \tilde{u}_n\right) \\
= & -(n-2) \frac{1}{k^2} \frac{1}{\frac{1}{k^2}+|\nabla \tilde{u}|^2} \sum_i X_i \tilde{u}_{n i} \\
= & -\left(1+\frac{(n-2)|\nabla \tilde{u}|^2}{\frac{1}{k^2}+|\nabla \tilde{u}|^2}\right)^{-1} \frac{(n-2) \frac{1}{k^2}}{\frac{1}{k^2}+|\nabla \tilde{u}|^2} \sum_i X_{i n} \tilde{u}_i.
\end{aligned}
\end{equation}
Note that
$$
\min \{1, n-1\} \leq\left(1+\frac{(n-2)|\nabla \tilde{u}|^2}{\frac{1}{k^2}+|\nabla \tilde{u}|^2}\right) \leq \max \{1, n-1\} .
$$
Therefore, using the fact that $a b \leq \frac{a^2+b^2}{2}$, we have
$$
\begin{aligned}
\left|I^k\right|\leq C_{n} \frac{\frac{1}{k^2}|\nabla \tilde{u}|}{\frac{1}{k^2}+|\nabla \tilde{u}|^2}|\nabla X|\leq C_{n} \frac{1}{k}|\nabla X|.
\end{aligned}
$$
Because $X^k \in W_{\text{loc}}^{1,2}\bigg(\overline{ \mathbb{R}^{n}_{+}}\cap B_{6R}(0)    \bigg)$, uniformly in k, we obtain $I^k \rightarrow 0$ in $L_{\text{loc}}^{2}\bigg(\overline{ \mathbb{R}^{n}_{+}}\cap B_{6R}(0)    \bigg)$. Then, it follows from \eqref{eq:9-4} and \eqref{eq:9-7} that
\begin{equation}\label{eq:9-8}
    \begin{aligned}
&  \sum_i \partial_i\left(\sum_j e^{-(\frac{n-1}{n}-\epsilon)u} X_{i j} X_j-\frac{n-1}{n}\left( \frac{n-1}{n}+\epsilon \right) e^{-(\frac{n-1}{n}-\epsilon)u}|X|^{\frac{n}{n-1}} X_i+\frac{e^{(\frac{1}{n}+\epsilon)u}}{n}X_{i}\right) \\
& +\frac{n-1}{n^{2}}\left[\sum_i \partial_n\left(X_i \tilde{u}_i\right)-n \sum_i \partial_i\left(X_i \tilde{u}_n)-n\partial_{n}(e^{u}\right)\right] \\
 =&\sum_{i, j} e^{-(\frac{n-1}{n}-\epsilon)u}\left(W_{i j}-\frac{n-1}{n}L_{i j}\right)\left(W_{j i}-\frac{n-1}{n} L_{j i}\right)+E^k+\frac{n-1}{n^{2}} I^k+M^{k}+P^{k}_{1,\epsilon}.
\end{aligned}
\end{equation}

Let $\eta \in C_c^{\infty}\left(B_R\right)$. We multiply \eqref{eq:9-8} by $\eta$ and integrate over $\mathbb{R}^{n}_{+}$.
Then, it follows from the divergence theorem and $\nu=(0, \ldots, 0,-1)$ on $\partial \mathbb{R}^{n}_{+}$ that
\begin{equation}\label{eq:9-9}
    \begin{aligned}
& \int_{\mathbb{R}^{n}_{+}}\left(\sum_{i, j} e^{-(\frac{n-1}{n}-\epsilon)u}\left(W_{i j}-\frac{n-1}{n}L_{i j}\right)\left(W_{j i}-\frac{n-1}{n} L_{j i}\right)+E^k+\frac{n-1}{n^{2}} I^k+M^{k}+P^{k}_{1,\epsilon}\right) \eta \textup{d}x \\
 =&-\int_{\mathbb{R}^{n}_{+}} \sum_i\left(\sum_j e^{-(\frac{n-1}{n}-\epsilon)u} X_{i j} X_j-\frac{n-1}{n}\left( \frac{n-1}{n}+\epsilon \right) e^{-(\frac{n-1}{n}-\epsilon)u}|X|^{\frac{n}{n-1}} X_i+\frac{e^{(\frac{1}{n}+\epsilon)u}}{n}X_{i}\right) \eta_i \textup{d}x \\
& -\frac{n-1}{n^{2}} \int_{\mathbb{R}^{n}_{+}}\left(\sum_i X_i \tilde{u}_i \eta_n-n \sum_i X_i \tilde{u}_n \eta_i-ne^{u}\eta_{n}\right) \textup{d}x \\
& -\int_{\partial \mathbb{R}^{n}_{+}}\left(\sum_j e^{-(\frac{n-1}{n}-\epsilon)u} X_{n j} X_j-\frac{n-1}{n}\left( \frac{n-1}{n}+\epsilon \right) e^{-(\frac{n-1}{n}-\epsilon)u}|X|^{\frac{n}{n-1}} X_n+\frac{e^{(\frac{1}{n}+\epsilon)u}}{n}X_{n}\right) \eta \textup{d}\sigma \\
& -\frac{n-1}{n^{2}} \int_{\partial \mathbb{R}^{n}_{+}}\left(\sum X_i \tilde{u}_i-n X_n \tilde{u}_n-ne^{u}\right) \eta \textup{d}\sigma.
\end{aligned}
\end{equation}

Since $X_n^k=-e^{\frac{n-1}{n}u}$ on $\partial \mathbb{R}^{n}_{+}$ and $\sum_i X_{i i}^k=-e^{u}$, we have
\begin{equation}\label{eq:9-10}
    \begin{aligned}
& \int_{\partial \mathbb{R}^{n}_{+}}\left(\sum_j e^{-(\frac{n-1}{n}-\epsilon)u} X_{n j} X_j-\frac{n-1}{n}\left( \frac{n-1}{n}+\epsilon \right) e^{-(\frac{n-1}{n}-\epsilon)u}|X|^{\frac{n}{n-1}} X_n+\frac{e^{(\frac{1}{n}+\epsilon)u}}{n}X_{n}\right) \eta \textup{d}\sigma \\
= & \int_{\partial \mathbb{R}^{n}_{+}}\left(-\frac{n-1}{n}e^{\epsilon u}\sum_{j=1}^{n-1} u_j X_j+e^{\epsilon u} \sum_{j=1}^{n-1} X_{j j}+\frac{n-1}{n}\left( \frac{n-1}{n}+\epsilon \right)e^{\epsilon u}|X|^{\frac{n}{n-1}}+\frac{n-1}{n}e^{(1+\epsilon)u}\right) \eta \textup{d}\sigma \\
= & \int_{\partial \mathbb{R}^{n}_{+}}\Bigg[\left(-\frac{n-1}{n} e^{\epsilon u} \sum_{j=1}^{n-1} X_j \tilde{u}_j+\left( \frac{n-1}{n} \right)^{2}e^{\epsilon u}\sum_{j=1}^n X_j \tilde{u}_j+J^k\right) \eta-\sum_{j=1}^{n-1} e^{\epsilon u} X_j \eta_j
\\&-\epsilon  e^{\epsilon u}u_{j} X_j \eta + \epsilon  e^{\epsilon u} |X|^{\frac{n}{n-1}}\Bigg]\textup{d}\sigma \\
= & -\frac{n-1}{n^{2}} \int_{\partial \mathbb{R}^{n}_{+}}\left(\sum_i X_i \tilde{u}_i-n X_n \tilde{u}_n-ne^{(1+\epsilon)u}\right) \eta \textup{d}\sigma+\int_{\partial \mathbb{R}^{n}_{+}} \left(J^k \eta-\sum_{j=1}^{n-1} e^{\epsilon u} X_j \eta_j \right)\textup{d}\sigma+\epsilon L^{k},
\end{aligned}
\end{equation}
where
$$
\begin{aligned}
J^k= & \frac{n-1 }{n} \sum_{j=1}^{n-1} X_j\left(\tilde{u}_j-u_j\right)+ \left(\frac{n-1 }{n}\right)^{2}\left(|X|^{\frac{n}{n-1}}-\left(\frac{1}{k^2}+|\nabla \tilde{u}|^2\right)^{\frac{n}{2}}\right) \\
& +\left(\frac{n-1 }{n}\right)^{2}\left(\frac{1}{k^2}+|\nabla \tilde{u}|^2\right)^{\frac{n}{2}} \frac{1 / k^2}{\frac{1}{k^2}+|\nabla \tilde{u}|^2},
\end{aligned}
$$
and
\begin{align*}
    L^{k} = \int_{\partial \mathbb{R}^{n}_{+}} (e^{\epsilon u}|X|^{\frac{n}{n-1}}-e^{\epsilon u}u_{j} X_j )\eta \textup{d}\sigma.
\end{align*}
Since $\left(u_{k,R}\right)_{k \in \mathbb{N}}$ is bounded in $C_{\text{loc}}^{1, \alpha}\bigg( \overline{ \mathbb{R}^{n}_{+}\cap B_{6R}(0)    }\setminus\big(\partial \mathbb{R}^n_{+}\cap\partial B_{6R}(0)\big)\bigg)$, we can see that
$$\left(\frac{1}{k^2}+|\nabla \tilde{u}|^2\right)^{\frac{n}{2}} \frac{\frac{1}{k^2}}{\frac{1}{k^2}+|\nabla \tilde{u}|^2} \rightarrow 0$$
and $|X|^{\frac{n}{n-1}} \rightarrow\left(\frac{1}{k^2}+|\nabla \tilde{u}|^2\right)^{\frac{n}{2}}$ uniformly on compact subset of $ \overline{ \mathbb{R}^{n}_{+}\cap B_{6R}(0)    }\setminus\big(\partial \mathbb{R}^n_{+}\cap\partial B_{6R}(0)\big)$.

Also, using the fact that $\left(u_{k,R}\right)_{k \in \mathbb{N}}$ converges in $C_{\text{loc}}^1\bigg(\overline{ \mathbb{R}^{n}_{+}\cap B_{6R}(0)    }\setminus\big(\partial \mathbb{R}^n_{+}\cap\partial B_{6R}(0)\big)\bigg)$ to $u$, we obtain $J^k \rightarrow 0$ uniformly on compact subset of $ \overline{ \mathbb{R}^{n}_{+}\cap B_{6R}(0)    }\setminus\big(\partial \mathbb{R}^n_{+}\cap\partial B_{6R}(0)\big)$ .
It follows from \eqref{eq:9-9} and \eqref{eq:9-10} that
\begin{equation}\label{eq:9-11}
    \begin{aligned}
& \int_{\mathbb{R}^{n}_{+}}\left(\sum_{i, j} e^{-(\frac{n-1}{n} -\epsilon )u}\left(W_{i j}-\frac{n-1}{n} L_{i j}\right)\left(W_{j i}-\frac{n-1}{n} L_{j i}\right)\right) \eta \textup{d}x \\
= & -\int_{\mathbb{R}^{n}_{+}} \sum_{i, j}\left(e^{-(\frac{n-1}{n} -\epsilon )u} X_{i j} X_j-\frac{n-1}{n}\left( \frac{n-1}{n}+\epsilon \right)e^{-(\frac{n-1}{n} -\epsilon )u}|X|^{\frac{n}{n-1}} X_i\right) \eta_i \textup{d}x \\ & -\int_{\partial \mathbb{R}^{n}_{+}} \left(J^k \eta-e^{\epsilon u} X_j \eta_j\right) \textup{d}\sigma+\epsilon L^{k} \\
& -\frac{n-1}{n^{2}} \int_{\mathbb{R}^{n}_{+}}\left(\sum_i X_i \tilde{u}_i \eta_n-N \sum_i X_i \tilde{u}_n \eta_i\right) \textup{d}x-\int_{\mathbb{R}^{n}_{+}}\left(E^k+\frac{n-1}{n^{2}} I^k\right) \eta \textup{d}x.
\end{aligned}
\end{equation}
Using \eqref{eq:9-3} and \eqref{eq:9-11}, we get
\begin{equation}\label{eq:9-12}
    \begin{split}
        &\sum_{i, j} \int_{\mathbb{R}^{n}_{+}} e^{-(\frac{n-1}{n} -\epsilon )u}\left|W_{i j,R}^k-\frac{n-1}{n} L_{i j,R}^k\right|^2 \eta \textup{d}x \\
\leq&C_n\bigg( \sum_{i, j}\left|\int_{\mathbb{R}^{n}_{+}} e^{-(\frac{n-1}{n} -\epsilon )u} X_{i j,R}^k X_j^k \eta_i \textup{d}x\right|+\int_{\mathbb{R}^{n}_{+}} e^{-(\frac{n-1}{n} -\epsilon )u}\left|X^k\right|^{\frac{2 n-1}{n-1}}|\nabla \eta| \textup{d}x+\int_{\partial \mathbb{R}^{n}_{+}}\left|J^k\right| \eta \textup{d}\sigma \\
 &+\int_{\mathbb{R}^{n}_{+}}\left|X^k\right|^{\frac{n}{n-1}}|\nabla \eta| \textup{d}x+\int_{\partial \mathbb{R}^{n}_{+}} \left|X^k\right||\nabla \eta| \textup{d}\sigma+\int_{\mathbb{R}^{n}_{+}}\left(\left|E^k\right|+\left|I^k\right|\right) \eta \textup{d}x+\epsilon |L^{k}|+\left|\int_{\mathbb{R}^{n}_{+}}P_{1,\epsilon}^{k}\eta \textup{d}x\right|\bigg).
    \end{split}
\end{equation}

Note that $\left(u_{k,R}\right)_{k \in \mathbb{N}}$ converges in $C_{\text{loc}}^1\bigg(\overline{ \mathbb{R}^{n}_{+}\cap B_{6R}(0)    }\setminus\big(\partial \mathbb{R}^n_{+}\cap\partial B_{6R}(0)\big)\bigg)$ to $u$, $a^k\left(\nabla u_{k,R}\right) \rightharpoonup a(\nabla u)$ in $W_{\text{loc}}^{1,2}\bigg(\overline{ \mathbb{R}^{n}_{+}}\cap B_{6R}(0)    \bigg)$, $L^{k}$ is uniform bounded for $k$, as $k \rightarrow \infty$ and $\left|\displaystyle\int_{\mathbb{R}^{n}_{+}}P_{1,\epsilon}^{k}\eta \textup{d}x \right|\leq C_{n}\epsilon$, as $k \rightarrow \infty$. Besides, we have shown $E^k, I^k \rightarrow 0$ in $L_{\text {loc }}^2\bigg(\overline{ \mathbb{R}^{n}_{+}\cap B_{6R}(0)    }\setminus\big(\partial \mathbb{R}^n_{+}\cap\partial B_{6R}(0)\big)\bigg)$ and $J^k \rightarrow 0$ uniformly on compact subset of $ \overline{ \mathbb{R}^{n}_{+}\cap B_{6R}(0)    }\setminus\big(\partial \mathbb{R}^n_{+}\cap\partial B_{6R}(0)\big)$ . So, by passing to the limit as $k \rightarrow \infty$ into \eqref{eq:9-12}, we obtain
$$
\begin{aligned}
& \underset{k \rightarrow \infty}{\limsup } \sum_{i, j} \int_{\mathbb{R}^{n}_{+}} e^{-(\frac{n-1}{n} -\epsilon )u}\left|W_{i j,R}^k-\frac{n-1}{n} L_{i j,R}^k\right|^2 \eta \textup{d}x \\
\leq &C_{n}\left(\int_{\mathbb{R}^{n}_{+}} e^{-(\frac{n-1}{n} -\epsilon )u}|\nabla a(\nabla u)||\nabla u|^{n-1}|\nabla \eta| \textup{d}x+\int_{\mathbb{R}^{n}_{+}} e^{-(\frac{n-1}{n} -\epsilon )u}|\nabla u|^{2 n-1}|\nabla \eta| \textup{d}x\right. \\
& \left.+\int_{\mathbb{R}^{n}_{+}}|\nabla u|^n|\nabla \eta| \textup{d}x+\int_{\partial \mathbb{R}^{n}_{+}} |\nabla u|^{n-1}|\nabla \eta| \textup{d}\sigma+\epsilon\right) .
\end{aligned}
$$

Let $R>1$ and $\eta$ be a non-negative cut-off function such that $\eta=1$ in $B_R, \eta=0$ outside $B_{2 R}$ and $|\nabla \eta| \leq \frac{C_n}{R}$. It follows from Corollary \ref{cor3-1}, Proposition \ref{newSECONprop1} and \eqref{eq:9-12} that
\begin{equation}\label{eq:9-13}
    \limsup _{k \rightarrow \infty} \sum_{i, j} \int_{\mathbb{R}^{n}_{+}} e^{-(\frac{n-1}{n} -\epsilon )u}\left|W_{i j,R}^k-\frac{n-1}{n} L_{i j,R}^k\right|^2 \eta \textup{d}x \leq C \left(R^{-\frac{n^{2}}{n-1}\epsilon  }+\epsilon\right) .
\end{equation}
Since $W_{i j,R}^k-\frac{n-1}{n} L_{i j,R}^k \rightharpoonup W_{i j}-\frac{n-1}{n} L_{i j}$ in $L_{\text{loc}}^{2}\bigg(\overline{ \mathbb{R}^{n}_{+}}\cap B_{6R}(0)    \bigg), e^{-(\frac{n-1}{2n} -\frac{\epsilon}{2} )u} \eta^{\frac{1}{2}} \in L^{\infty}(\overline{ \mathbb{R}^{n}_{+}    })$, we have
\begin{equation}\label{eq:9-14}
    e^{-(\frac{n-1}{2n} -\frac{\epsilon}{2} )u} \eta^{\frac{1}{2}}\left(W_{i j,R}^k-\frac{n-1}{n} L_{i j,R}^k\right) \rightharpoonup e^{-(\frac{n-1}{2n} -\frac{\epsilon}{2} )u} \eta^{\frac{1}{2}}\left(W_{i j}-\frac{n-1}{n} L_{i j}\right) \quad \text { in } \quad L^2( \mathbb{R}^{n}_{+}    ).
\end{equation}
Therefore, it follows from \eqref{eq:9-13} and \eqref{eq:9-14} that
$$
\sum_{i, j} \int_{\mathbb{R}^{n}_{+}} e^{-(\frac{n-1}{n} -\epsilon )u}\left|W_{i j}-\frac{n-1}{n} L_{i j}\right|^2 \eta \textup{d}x \leq C\left( R^{-\frac{n^{2}}{n-1}\epsilon  }+\epsilon\right).
$$
Letting $R \rightarrow \infty$ and then letting $ \epsilon \to 0$, we get $W_{i j}-\frac{n-1}{n} L_{i j}=0$ almost everywhere in $\mathbb{R}^{n}_{+}$. Since $u \in C_{\text{loc}}^{1, \alpha}\bigg( \overline{ \mathbb{R}^{n}_{+}   }\bigg)$, we have $L_{i j} \in C_{\text{loc}}^\alpha(\overline{ \mathbb{R}^{n}_{+}    })$, which implies $X_i \in C_{\text{loc}}^{1, \alpha}\bigg( \overline{ \mathbb{R}^{n}_{+}   }\bigg)$. This ends the proof of Proposition \ref{prop9-1}.

\end{proof}

\noindent{\bf{Proof of Theorem \ref{Th:1-1}. }}
We consider the auxiliary function $v=ne^{-\frac{1}{n}u}$, where $u$ is a solution of \eqref{eq:9-1}. A straightforward computation shows that $v>0$ satisfies the following equation
\begin{equation}\label{reproofth1-0}
    \begin{cases}\Delta_n v=(n-1) \frac{|\nabla v|^n}{v}+ n\frac{1}{v} & \text { in } \mathbb{R}^{n}_{+}, \\ |\nabla v|^{n-2} \frac{\partial v}{\partial t}= 1 & \text { on } \partial \mathbb{R}^{n}_{+}.\end{cases}
\end{equation}
Since $a(\nabla u) \in C_{\text{loc}}^{0, \alpha}(\overline{\mathbb{R}^{n}_{+}})$ and $a(\nabla v)=-e^{-\frac{n-1}{n}u} a(\nabla u)$, we have $a(\nabla v) \in C_{\text{loc}}^{0, \alpha}(\overline{\mathbb{R}^{n}_{+}})$. Then, it follows from Proposition \ref{prop9-1} that
\begin{equation}\label{reproofth1-1}
    \partial_j\left(a_i(\nabla v(x))\right)=\lambda(x) \delta_{i j},
\end{equation}
where $\lambda(x)=\frac{1}{n} e^{-\frac{n-1}{n}u}|\nabla u|^n, i, j \in\{1, \ldots, n\}$. The elliptic regularity theory yields that $v \in C_{\text {loc }}^{2, \alpha}(\mathbb{R}^{n}_{+})$, which implies that $\lambda(x) \in C_{\text{loc}}^{1, \alpha}(\mathbb{R}^{n}_{+} )$. Then, using \eqref{reproofth1-1}, we get $a(\nabla v) \in C_{\text{loc}}^{2, \alpha}(\mathbb{R}^{n}_{+} )$. Therefore, it follows from \eqref{reproofth1-1} that for $j \neq i$, we have
$$
\partial_j \lambda(x)=\partial_j \partial_i\left(a_i(\nabla v(x))\right)=\partial_i \partial_j\left(a_i(\nabla v(x))\right)=0
$$
for any $x \in \mathbb{R}^{n}_{+} $, which implies that $\lambda(x)$ is a constant on $\mathbb{R}^{n}_{+}$.
 Then, we deduce that
$$
a(\nabla v)=\lambda\left(x-x_0\right) \quad \text { in } \mathbb{R}^{n}_{+}
$$
for some $x_0 \in \mathbb{R}^n$. Using \eqref{reproofth1-0}, we obtain
\begin{equation}\label{reproofth1-2}
    n \lambda=\operatorname{div}(a(\nabla v))=(n-1) \frac{|\nabla v|^n}{v}+\frac{n}{v} \quad \text { in } \mathbb{R}^{n}_{+}.
\end{equation}

Now, we divide $x_0$ into two cases.

\noindent {\bf{Case 1: $x_0 \in \mathbb{R}_{+}^n$}}. Using $a_n(\nabla v)=1$ on $\partial \mathbb{R}^{n}_{+}$ and \eqref{reproofth1-2}, we get $\lambda=-\frac{1}{x_0^n}<0$ and $v<0$, which contradicts $v=ne^{-\frac{1}{n}u}>0$.

\noindent {\bf{Case 2: $x_0 \in \mathbb{R}_{-}^n$}}. Since $a_n(\nabla v)=1$ on $\partial \mathbb{R}^{n}_{+}$, we get $\lambda=-\frac{1}{x_0^n}>0$. The fact that $|\nabla v|^{n-2} \nabla v=\lambda\left(x-x_0\right)$ in $\overline{\mathbb{R}^{n}_{+}}$ implies $\nabla v=\left(\lambda\left|x-x_0\right|\right)^{\frac{1}{n-1}} \frac{x-x_0}{\left|x-x_0\right|}$. Therefore, it follows from \eqref{reproofth1-2} that
$$
v(x)=\frac{(n-1)(\lambda|x-x_0|)^{\frac{n}{n-1}}+n}{n\lambda},
$$
where $\lambda=-\frac{1}{x_0^n}>0$. Using $v=ne^{-\frac{1}{n}u}$, we get Theorem \ref{Th:1-1}.

\appendix

\section{The existence of the solution of $n$-Laplacian equation with mixed boundary condition}

In this section, we will prove the following existence result, which is useful in the proof of Propositon \ref{newSECONprop1}.
Recall that $$
\begin{aligned}
a_i(x) & :=|x|^{n-2} x_i, \quad H(x):=|x|^n, \quad \forall \,\,  x \in \mathbb{R}^n, \\
a_i^k(x) & :=\left(a_i * \phi_k\right)(x), \quad H_k(x):=\left(H * \phi_k\right)(x), \quad \forall \,\, x \in \mathbb{R}^n,
\end{aligned}
$$
where $\left\{\phi_k\right\}$ is a family of radially symmetric smooth mollifiers.
\begin{lem}\label{Apale1}
    Assume that $\Omega\subset \mathbb{R}^{n}_{+}$ is a bounded domain, $f,g\in L^{\infty}(\mathbb{R}^{n}_{+} )$ and $ u \in W^{1,n}_{\text{loc}}(\mathbb{R}^{n}_{+} ) $. Then the following equation
   \begin{equation*}
    \begin{cases}
        -\operatorname{div}\left(a^k\left(\nabla \bar{v}_{k}\right)\right)=f(x) & \text { in } \mathbb{R}^{n}_{+} \cap\Omega, \\
a^k\left(\nabla \bar{v}_{k})\right) \cdot \nu=g(x) & \text { on } \partial \mathbb{R}^{n}_{+} \cap\Omega, \\
\bar{v}_{k}=u & \text { on } \mathbb{R}^{n}_{+} \cap \partial\Omega,
    \end{cases}
\end{equation*}has a solution $\bar{v}_{k}\in  W^{1,n}_{0}\big( \overline{\mathbb{R}^{n}_{+}}\cap \Omega  \big)$. Moreover, $\|\bar{v}_{k}\|_{W^{1,n}_{0}\big( \overline{\mathbb{R}^{n}_{+}}\cap \Omega  \big)}$ is uniformly bounded as $k\to+\infty$.
\end{lem}
\begin{proof}
    We consider the minimizer $\bar{v}_{k}$ of the variation problem
\begin{equation*}
    \begin{split}
        \min_v\bigg\{J_k(v):=\int_{\mathbb{R}^{n}_{+} \cap \Omega} \left(\frac{1}{n} H_{k}\big(\nabla (v_k+u)\big) -f v\right)\mathrm{d} x-\int_{\partial \mathbb{R}^{n}_{+} \cap \Omega} g v \textup{d}\sigma, \, v\in W^{1,n}_{0}\left(\overline{\mathbb{R}^{n}_{+}}\cap \Omega \right)\bigg\}.
    \end{split}
\end{equation*}
We can see that $ J_k(v) \in C^{1}\bigg(W^{1,n}_{0}\big( \overline{\mathbb{R}^{n}_{+}}\cap \Omega  \big),\mathbb{R}\bigg)$. Since $v\big|_{\partial \Omega\cap \mathbb{R}^{n}_{+} =0}$, by Sobolev inequality, Sobolev trace inequality and the fact that $|a+b|^{n}\leq 2^{n} ( |a|^n +|b|^n) $, we get that
\begin{equation*}
    \begin{split}
        J_k(v) \geq &\int_{\mathbb{R}^{n}_{+} \cap \Omega} \frac{1}{n} H_{k}\big(\nabla (v_k+u)\big) - C_{\Omega,n}\left(\int_{\mathbb{R}^{n}_{+} \cap \Omega}|\nabla v|^{n}\right)^{\frac{1}{n}}
        \\= &  \int_{\mathbb{R}^{n}_{+} \cap \Omega} \frac{1}{n} \int_{\mathbb{R}^{n}}|\nabla (v_k+u)(x) -y  |^{n} \phi_{k}(y)\textup{d}y \textup{d}x  - C_{\Omega,n}\left(\int_{\mathbb{R}^{n}_{+} \cap \Omega}|\nabla v|^{n}\right)^{\frac{1}{n}}
        \\\geq & \int_{\mathbb{R}^{n}_{+}\cap \Omega}\int_{\mathbb{R}^{n}}\frac{1}{2^{n}n}|\nabla v|^{n}(x)\phi_{k}(y)\textup{d}y \textup{d}x -\int_{\mathbb{R}^{n}_{+}\cap \Omega}\frac{1}{n}H_{k}\big(\nabla u\big)\textup{d}x- C_{\Omega,n}\left(\int_{\mathbb{R}^{n}_{+} \cap \Omega}|\nabla v|^{n}\right)^{\frac{1}{n}}
        \\\geq& \int_{\mathbb{R}^{n}_{+}\cap \Omega}\frac{1}{2^{n}n}|\nabla v(x)|^{n} \textup{d}x-\frac{2^n}{n}\int_{\mathbb{R}^{n}_{+}\cap \Omega}\int_{\mathbb{R}^{n}}|\nabla u(x)|^{n}\phi_{k}(y)\textup{d}y\textup{d}x
        \\&-\frac{2^n}{n}\int_{\mathbb{R}^{n}_{+}\cap \Omega}\int_{\mathbb{R}^{n}}|y|^{n}\phi_{k}(y)\textup{d}y\textup{d}x - C_{\Omega,n}\left(\int_{\mathbb{R}^{n}_{+} \cap \Omega}|\nabla v|^{n}\right)^{\frac{1}{n}}
        \\\geq &\int_{\mathbb{R}^{n}_{+}\cap \Omega}\frac{1}{2^{n}n}|\nabla v(x)|^{n} \textup{d}x-\frac{2^n}{n}\int_{\mathbb{R}^{n}_{+}\cap \Omega}|\nabla u(x)|^{n}\textup{d}x
        \\&-C_{\Omega,n,k} - C_{\Omega,n}\left(\int_{\mathbb{R}^{n}_{+} \cap \Omega}|\nabla v|^{n}\right)^{\frac{1}{n}}.
    \end{split}
\end{equation*}
Since $u\in W^{1,n}_{\text{loc}}(\mathbb{R}^{n}_{+})$, we get that $J_k(v)$ is bounded from below. Thus by Ekeland variational principle, we get that there is a sequence $\{v_{j,k}\}_{j=1}^{+\infty}\subset W^{1,n}_{0}\big( \overline{\mathbb{R}^{n}_{+}}\cap \Omega  \big)$ such that
\begin{equation*}
    \begin{split}
        J_k(v_{j,k}) \to \inf J_k \quad \text{ and } \quad J'(v_{j,k}) \to 0, \quad \text{ as }j\to 0.
    \end{split}
\end{equation*}
Thus
\begin{equation*}
    \begin{split}
        &\inf J+1+o_{k}(1)\|v_{j,k}\|_{W^{1,n}_{0}\big( \overline{\mathbb{R}^{n}_{+}}\cap \Omega  \big) }
        \\\geq &\int_{\mathbb{R}^{n}_{+} \cap \Omega} \frac{n-1}{n} H_{k}\big(\nabla (v_{j,k}+u)\big)-\int_{\mathbb{R}^{n}_{+} \cap \Omega} \frac{1}{n} \int_{\mathbb{R}^{n}}|\nabla (v_{j,k}+u)(x) -y  |^{n-1}|\nabla u(x)-y| \phi_{k}(y)\textup{d}y \textup{d}x
        \\\geq &\int_{\mathbb{R}^{n}_{+} \cap \Omega} \frac{n-1}{n} H_{k}\big(\nabla (v_{j,k}+u)\big)-\int_{\mathbb{R}^{n}_{+} \cap \Omega} \frac{2^{n-1}}{n} \int_{\mathbb{R}^{n}}|\nabla v_{j,k}(x)  |^{n-1}|\nabla u(x)-y| \phi_{k}(y)\textup{d}y \textup{d}x
        \\&-\int_{\mathbb{R}^{n}_{+} \cap \Omega} \frac{2^{n-1}}{n} \int_{\mathbb{R}^{n}}|\nabla u(x)-y|^{n} \phi_{k}(y)\textup{d}y \textup{d}x
        \\\geq &\int_{\mathbb{R}^{n}_{+}\cap \Omega}\frac{n-1}{2^{n}n}|\nabla v(x)|^{n} \textup{d}x-C_{\Omega,n,u,k}\left(\int_{\mathbb{R}^{n}_{+}\cap \Omega}\frac{n-1}{2^{n}n}|\nabla v(x)|^{n} \textup{d}x\right)^{\frac{n-1}{n}}-C_{\Omega,n,u,k}.
    \end{split}
\end{equation*}
Thus $\{v_{j,k}\}_{j=1}^{+\infty}\subset W^{1,n}_{0}\big( \overline{\mathbb{R}^{n}_{+}}\cap \Omega  \big)$ is bounded in $ W^{1,n}_{0}\big( \overline{\mathbb{R}^{n}_{+}}\cap \Omega  \big)$. We can assume that $ v_{j,k} \rightharpoonup v_k$ in $W^{1,n}_{0}\big( \overline{\mathbb{R}^{n}_{+}}\cap \Omega  \big) $. Thus for all $p>0$, $v_{j,k}\to v_k $ in $L^{p}\big( \overline{\mathbb{R}^{n}_{+}}\cap \Omega  \big)$ and $L^{p}\big( \overline{\partial\mathbb{R}^{n}_{+}}\cap \Omega  \big)$. Since $v_{j,k},v $ are bounded in $ W^{1,n}_{0}\big( \overline{\mathbb{R}^{n}_{+}}\cap \Omega  \big)$ and $ v_{j,k} \rightharpoonup v_k$ in $W^{1,n}_{0}\big( \overline{\mathbb{R}^{n}_{+}}\cap \Omega  \big) $, we get that $  \langle   J_k'(v_{j,k})-J_k'(v),v_{j,k}-v_k  \rangle\to 0$. Thus
\begin{equation}\label{ApAeq1}
    \begin{split}
        \int_{\mathbb{R}^{n}_{+} \cap \Omega}\langle a^{k}\big(\nabla(v_{j,k}+u)\big) - a^{k}\big(\nabla(v_k+u)\big),\nabla v_{j,k}-\nabla v_k\rangle\to0.
    \end{split}
\end{equation}
Since $\langle |\xi|^{n-2}\xi -|\eta|^{n-2}\eta,\xi-\eta\rangle \geq C_{n}|\xi-\eta|^{n} $, where $C_n>0$ is a constant depending on $n$, we get that
\begin{equation}\label{ApAeq2}
    \begin{split}
        &\int_{\mathbb{R}^{n}_{+} \cap \Omega}\langle a^{k}\big(\nabla(v_{j,k}+u)\big) - a^{k}\big(\nabla(v_k+u)\big),\nabla v_{j,k}-\nabla v_k\rangle
        \\=&\int_{\mathbb{R}^{n}_{+} \cap \Omega}\bigg\langle \int_{\mathbb{R}^{n}}|\nabla (v_{j,k}+u )(x) -y|^{n-2}\big(\nabla (v_{j,k}+u )(x) -y\big)\phi_{k}(y)\textup{d}y
        \\&- \int_{\mathbb{R}^{n}}|\nabla (v_k+u )(x) -y|^{n-2}\big(\nabla (v_{j,k}+u )(x) -y\big)\phi_{k}(y)\textup{d}y  ,\nabla v_{j,k}(x)-\nabla v(x)\bigg\rangle\textup{d}x
        \\\geq &C_{n}\int_{\mathbb{R}^{n}_{+} \cap \Omega} \int_{\mathbb{R}^{n}}| \nabla (v_{j,k}-v_k)(x) |^{n}\phi_{k}(y)\textup{d}y \textup{d}x
        \\=&C_{n}\int_{\mathbb{R}^{n}_{+} \cap \Omega} | \nabla (v_{j,k}-v_k)(x) |^{n} \textup{d}x .
    \end{split}
\end{equation}
Combining \eqref{ApAeq1} and \eqref{ApAeq2}, we get that $v_{j,k} \to v_k$ in $W^{1,n}_{0}\big( \overline{\mathbb{R}^{n}_{+}}\cap \Omega  \big)  $. So we derive that
\begin{equation*}
    \begin{split}
        \int_{\mathbb{R}^{n}_{+} \cap \Omega} a^{k}\big(\nabla (v_k+u)\big)\nabla \eta\textup{d}x - \int_{\partial \mathbb{R}^{n}_{+} \cap \Omega}g\eta = 0,
    \end{split}
\end{equation*}
where $\eta \in W^{1,n}_{0}\big( \overline{\mathbb{R}^{n}_{+}}\cap \Omega  \big)$. That is, $ v$ is the solution of
\begin{equation}\label{ApAeq3}
    \begin{cases}
        -\operatorname{div}\left(a^k\left(\nabla (v_k+u)\right)\right)=f & \text { in } \mathbb{R}^{n}_{+} \cap\Omega, \\
a^k\left(\nabla (v_k+u))\right) \cdot \nu=g & \text { on } \partial \mathbb{R}^{n}_{+} \cap\Omega, \\
v=0 & \text { on } \mathbb{R}^{n}_{+} \cap \partial\Omega.
    \end{cases}
\end{equation}
Set $\bar{v}_{k} = v_k+u $, then $ u$ is the solution of
\begin{equation*}
    \begin{cases}
        -\operatorname{div}\left(a^k\left(\nabla \bar{v}_{k}\right)\right)=f & \text { in } \mathbb{R}^{n}_{+} \cap\Omega, \\
a^k\left(\nabla \bar{v}_{k})\right) \cdot \nu=g & \text { on } \partial \mathbb{R}^{n}_{+} \cap\Omega, \\
\bar{v}_{k}=u & \text { on } \mathbb{R}^{n}_{+} \cap \partial\Omega.
    \end{cases}
\end{equation*}
Now, we prove that $\|\bar{v}_{k}\|_{W^{1,n}_{0}\big( \overline{\mathbb{R}^{n}_{+}}\cap \Omega  \big)}$ is uniformly bounded as $k\to+\infty$. Since $u\in C^{1,\alpha}_{\text{loc}}(\overline{\mathbb{R}^n_{+}})$, by testing \eqref{ApAeq3} by $v_k$, we get that
\begin{equation}\label{ApAeq4}
    \begin{split}
       &\int_{\mathbb{R}^{n}_{+} \cap \Omega} \frac{1}{n} H_{k}\big(\nabla (v_{k}+u)\big)
       \\=&\int_{\mathbb{R}^{n}_{+} \cap \Omega} \frac{1}{n}\int_{\mathbb{R}^{n}}|\nabla (v_{k}+u)(x) -y  |^{n-1}|(\nabla u(x)-y) \phi_{k}(y)\textup{d}y \textup{d}x
       \\&+\int_{\mathbb{R}^{n}_{+} \cap \Omega} fv_k +\int_{\partial\mathbb{R}^{n}_{+} \cap \Omega}gv_k
       \\\leq &\int_{\mathbb{R}^{n}_{+} \cap \Omega} \frac{2^n}{n}\int_{\mathbb{R}^{n}}|\nabla v_{k}(x)   |^{n-1}|\nabla u(x)-y| \phi_{k}(y)\textup{d}y \textup{d}x
       \\&+\int_{\mathbb{R}^{n}_{+} \cap \Omega} \frac{2^n}{n}\int_{\mathbb{R}^{n}}|\nabla u(x)-y|^n \phi_{k}(y)\textup{d}y \textup{d}x
       \\&+C_{\Omega,n}\left(\int_{\mathbb{R}^{n}_{+} \cap \Omega}|\nabla v_k|^n\textup{d}x\right)^{\frac{1}{n}}
       \\\leq &\int_{\mathbb{R}^{n}_{+} \cap \Omega} \frac{2^n}{n}|\nabla v_{k}(x)   |^{n-1}(|\nabla u| +o_{k}(1))\textup{d}x
       \\&+\int_{\mathbb{R}^{n}_{+} \cap \Omega} o_{k}(1)\textup{d}x+C_{\Omega,n}\left(\int_{\mathbb{R}^{n}_{+} \cap \Omega}|\nabla v_k|^n\textup{d}x\right)^{\frac{1}{n}}
       \\\leq & \big(C_{\Omega,n,u}+o_k(1)\big) \left(\int_{\mathbb{R}^{n}_{+} \cap \Omega}|\nabla v_k|^n\textup{d}x\right)^{\frac{n-1}{n}}+C_{\Omega,n}\left(\int_{\mathbb{R}^{n}_{+} \cap \Omega}|\nabla v_k|^n\textup{d}x\right)^{\frac{1}{n}}+o_k(1).
    \end{split}
\end{equation}
Similar to \eqref{ApAeq4}, we get that
\begin{equation}\label{ApAeq5}
      \int_{\mathbb{R}^{n}_{+} \cap \Omega}|\nabla v_k|^n\textup{d}x  \leq \int_{\mathbb{R}^{n}_{+} \cap \Omega}  H_{k}\big(\nabla (v_{k}+u)\big)+\int_{\mathbb{R}^{n}_{+} \cap \Omega}|\nabla u|^n\textup{d}x+o_k(1).
\end{equation}
Combining \eqref{ApAeq4} and \eqref{ApAeq5}, we get that $\|v_{k}\|_{W^{1,n}_{0}\big( \overline{\mathbb{R}^{n}_{+}}\cap \Omega  \big)}$ is uniformly bounded as $k\to+\infty$. So $\|\bar{v}_{k}\|_{W^{1,n}_{0}\big( \overline{\mathbb{R}^{n}_{+}}\cap \Omega  \big)}$ is uniformly bounded as $k\to+\infty$.
\end{proof}

\begin{lem}\label{Apale2}
    Assume that $\Omega\subset \mathbb{R}^{n}_{+}$ is a bounded domain, $f,g\in L^{\infty}(\mathbb{R}^{n}_{+} )$ and $ u \in W^{1,n}_{\text{loc}}(\mathbb{R}^{n}_{+} ) $. Then the following equation
   \begin{equation*}
    \begin{cases}
        -\operatorname{div}\left(\left|\frac{1}{k^2}+|\nabla \bar{v}_{k} |^2\right|^{\frac{n-2}{2}}\nabla \bar{v}_{k}\right)=f(x) & \text { in } \mathbb{R}^{n}_{+} \cap\Omega, \\
\left|\frac{1}{k^2}+|\nabla \bar{v}_{k}|^2 \right|^{\frac{n-2}{2}}\nabla \bar{v}_{k} \cdot \nu=g(x) & \text { on } \partial \mathbb{R}^{n}_{+} \cap\Omega, \\
\bar{v}_{k}=u & \text { on } \mathbb{R}^{n}_{+} \cap \partial\Omega,
    \end{cases}
\end{equation*}has a solution $\bar{v}_{k}\in  W^{1,n}_{0}\big( \overline{\mathbb{R}^{n}_{+}}\cap \Omega  \big)$. Moreover, $\|\bar{v}_{k}\|_{W^{1,n}_{0}\big( \overline{\mathbb{R}^{n}_{+}}\cap \Omega  \big)}$ is uniformly bounded as $k\to+\infty$.
\end{lem}
\begin{proof}
    We consider the minimizer $\bar{v}_{k}$ of the variation problem
\begin{equation*}
    \begin{split}
        \min_v\bigg\{J_k(v):=\int_{\mathbb{R}^{n}_{+} \cap \Omega} \left(\frac{1}{n}\left|\frac{1}{k^2}+|\nabla (v_k+u)|^2\right|^{\frac{n}{2}} -f v\right)\mathrm{d} x-\int_{\partial \mathbb{R}^{n}_{+} \cap \Omega} g v \textup{d}\sigma,
\\ v\in W^{1,n}_{0}\left(\overline{\mathbb{R}^{n}_{+}}\cap \Omega \right)\bigg\}.
    \end{split}
\end{equation*}
We can see that $ J_k(v) \in C^{1}\bigg(W^{1,n}_{0}\big( \overline{\mathbb{R}^{n}_{+}}\cap \Omega  \big),\mathbb{R}\bigg)$. Since $v\big|_{\partial \Omega\cap \mathbb{R}^{n}_{+} =0}$, by Sobolev inequality, Sobolev trace inequality and the fact that $|a+b|^{n}\leq 2^{n} ( |a|^n +|b|^n) $, we get that
\begin{equation*}
    \begin{split}
        J_k(v) \geq &\int_{\mathbb{R}^{n}_{+} \cap \Omega} \frac{1}{n} \left|\frac{1}{k^2}+|\nabla (v_k+u)|^2\right|^{\frac{n}{2}} - C_{\Omega,n}\left(\int_{\mathbb{R}^{n}_{+} \cap \Omega}|\nabla v|^{n}\right)^{\frac{1}{n}}
        \\\geq &\int_{\mathbb{R}^{n}_{+}\cap \Omega}\frac{1}{2^{n}n}|\nabla v(x)|^{n} \textup{d}x-\frac{1}{n}\int_{\mathbb{R}^{n}_{+}\cap \Omega}|\nabla u(x)|^{n}\textup{d}x
        \\&-C_{\Omega,n,k} - C_{\Omega,n}\left(\int_{\mathbb{R}^{n}_{+} \cap \Omega}|\nabla v|^{n}\right)^{\frac{1}{n}}.
    \end{split}
\end{equation*}
Since $u\in W^{1,n}_{\text{loc}}(\mathbb{R}^{n}_{+})$, we get that $J_k(v)$ is bounded from below. Thus by Ekeland variational principle, we get that there is a sequence $\{v_{j,k}\}_{j=1}^{+\infty}\subset W^{1,n}_{0}\big( \overline{\mathbb{R}^{n}_{+}}\cap \Omega  \big)$ such that
\begin{equation*}
    \begin{split}
        J_k(v_{j,k}) \to \inf J \quad \text{ and } \quad J'(v_{j,k}) \to 0, \quad \text{ as }k\to 0.
    \end{split}
\end{equation*}
Thus
\begin{equation*}
    \begin{split}
        &\inf J+1+o_{k}(1)\|v_{j,k}\|_{W^{1,n}_{0}\big( \overline{\mathbb{R}^{n}_{+}}\cap \Omega  \big) }
        \\\geq &\int_{\mathbb{R}^{n}_{+} \cap \Omega} \frac{n-1}{n} \left|\frac{1}{k^2}  +|\nabla (v_{j,k}+u)|^{2} \right|^{\frac{n}{2}}-\int_{\mathbb{R}^{n}_{+} \cap \Omega} \frac{1}{n}\left|\frac{1}{k^2}  +|\nabla (v_{j,k}+u)|^{2} \right|^{\frac{n-2}{2}}\left(\frac{1}{k^2}+|\nabla u|^2\right)  \textup{d}x
        \\\geq &\int_{\mathbb{R}^{n}_{+}\cap \Omega}\frac{n-1}{2^{n}n}|\nabla v(x)|^{n} \textup{d}x-C_{\Omega,n,u,k}\left(\int_{\mathbb{R}^{n}_{+}\cap \Omega}\frac{n-1}{2^{n}n}|\nabla v(x)|^{n} \textup{d}x\right)^{\frac{n-1}{n}}-C_{\Omega,n,u,k}.
    \end{split}
\end{equation*}
Thus $\{v_{j,k}\}_{j=1}^{+\infty}\subset W^{1,n}_{0}\big( \overline{\mathbb{R}^{n}_{+}}\cap \Omega  \big)$ is bounded in $ W^{1,n}_{0}\big( \overline{\mathbb{R}^{n}_{+}}\cap \Omega  \big)$. Then we can assume that $ v_{j,k} \rightharpoonup v_k$ in $W^{1,n}_{0}\big( \overline{\mathbb{R}^{n}_{+}}\cap \Omega  \big) $. Thus for all $p>0$, $v_{j,k}\to v_k $ in $L^{p}\big( \overline{\mathbb{R}^{n}_{+}}\cap \Omega  \big)$ and $L^{p}\big( \overline{\partial\mathbb{R}^{n}_{+}}\cap \Omega  \big)$. Since $v_{j,k},v $ are bounded in $ W^{1,n}_{0}\big( \overline{\mathbb{R}^{n}_{+}}\cap \Omega  \big)$ and $ v_{j,k} \rightharpoonup v_k$ in $W^{1,n}_{0}\big( \overline{\mathbb{R}^{n}_{+}}\cap \Omega  \big) $, we get that $  \langle   J_k'(v_{j,k})-J_k'(v),v_{j,k}-v_k  \rangle\to 0$. Thus
\begin{equation}\label{reApAeq1}
    \begin{split}
        \int_{\mathbb{R}^{n}_{+} \cap \Omega}\left\langle \left|\frac{1}{k^2}+|\nabla (v_{j,k}+u)  |^2\right|^{\frac{n-2}{2}}\nabla (v_{j,k}+u)  - \left|\frac{1}{k^2}+|\nabla (v_k+u)  |^2\right|^{\frac{n-2}{2}}\nabla (v_k+u),\nabla v_{j,k}-\nabla v_k\right\rangle\to0.
    \end{split}
\end{equation}
Since $\left\langle\left|\frac{1}{k^2}+ |\xi|^2\right|^{\frac{n-2}{2}}\xi -\left|\frac{1}{k^2}+|\eta|^2\right|^{\frac{n-2}{2}}\eta,\xi-\eta\right\rangle \geq C_{n}|\xi-\eta|^{n} $, where $C_{n}>0$ is a constant depending on $n$, we get that
\begin{equation}\label{reApAeq2}
    \begin{split}
        &\int_{\mathbb{R}^{n}_{+} \cap \Omega}\left\langle \left|\frac{1}{k^2}+|\nabla (v_{j,k}+u)  |^2\right|^{\frac{n-2}{2}}\nabla (v_{j,k}+u) - \left|\frac{1}{k^2}+|\nabla (v_k+u)  |^2\right|^{\frac{n-2}{2}}\nabla (v_k+u),\nabla v_{j,k}-\nabla v_k\right\rangle
        \\\geq &C_{n}\int_{\mathbb{R}^{n}_{+} \cap \Omega} | \nabla (v_{j,k}-v_k)(x) |^{n} \textup{d}x .
    \end{split}
\end{equation}
Combining \eqref{reApAeq1} and \eqref{reApAeq2}, we get that $v_{j,k} \to v_k$ in $W^{1,n}_{0}\big( \overline{\mathbb{R}^{n}_{+}}\cap \Omega  \big)  $. So we get that
\begin{equation*}
    \begin{split}
        \int_{\mathbb{R}^{n}_{+} \cap \Omega} \left|\frac{1}{k^2} +|\nabla (v_k+u)|^2\right|^{\frac{n-2}{2}}\nabla (v_k+u) \cdot\nabla \eta\textup{d}x - \int_{\partial \mathbb{R}^{n}_{+} \cap \Omega}g\eta = 0,
    \end{split}
\end{equation*}
where $\eta \in W^{1,n}_{0}\big( \overline{\mathbb{R}^{n}_{+}}\cap \Omega  \big)$. That is, $ v$ is the solution of
\begin{equation}\label{reApAeq3}
    \begin{cases}
        -\operatorname{div}\left(\left|\frac{1}{k^2} +|\nabla (v_k+u)|^2\right|^{\frac{n-2}{2}}\nabla (v_k+u)\right)=f & \text { in } \mathbb{R}^{n}_{+} \cap\Omega, \\
\left|\frac{1}{k^2} +|\nabla (v_k+u)|^2\right|^{\frac{n-2}{2}}\nabla (v_k+u) \cdot \nu=g & \text { on } \partial \mathbb{R}^{n}_{+} \cap\Omega, \\
v=0 & \text { on } \mathbb{R}^{n}_{+} \cap \partial\Omega.
    \end{cases}
\end{equation}
Set $\bar{v}_{k} = v_k+u $, then $ u$ is the solution of
\begin{equation*}
    \begin{cases}
        -\operatorname{div}\left(\left|\frac{1}{k^2} +|\nabla \bar{v}_{k}|^2\right|^{\frac{n-2}{2}}\nabla \bar{v}_{k}\right)=f & \text { in } \mathbb{R}^{n}_{+} \cap\Omega, \\
\left|\frac{1}{k^2} +|\nabla \bar{v}_{k}|^2\right|^{\frac{n-2}{2}}\nabla \bar{v}_{k} \cdot \nu=g & \text { on } \partial \mathbb{R}^{n}_{+} \cap\Omega, \\
\bar{v}_{k}=u & \text { on } \mathbb{R}^{n}_{+} \cap \partial\Omega.
    \end{cases}
\end{equation*}

Now, we prove that $\|\bar{v}_{k}\|_{W^{1,n}_{0}\big( \overline{\mathbb{R}^{n}_{+}}\cap \Omega  \big)}$ is uniformly bounded as $k\to+\infty$. Since $u\in C^{1,\alpha}_{\text{loc}}(\overline{\mathbb{R}^n_{+}})$, by testing \eqref{reApAeq3} by $v_k$, similar to the proof of Lemma \ref{Apale1}, we get that
\begin{equation*}
    \begin{split}
         \int_{\mathbb{R}^{n}_{+} \cap \Omega}|\nabla v_k|^n\textup{d}x\leq& \big(C_{\Omega,n,u}+o_k(1)\big) \left(\int_{\mathbb{R}^{n}_{+} \cap \Omega}|\nabla v_k|^n\textup{d}x\right)^{\frac{n-1}{n}}+C_{\Omega,n}\left(\int_{\mathbb{R}^{n}_{+} \cap \Omega}|\nabla v_k|^n\textup{d}x\right)^{\frac{1}{n}} \\ &+o_k(1)+C_{\Omega,n,u}.
    \end{split}
\end{equation*}
This implies that $\|\bar{v}_{k}\|_{W^{1,n}_{0}\big( \overline{\mathbb{R}^{n}_{+}}\cap \Omega  \big)}$ is uniformly bounded as $k\to+\infty$.
\end{proof}

\section{Proof of Lemma \ref{rele:2}}
In this section, we will give the proof of Lemma \ref{rele:2}.

\noindent{\bf{Proof of Lemma \ref{rele:2}.}}
 We can prove Lemma \ref{rele:2} through a similar way as the proof of \cite[Lemma 2.3]{CL2}. 
It is enough to prove the result for the case $r=1$. Set $k=(s+s^{n-1}+\|h\|_{C^{1}(\overline{\Omega})})^{\frac{1}{n-1}}+\|f\|_{L^{\frac{n}{n-\varepsilon}}(B_{2}(x_{0})\cap\Omega)}^{\frac{1}{n-1}}+\|g\|_{L^{\frac{n-1}{n-1-\varepsilon}}(B_{2}(x_{0})\cap\Omega)}^{\frac{1}{n-1}}$.
    For fixed numbers $q \geqslant 1$ and $l>k$, we define the functions
$$
F(x)=\left\{\begin{array}{lll}
x^q & \text { if } & k \leqslant x \leqslant l, \\
q l^{q-1} x-(q-1) l^q & \text { if } & l \leqslant x,
\end{array}\right.
$$
and
$$
G(x)=\operatorname{sign} x \cdot\left\{F(|x|+k) F^{\prime}(|x|+k)^{n-1}-q^{n-1} k^\beta\right\}, \quad-\infty<x<\infty,
$$
where $q$ and $\beta$ are related by $n q=n+\beta-1$.

 \noindent{\bf{Proof of (i).}}
  Let $\eta$ be a non-negative smooth function with compact support in $ B_{3}(x_0)$, $ \phi(x)=\eta^{n}G(w)$, $w=u-h$ and $\overline{w}=|w|+k$. We know that, for any real value $\beta>0 $, $ \phi(x)$ is a testing function of \eqref{rehareq1}. Thus we get that
    \begin{equation*}
        \begin{split}
            \int_{\Omega} \langle\mathbf{a}(x,\nabla u),\nabla\phi(x)    \rangle
            =&\int_{\Omega} f(x)\phi(x).
        \end{split}
    \end{equation*}
    From \eqref{mixeq1} and \eqref{mixeq3}, we get that
    \begin{equation*}
        \begin{split}
            &|\mathbf{a}(x,\nabla u) -\mathbf{a}(x,\nabla \overline{w})|1_{\{u(x)>h(x)\}}+|\mathbf{a}(x,\nabla u) +\mathbf{a}(x,\nabla \overline{w})|1_{\{u(x)<h(x)\}}\\\leq &C_{d,e}\bigg(|\mathbf{a}(x,\nabla h)|+\big(s+ | \nabla \overline{w}|\big)^{n-2}| \nabla h|+ | \nabla h|^{n-1}\bigg)
            \\\leq &C_{d,e}\bigg(s+s^{n-1}+ | \nabla \overline{w}|^{n-2}| \nabla h|+ | \nabla h|^{n-1}\bigg).
        \end{split}
    \end{equation*}
 Let $v=\overline{w}^{q}$ and $ q=\frac{n+\beta-1}{n}$. Through a similar computation as in the proof of Theorem 2 in \cite{S}, we get that
\begin{equation*}
    \begin{split}
        &\int_{\Omega} |\eta F'(\overline{w})\nabla (\overline{w})|^{n}
        \\\leq &C_{d,e,n}\int_{\Omega} \bigg(|f\eta^{n}G(w)|+(s+|\nabla h|^{n-1}+|\nabla \overline{w}|^{n-1})|\nabla\eta|\eta^{n-1}G(w)
        \\&+\big(s+s^{n-1}+ | \nabla \overline{w}|^{n-2}| \nabla h|+ | \nabla h|^{n-1} \big)q^{-1}\beta |F'(\overline{w})^{n}||\nabla \overline{w}|\eta^n\bigg)
        \\\leq &C_{d,e,n}\int_{\Omega} \bigg(|f\eta^{n}F(\overline{w})F'(\overline{w})^{n-1}|+(k^{n-1}+|\nabla \overline{w}|^{n-1})|\nabla\eta|\eta^{n-1}F(\overline{w})|F'(\overline{w})^{n-1}|
        \\&+\big(k^{n-1}+ C_{h}| \nabla \overline{w}|^{n-2} \big)q^{-1}\beta |F'(\overline{w})^{n}||\nabla \overline{w}|\eta^n\bigg).
    \end{split}
\end{equation*}
Thus
\begin{equation}\label{rehareq4}
    \begin{split}
        &\int_{\Omega}|\eta\nabla v|^{n}
        \\\leq &C_{d,e,n}\bigg(\int_{\Omega} \left|\frac{q^{n-1}f}{k^{n-1} }\right||\eta v|^{n}+(1+\beta)\int_{\Omega}|v\nabla\eta ||\eta\nabla v|^{n-1}
        \\&+q^{n-1}\int_{\Omega}|v\nabla\eta ||\eta v|^{n-1} +q^{n-1}\beta\int_{\Omega}|\nabla v\eta ||\eta v|^{n-1}\bigg).
    \end{split}
\end{equation}
Set $\alpha_1 =\frac{n}{1+\frac{\epsilon}{2n}} $, by H\"{o}lder inequality and Sobolev inequality, we get that
\begin{equation}\label{rehareq5}
    \int_{\Omega}|v\nabla\eta ||\eta\nabla v|^{n-1} \leq \| v\nabla\eta \|_{L^{n}(\Omega)}\| \eta\nabla v  \|_{ L^{n}(\Omega)}^{n-1},
\end{equation}
\begin{equation}\label{rehareq6}
    \begin{split}
        \int_{\Omega} \left|\frac{f}{k^{n-1} }\right||\eta v|^{n} =&\int_{\Omega} \left|\frac{f}{k^{n-1} }\right||\eta v|^{\frac{\epsilon}{2}}|\eta v|^{n-\frac{\epsilon}{2}}\\
        \leq &k^{1-n}\|f\|_{L^{\frac{n}{n-\epsilon}}(\Omega)}\|\eta v  \|_{L^{\alpha_1} (\Omega  )  }^{\frac{\epsilon}{2}}\|\eta v  \|^{n-\frac{\epsilon}{2}}_{L^{\alpha_1^{*}} (\Omega  )}\\
        \leq &C_{\epsilon,n}\|\eta v  \|_{L^{n} (\Omega  )  }^{\frac{\epsilon}{2}}(   \|v\nabla \eta \|^{n-\frac{\epsilon}{2}}_{L^{n} (\Omega  )  }  + \|  \eta \nabla v\|^{n-\frac{\epsilon}{2}}_{L^{n} (\Omega  )  } ),
    \end{split}
\end{equation}
\begin{equation}\label{rehareq7}
    \begin{split}
       \int_{\Omega}|v\nabla\eta ||\eta v|^{n-1} \leq & C_{n}\|v\nabla\eta \|_{L^{n} (\Omega  )}\|\eta v  \|^{n-1}_{L^{\alpha_1^{*}} (\Omega  )} \\
       \leq &C_{n}\|v\nabla\eta \|_{L^{n} (\Omega  )}(\|\nabla\eta v  \|^{n-1}_{n} + \|\nabla v \eta   \|^{n-1}_{L^{n} (\Omega  )} ),
    \end{split}
\end{equation}
and
\begin{equation}\label{rerehareq7}
    \begin{split}
       \int_{\Omega}|\nabla v\eta ||\eta v|^{n-1} \leq &\| v\eta \|^{n-1}_{L^{n}(\Omega)}\| \eta\nabla v  \|_{ L^{n}(\Omega)},
    \end{split}
\end{equation}
where  $\alpha_1^{*}=\frac{\alpha_1 n}{n-\alpha_1} $.
From \eqref{rehareq4}, \eqref{rehareq5}, \eqref{rehareq6}, \eqref{rehareq7}  and \eqref{rerehareq7}, we get that
\begin{equation*}
    \begin{split}
        &\int_{\Omega}|\eta\nabla v|^{n}
        \\\leq &C_{d,e,n,h,\epsilon}\bigg(  (1+\beta)\| v\nabla\eta \|_{L^{n}(\Omega)}\| \eta\nabla v  \|_{ L^{n}(\Omega)}^{n-1}+q^{n-1}\|\eta v  \|_{L^{n} (\Omega  )  }^{\frac{\epsilon}{2}}(   \|v\nabla \eta \|^{n-\frac{\epsilon}{2}}_{L^{n} (\Omega  )  }  + \|  \eta \nabla v\|^{n-\frac{\epsilon}{2}}_{L^{n} (\Omega  )  } )\\
        &+q^{n-1}\|v\nabla\eta \|_{L^{n} (\Omega  )}(\|\nabla\eta v  \|^{n-1}_{n} + \|\nabla v \eta   \|^{n-1}_{L^{n} (\Omega  )} ) +q^{n-1}\| v\eta \|^{n-1}_{L^{n}(\Omega)}\| \eta\nabla v  \|_{ L^{n}(\Omega)}\bigg).
    \end{split}
\end{equation*}
By Young inequality, we get that
\begin{equation*}
    \begin{split}
        \int_{\Omega}|\eta\nabla v|^{n}\leq &C_{d,e,n,h,\epsilon}\bigg(  (1+\beta +q)^{n^2} \| v\nabla\eta \|_{L^{n}(\Omega)}^{n} +\epsilon (q)^{\frac{2n^{2}}{\epsilon}} \|\eta v  \|_{L^{n} (\Omega  )  }^{n}
         \bigg).
    \end{split}
\end{equation*}
Thus
\begin{equation*}
    \begin{split}
       \left(\int_{\Omega}|\eta\nabla v|^{n}\right)^{\frac{1}{n}}\leq &C_{d,e,n,h,\epsilon}  (1+\beta +q)^{\frac{2n}{\epsilon}}(\|\eta v  \|_{L^{n} (\Omega  )  } +\| v\nabla\eta \|_{L^{n}(\Omega)}  ) .
    \end{split}
\end{equation*}
 By Sobolev inequality, we get that
 \begin{equation}\label{rehareq8}
     \begin{split}
   \|\eta v  \|_{L^{\alpha_1^{*}} (\Omega  )  }    \leq C_{d,e,n,h,\epsilon}  (1+\beta +q)^{\frac{2n}{\epsilon}}(\|\eta v  \|_{L^{n} (\Omega  )  } +\| v\nabla\eta \|_{L^{n}(\Omega)}  ).
     \end{split}
 \end{equation}
Set $ m>m'>0$, and let $\eta$ be the smooth function such that $\eta \equiv 1$ in $B_{m'}(x_0)$, $\eta \equiv 0$ in $B_{m}^{c}(x_0)$ and $|\nabla\eta| \leq \frac{C}{m-m'}$. From \eqref{rehareq8}, we deduce that
\begin{equation*}
    \begin{split}
        &\| v  \|_{L^{\alpha_1^{*}} (\Omega\cap B_{m'}(x_0)  )  }
        \\\leq&  C_{d,e,n,h,\epsilon} (1+\beta +q)^{\frac{2n}{\epsilon}}(m-m')\| v  \|_{L^{n} (\Omega \cap B_{m}(x_0) )  }.
    \end{split}
\end{equation*}
Thus by classical Moser iteration argument, we get our desired result.

\noindent{\bf{Proof of (ii).}}
  Let $\eta$ be a non-negative smooth function with compact support in $ B_{3}(x_0)$, $ \phi(x)=\eta^{n}G(u)$ and $\overline{u}=|u|+k$. We know that, for any $\beta>0 $, $ \phi(x)$ is a testing function of \eqref{rehareq1}. Thus we get that
    \begin{equation*}
        \begin{split}
            \int_{\Omega} \langle\mathbf{a}(x,\nabla u),\nabla\phi(x)    \rangle
            =\int_{\Omega} f(x)\phi(x)+\int_{\Gamma_2} g(x)\phi(x).
        \end{split}
    \end{equation*}
 Let $v=\overline{u}^{q}$ and $ q=\frac{n+\beta-1}{n}$. Through a similar computation as the proof of Theorem 2 in \cite{S}, we get that
\begin{equation*}
    \begin{split}
        &\int_{\Omega} |\eta F'(\overline{u})\nabla (\overline{u})|^{n}
        \\\leq &C_{d,e,n}\int_{\Omega} \bigg(|f\eta^{n}G(w)|+(s+|\nabla \overline{u}|^{n-1})|\nabla\eta|\eta^{n-1}G(w)
        \bigg)
       +\int_{\Gamma_2} g(x)G(w)
        \\\leq &C_{d,e,n}\int_{\Omega} \bigg(|f\eta^{n}F(\overline{u})F'(\overline{u})^{n-1}|+(k^{n-1}+|\nabla \overline{u}|^{n-1})|\nabla\eta|\eta^{n-1}F(\overline{u})|F'(\overline{u})^{n-1}|\bigg)
        \\&+\int_{\Gamma_2} g(x)F(\overline{u})|F'(\overline{u})^{n-1}|.
    \end{split}
\end{equation*}
Thus
\begin{equation}\label{rehareq11}
    \begin{split}
        &\int_{\Omega}|\eta\nabla v|^{n}
        \\\leq &C_{d,e,n}\bigg(\int_{\Omega} \left|\frac{q^{n-1}f}{k^{n-1} }\right||\eta v|^{n}+\int_{\Omega}|v\nabla\eta ||\eta\nabla v|^{n-1}
        \\&+q^{n-1}\int_{\Omega}|v\nabla\eta ||\eta v|^{n-1} +\int_{\Gamma_2} \left|\frac{q^{n-1}g}{k^{n-1} }\right||\eta v|^{n}\bigg).
    \end{split}
\end{equation}
Set $\alpha_1 =\frac{n}{1+\frac{\epsilon}{2n}} $ and $\alpha_2 =\frac{n^2-\frac{\epsilon}{4}}{n} $, by H\"{o}lder's inequality, Lemma \ref{relem:2.1} and Sobolev trace inequality, we get that
\begin{equation}\label{rehareq12}
    \int_{\Omega}|v\nabla\eta ||\eta\nabla v|^{n-1} \leq \| v\nabla\eta \|_{L^{n}(\Omega)}\| \eta\nabla v  \|_{ L^{n}(\Omega)}^{n-1},
\end{equation}
\begin{equation}\label{rehareq13}
    \begin{split}
        \int_{\Omega} \left|\frac{f}{k^{n-1} }\right||\eta v|^{n} =&\int_{\Omega} \left|\frac{f}{k^{n-1} }\right||\eta v|^{\frac{\epsilon}{2}}|\eta v|^{n-\frac{\epsilon}{2}}\\
        \leq &k^{1-n}\|f\|_{L^{\frac{n}{n-\epsilon}}(\Omega)}\|\eta v  \|_{L^{\alpha_1} (\Omega  )  }^{\frac{\epsilon}{2}}\|\eta v  \|^{n-\frac{\epsilon}{2}}_{L^{\alpha_1^{*}} (\Omega  )}\\
        \leq &C_{\epsilon,n}\|\eta v  \|_{L^{n} (\Omega  )  }^{\frac{\epsilon}{2}}(   \|v\nabla \eta \|^{n-\frac{\epsilon}{2}}_{L^{n} (\Omega  )  }  + \|  \eta \nabla v\|^{n-\frac{\epsilon}{2}}_{L^{n} (\Omega  )  } ),
    \end{split}
\end{equation}
\begin{equation}\label{rehareq14}
    \begin{split}
       \int_{\Omega}|v\nabla\eta ||\eta v|^{n-1} \leq & C_{n}\|v\nabla\eta \|_{L^{n} (\Omega  )}\|\eta v  \|^{n-1}_{L^{\alpha_1^{*}} (\Omega  )} \\
       \leq &C_{n}\|v\nabla\eta \|_{L^{n} (\Omega  )}(\|\nabla\eta v  \|^{n-1}_{n} + \|\nabla v \eta   \|^{n-1}_{L^{n} (\Omega  )} ),
    \end{split}
\end{equation}
and
\begin{equation}\label{rehareq15}
    \begin{split}
       \int_{\Gamma_2}\left|\frac{g}{k^{n-1} }\right||\eta v|^{n} =&\int_{\Gamma_2}\left|\frac{g}{k^{n-1} }\right||\eta v|^{\frac{\epsilon}{4}}|\eta v|^{n-\frac{\epsilon}{4}} \\
       \leq &\|g\|_{L^{\frac{n-1}{n-1-\frac{\epsilon}{4}}}(  \Gamma_2     )  }\|\eta v  \|_{L^{\alpha_2} (\Gamma_2  )  }^{\frac{\epsilon}{4}}\|\eta v  \|^{n-\frac{\epsilon}{4}}_{L^{\alpha_{2,*}} (\Gamma_2  )}\\
       \leq &C_{\epsilon,n}\|\eta v  \|_{L^{n} (\Gamma_2  )  }^{\frac{\epsilon}{4}}(   \|v\nabla \eta \|^{n-\frac{\epsilon}{4}}_{L^{n} (\Omega  )  }  + \|  \eta \nabla v\|^{n-\frac{\epsilon}{4}}_{L^{n} (\Omega  )  } ),
    \end{split}
\end{equation}
where  $\alpha_1^{*}=\frac{\alpha_1 n}{n-\alpha_1} $  and $\alpha_{2,*}=\frac{\alpha_2(n-1)}{n-\alpha_2} $.
From \eqref{rehareq11}, \eqref{rehareq12}, \eqref{rehareq13}, \eqref{rehareq14} and \eqref{rehareq15}, we get that
\begin{equation*}
    \begin{split}
        &\int_{\Omega}|\eta\nabla v|^{n}
        \\\leq &C_{d,e,n,\epsilon}\bigg(  (1+\beta)\| v\nabla\eta \|_{L^{n}(\Omega)}\| \eta\nabla v  \|_{ L^{n}(\Omega)}^{n-1}+q^{n-1}\|\eta v  \|_{L^{n} (\Omega  )  }^{\frac{\epsilon}{2}}(   \|v\nabla \eta \|^{n-\frac{\epsilon}{2}}_{L^{n} (\Omega  )  }  + \|  \eta \nabla v\|^{n-\frac{\epsilon}{2}}_{L^{n} (\Omega  )  } )\\
        &+q^{n-1}\|v\nabla\eta \|_{L^{n} (\Omega  )}(\|\nabla\eta v  \|^{n-1}_{L^{n} (\Omega  ) } + \|\nabla v \eta   \|^{n-1}_{L^{n} (\Omega  )} ) +q^{n-1}\|\eta v  \|_{L^{n} (\Gamma_2  )  }^{\frac{\epsilon}{4}}(   \|v\nabla \eta \|^{n-\frac{\epsilon}{4}}_{L^{n} (\Omega  )  }  + \|  \eta \nabla v\|^{n-\frac{\epsilon}{4}}_{L^{n} (\Omega  )  } )\bigg).
    \end{split}
\end{equation*}
By Young inequality, we get that
\begin{equation*}
    \begin{split}
        \int_{\Omega}|\eta\nabla v|^{n}\leq &C_{d,e,n,\epsilon}\bigg(  (1+\beta +q)^{n^2} \| v\nabla\eta \|_{L^{n}(\Omega)}^{n} +\epsilon\big((q^{n})^{\frac{4n}{\epsilon}}+(q^{n})^{\frac{2n}{\epsilon}} \big)(\|\eta v  \|_{L^{n} (\Gamma_2  )  }^{n} +\|\eta v  \|_{L^{n} (\Omega)  }^{n}   )
         \\&+q^n\|v\nabla \eta \|^{n}_{L^{n} (\Omega  )  }\bigg).
    \end{split}
\end{equation*}
Thus
\begin{equation*}
    \begin{split}
       \left(\int_{\Omega}|\eta\nabla v|^{n} \right)^{\frac{1}{n}}\leq &C_{d,e,n,\epsilon}  (1+\beta +q)^{\frac{4n}{\epsilon}}(\|\eta v  \|_{L^{n} (\Omega  )  } +\| v\nabla\eta \|_{L^{n}(\Omega)}+\|\eta v  \|_{L^{n} (\Gamma_2  )  }^{n}  ) .
    \end{split}
\end{equation*}
By Lemma \ref{relem:2.1} and Sobolev trace inequality, we get that
 \begin{equation*}
     \begin{split}
    \|\eta v  \|_{L^{\alpha_{2,*}} (\Gamma_2  )  } +\|\eta v  \|_{L^{\alpha_1^{*}} (\Omega  )  }    \leq  C_{d,e,n,\epsilon}  (1+\beta +q)^{\frac{4n}{\epsilon}}(\|\eta v  \|_{L^{n} (\Gamma_2  )  } +\|\eta v  \|_{L^{n} (\Omega  )  } +\| v\nabla\eta \|_{L^{n}(\Omega)}  ).
     \end{split}
 \end{equation*}
 Since $ \alpha_{2,*}\geq \alpha_1^{*}$, we get that
 \begin{equation}\label{rehareq16}
 \begin{split}
           &\|\eta v  \|_{L^{\alpha_1^{*}} (\Gamma_2  )  } +\|\eta v  \|_{L^{\alpha_1^{*}} (\Omega  )  }
           \\\leq & C_{d,e,n,\epsilon}  (1+\beta +q)^{\frac{4n}{\epsilon}}(\|\eta v  \|_{L^{n} (\Gamma_2  )  } +\|\eta v  \|_{L^{n} (\Omega  )  } +\| v\nabla\eta \|_{L^{n}(\Omega)}  ).
 \end{split}
 \end{equation}
Set $ m>m'>0$, and let $\eta$ be the smooth function such that $\eta \equiv 1$ in $B_{m'}(x_0)$, $\eta \equiv 0$ in $B_{m}^{c}(x_0)$ and $|\nabla\eta| \leq \frac{C}{m-m'}$. From \eqref{rehareq16}, we deduce that
\begin{equation*}
    \begin{split}
        &\| v  \|_{L^{\alpha_1^{*}} (\Gamma_2\cap B_{m'}(x_0)  )  } +\| v  \|_{L^{\alpha_1^{*}} (\Omega\cap B_{m'}(x_0)  )  }
        \\\leq&  C_{d,e,n,\epsilon}  (1+\beta +q)^{\frac{4n}{\epsilon}}(m-m')(\| v  \|_{L^{n} (\Gamma_2\cap B_{m}(x_0)  )  } +\| v  \|_{L^{n} (\Omega \cap B_{m}(x_0) )  } ).
    \end{split}
\end{equation*}
Thus by classical Moser iteration argument, we get our desired result.

\noindent{\bf{Proof of (iii)}}
  Let $\eta$ be a non-negative smooth function with compact support in $ B_{3}(x_0)$, $ \phi(x)=\eta^{n}G(w)$, $w=u-h$ and $\overline{w}=|w|+k$. We know that for any $\beta>0 $, $ \phi(x)$ is a testing function of \eqref{rehareq1}. Thus we get that
    \begin{equation*}
        \begin{split}
            \int_{\Omega} \langle\mathbf{a}(x,\nabla u),\nabla\phi(x)    \rangle
            =&\int_{\Omega} f(x)\phi(x)+\int_{\Gamma_2} g(x)\phi(x).
        \end{split}
    \end{equation*}
    From \eqref{mixeq1} and \eqref{mixeq3}, we get that
    \begin{equation*}
        \begin{split}
            &|\mathbf{a}(x,\nabla u) -\mathbf{a}(x,\nabla \overline{w})|1_{\{u(x)>h(x)\}}+|\mathbf{a}(x,\nabla u) +\mathbf{a}(x,\nabla \overline{w})|1_{\{u(x)<h(x)\}}\\\leq &C_{d,e}\bigg(|\mathbf{a}(x,\nabla h)|+\big(s+ | \nabla \overline{w}|\big)^{n-2}| \nabla h|+ | \nabla h|^{n-1} \bigg)
            \\\leq &C_{d,e}\bigg(s+s^{n-1}+ | \nabla \overline{w}|^{n-2}| \nabla h|+ | \nabla h|^{n-1}\bigg).
        \end{split}
    \end{equation*}
 Let $v=\overline{w}^{q}$ and $ q=\frac{n+\beta-1}{n}$. Through a similar computation as the proof of Theorem 2 in \cite{S}, we get that
\begin{equation*}
    \begin{split}
        &\int_{\Omega} |\eta F'(\overline{w})\nabla (\overline{w})|^{n}
        \\\leq &C_{d,e,n}\int_{\Omega} \bigg(|f\eta^{n}G(w)|+(s+|\nabla h|^{n-1}+|\nabla \overline{w}|^{n-1})|\nabla\eta|\eta^{n-1}G(w)
        \\&+\big(s+s^{n-1}+ | \nabla \overline{w}|^{n-2}| \nabla h|+ | \nabla h|^{n-1} \big)q^{-1}\beta |F'(\overline{w})^{n}||\nabla \overline{w}|\eta^n\bigg)
        \\&+C_{d,e,n}\int_{\Gamma_2} g(x)G(w)
        \\\leq &C_{d,e,n}\int_{\Omega} \bigg(|f\eta^{n}F(\overline{w})F'(\overline{w})^{n-1}|+(k^{n-1}+|\nabla \overline{w}|^{n-1})|\nabla\eta|\eta^{n-1}F(\overline{w})|F'(\overline{w})^{n-1}|
        \\&+\big(k^{n-1}+ C_{h}| \nabla \overline{w}|^{n-2} \big)q^{-1}\beta |F'(\overline{w})^{n}||\nabla \overline{w}|\eta^n\bigg)
        \\&+C_{d,e,n}\int_{\Gamma_2} g(x)F(\overline{w})|F'(\overline{w})^{n-1}|.
    \end{split}
\end{equation*}
Thus
\begin{equation}\label{rehareq19}
    \begin{split}
        &\int_{\Omega}|\eta\nabla v|^{n}
        \\\leq &C_{d,e,n}\bigg(\int_{\Omega} \left|\frac{q^{n-1}f}{k^{n-1} }\right||\eta v|^{n}+(1+\beta)\int_{\Omega}|v\nabla\eta ||\eta\nabla v|^{n-1}
        \\&+q^{n-1}\int_{\Omega}|v\nabla\eta ||\eta v|^{n-1} +q^{n-1}\beta\int_{\Omega}|\nabla v\eta ||\eta v|^{n-1}+\int_{\Gamma_2} \left|\frac{q^{n-1}g}{k^{n-1} }\right||\eta v|^{n}\bigg).
    \end{split}
\end{equation}
Set $\alpha_1 =\frac{n}{1+\frac{\epsilon}{2n}} $ and $\alpha_2 =\frac{n^2-\frac{\epsilon}{4}}{n} $, by H\"{o}lder's inequality and Sobolev inequality, we get that
\begin{equation}\label{rehareq20}
    \int_{\Omega}|v\nabla\eta ||\eta\nabla v|^{n-1} \leq \| v\nabla\eta \|_{L^{n}(\Omega)}\| \eta\nabla v  \|_{ L^{n}(\Omega)}^{n-1},
\end{equation}
\begin{equation}\label{rehareq21}
    \begin{split}
        \int_{\Omega} \left|\frac{f}{k^{n-1} }\right||\eta v|^{n} =&\int_{\Omega} \left|\frac{f}{k^{n-1} }\right||\eta v|^{\frac{\epsilon}{2}}|\eta v|^{n-\frac{\epsilon}{2}}\\
        \leq &k^{1-n}\|f\|_{L^{\frac{n}{n-\epsilon}}(\Omega)}\|\eta v  \|_{L^{\alpha_1} (\Omega  )  }^{\frac{\epsilon}{2}}\|\eta v  \|^{n-\frac{\epsilon}{2}}_{L^{\alpha_1^{*}} (\Omega  )}\\
        \leq &C_{\epsilon,n}\|\eta v  \|_{L^{n} (\Omega  )  }^{\frac{\epsilon}{2}}(   \|v\nabla \eta \|^{n-\frac{\epsilon}{2}}_{L^{n} (\Omega  )  }  + \|  \eta \nabla v\|^{n-\frac{\epsilon}{2}}_{L^{n} (\Omega  )  } ),
    \end{split}
\end{equation}
\begin{equation}\label{rehareq22}
    \begin{split}
       \int_{\Omega}|v\nabla\eta ||\eta v|^{n-1} \leq & C_{n}\|v\nabla\eta \|_{L^{n} (\Omega  )}\|\eta v  \|^{n-1}_{L^{\alpha_1^{*}} (\Omega  )} \\
       \leq &C_{n}\|v\nabla\eta \|_{L^{n} (\Omega  )}(\|\nabla\eta v  \|^{n-1}_{n} + \|\nabla v \eta   \|^{n-1}_{L^{n} (\Omega  )} ),
    \end{split}
\end{equation}
\begin{equation}\label{rehareq23}
    \begin{split}
       \int_{\Omega}|\nabla v\eta ||\eta v|^{n-1} \leq &\| v\eta \|^{n-1}_{L^{n}(\Omega)}\| \eta\nabla v  \|_{ L^{n}(\Omega)},
    \end{split}
\end{equation}
and
\begin{equation}\label{rerehareq23}
    \begin{split}
       \int_{\Gamma_2}\left|\frac{g}{k^{n-1} }\right||\eta v|^{n} =&\int_{\Gamma_2}\left|\frac{g}{k^{n-1} }\right||\eta v|^{\frac{\epsilon}{4}}|\eta v|^{n-\frac{\epsilon}{4}} \\
       \leq &\|g\|_{L^{\frac{n-1}{n-1-\frac{\epsilon}{4}}}(  \Gamma_2     )  }\|\eta v  \|_{L^{\alpha_2} (\Gamma_2  )  }^{\frac{\epsilon}{4}}\|\eta v  \|^{n-\frac{\epsilon}{4}}_{L^{\alpha_{2,*}} (\Gamma_2  )}\\
       \leq &C_{\epsilon,n}\|\eta v  \|_{L^{n} (\Gamma_2  )  }^{\frac{\epsilon}{4}}(   \|v\nabla \eta \|^{n-\frac{\epsilon}{4}}_{L^{n} (\Omega  )  }  + \|  \eta \nabla v\|^{n-\frac{\epsilon}{4}}_{L^{n} (\Omega  )  } ),
    \end{split}
\end{equation}
where  $\alpha_1^{*}=\frac{\alpha_1 n}{n-\alpha_1} $  and $\alpha_{2,*}=\frac{\alpha_2(n-1)}{n-\alpha_2} $.
From \eqref{rehareq19}, \eqref{rehareq20}, \eqref{rehareq21}, \eqref{rehareq22}, \eqref{rehareq23} and \eqref{rerehareq23}, we get that
\begin{equation*}
    \begin{split}
        &\int_{\Omega}|\eta\nabla v|^{n}
        \\\leq &C_{d,e,n,h,\epsilon}\bigg(  (1+\beta)\| v\nabla\eta \|_{L^{n}(\Omega)}\| \eta\nabla v  \|_{ L^{n}(\Omega)}^{n-1}+q^{n-1}\|\eta v  \|_{L^{n} (\Omega  )  }^{\frac{\epsilon}{2}}(   \|v\nabla \eta \|^{n-\frac{\epsilon}{2}}_{L^{n} (\Omega  )  }  + \|  \eta \nabla v\|^{n-\frac{\epsilon}{2}}_{L^{n} (\Omega  )  } )\\
        &+q^{n-1}\|v\nabla\eta \|_{L^{n} (\Omega  )}(\|\nabla\eta v  \|^{n-1}_{n} + \|\nabla v \eta   \|^{n-1}_{L^{n} (\Omega  )} ) +q^{n-1}\| v\eta \|^{n-1}_{L^{n}(\Omega)}\| \eta\nabla v  \|_{ L^{n}(\Omega)}
        \\&+q^{n-1}\|\eta v  \|_{L^{n} (\Gamma_2  )  }^{\frac{\epsilon}{4}}(   \|v\nabla \eta \|^{n-\frac{\epsilon}{4}}_{L^{n} (\Omega  )  }  + \|  \eta \nabla v\|^{n-\frac{\epsilon}{4}}_{L^{n} (\Omega  )  } )\bigg).
    \end{split}
\end{equation*}
By Young inequality, we get that
\begin{equation*}
    \begin{split}
        &\int_{\Omega}|\eta\nabla v|^{n}
        \\\leq &C_{d,e,n,h,\epsilon}\bigg(  (1+\beta +q)^{n^2} \| v\nabla\eta \|_{L^{n}(\Omega)}^{n} +\epsilon\big((q^{n})^{\frac{4n}{\epsilon}}+(q^{n})^{\frac{2n}{\epsilon}} \big)\left(\|\eta v  \|_{L^{n} (\Gamma_2  )  }^{n} +\|\eta v  \|_{L^{n} (\Omega )  }^{n}\right)
         \bigg).
    \end{split}
\end{equation*}
Thus
\begin{equation*}
    \begin{split}
      \left(\int_{\Omega}|\eta\nabla v|^{n} \right)^{\frac{1}{n}}\leq &C_{d,e,n,h,\epsilon}  (1+\beta +q)^{\frac{4n}{\epsilon}}(\|\eta v  \|_{L^{n} (\Omega  )  } +\| v\nabla\eta \|_{L^{n}(\Omega)}+\|\eta v  \|_{L^{n} (\Gamma_2  )  }) .
    \end{split}
\end{equation*}
 By Lemma \ref{relem:2.1} and Sobolev trace inequality, we get that
 \begin{equation*}
     \begin{split}
    \|\eta v  \|_{L^{\alpha_{2,*}} (\Gamma_2  )  } +\|\eta v  \|_{L^{\alpha_1^{*}} (\Omega  )  }    \leq  C_{d,e,n,h,\epsilon}  (1+\beta +q)^{\frac{4n}{\epsilon}}(\|\eta v  \|_{L^{n} (\Gamma_2  )  } +\|\eta v  \|_{L^{n} (\Omega  )  } +\| v\nabla\eta \|_{L^{n}(\Omega)}  ).
     \end{split}
 \end{equation*}
 Since $ \alpha_{2,*}\geq \alpha_1^{*}$, we get that
 \begin{equation}\label{rehareq24}
 \begin{split}
           &\|\eta v  \|_{L^{\alpha_1^{*}} (\Gamma_2  )  } +\|\eta v  \|_{L^{\alpha_1^{*}} (\Omega  )  }
           \\\leq & C_{d,e,n,h,\epsilon}  (1+\beta +q)^{\frac{4n}{\epsilon}}(\|\eta v  \|_{L^{n} (\Gamma_2  )  } +\|\eta v  \|_{L^{n} (\Omega  )  } +\| v\nabla\eta \|_{L^{n}(\Omega)}  ).
 \end{split}
 \end{equation}
Set $ m>m'>0$, and let $\eta$ be a smooth function such that $\eta \equiv 1$ in $B_{m'}(x_0)$, $\eta \equiv 0$ in $B_{m}^{c}(x_0)$ and $|\nabla\eta| \leq \frac{C}{m-m'}$. From \eqref{rehareq24}, we deduce that
\begin{equation*}
    \begin{split}
        &\| v  \|_{L^{\alpha_1^{*}} (\Gamma_2\cap B_{m'}(x_0)  )  } +\| v  \|_{L^{\alpha_1^{*}} (\Omega\cap B_{m'}(x_0)  )  }
        \\\leq&  C_{d,e,n,h,\epsilon}  (1+\beta +q)^{\frac{4n}{\epsilon}}(m-m')(\| v  \|_{L^{n} (\Gamma_2\cap B_{m}(x_0)  )  } +\| v  \|_{L^{n} (\Omega \cap B_{m}(x_0) )  } ).
    \end{split}
\end{equation*}
Thus by the standard Moser iteration argument, we get our desired result.

\vskip0.2in

\centerline{\bf Declarations}
\vskip0.2in
 {Data availability statement:}  All data needed are contained in the manuscript.

 \vskip0.2in
 { Funding and/or Conflicts of interests/Competing interests:} The authors declare that there are no financial, competing or conflict of interests.

\end{document}